\documentclass[english]{article}
\usepackage{mathptmx}
\usepackage[T1]{fontenc}
\usepackage[latin9]{inputenc}
\usepackage{geometry}
\usepackage{babel}
\usepackage{float}
\usepackage{dsfont}
\usepackage{amsmath}
\usepackage{amsthm}
\usepackage{amssymb}
\usepackage{graphicx}
\usepackage[unicode=true,pdfusetitle,
 bookmarks=true,bookmarksnumbered=false,bookmarksopen=false,
 breaklinks=false,pdfborder={0 0 0},pdfborderstyle={},backref=false,colorlinks=false]
 {hyperref}

\makeatletter
\numberwithin{equation}{section}
\numberwithin{figure}{section}
\theoremstyle{plain}
\newtheorem{thm}{\protect\theoremname}[section]
\theoremstyle{remark}
\newtheorem{rem}[thm]{\protect\remarkname}
\theoremstyle{definition}
\newtheorem{defn}[thm]{\protect\definitionname}
\theoremstyle{plain}
\newtheorem{prop}[thm]{\protect\propositionname}
\theoremstyle{plain}
\newtheorem{lem}[thm]{\protect\lemmaname}
\theoremstyle{plain}
\newtheorem{cor}[thm]{\protect\corollaryname}

\@ifundefined{date}{}{\date{}}
\makeatother

\providecommand{\corollaryname}{Corollary}
\providecommand{\definitionname}{Definition}
\providecommand{\lemmaname}{Lemma}
\providecommand{\propositionname}{Proposition}
\providecommand{\remarkname}{Remark}
\providecommand{\theoremname}{Theorem}

\begin{document}
\title{Structural limit of volatility target under stochastic volatility
and discrete corrections}
\author{Xuan Liu\thanks{Nomura Securities, Hong Kong SAR. Email: xuan.liu@nomura.com}\enskip{}and
Michel Gauthier\thanks{Nomura Securities, Tokyo, Japan. Email: michel.gauthier@nomura.com}}
\maketitle
\begin{abstract}
The exact limiting distribution of a volatility target index has recently
been established under a log-normal assumption on the risky asset.
In this paper, we show that the structural limit remains valid under
a broad class of stochastic volatility processes. In addition, we
show that, when the limit of the joint distribution of the volatility
target index and the driving Brownian motions is considered, an extra
independent Brownian motion must be involved. Besides the structural
limit theorem, we derive a correction formula for the expected quadratic
variation of the discrete rebalancing volatility target index, under
the assumption that the volatility of the underlying risky asset is
a deterministic function of time. As a consequence of the discrete
correction formula, it is shown that, as the rebalancing time step
$\Delta t\to0$ and the observation window parameter $\lambda\to1$
simultaneously along a path, the expected quadratic variation converges
to the target variance if and only if $(1-\lambda)^{-1}\Delta t\to0$
along the same path. Numerical results are provided to support the
structural limit, using the local volatility model and the Heston
stochastic volatility model as examples, as well as the effectiveness
of the discrete correction formula.
\end{abstract}

\section{Introduction\label{sec:-3}}

Volatility target mechanisms have become a cornerstone of modern portfolio
construction, underpinning risk-managed indices, volatility-controlled
funds, and insurance-linked investment products. They are routinely
utilized in fixed indexed annuities (FIA) and quantitative investment
strategies (QIS). Volatility target strategies dynamically adjust
their exposures to the risky assets in order to maintain a stable
level of volatility. As an investment strategy, volatility target
usually provides better risk-return balance; there have been numerous
empirical studies reporting improved Sharpe ratios generated by volatility
target strategies (see for example \cite{Haeusler88}). As the underlying
assets for financial derivatives, volatility target indices reduce
the degree of volatility uncertainty for pricing and risk management.
Despite its extensive use in the financial industry, modeling of the
dynamics of volatility target indices as an It\^{o} diffusion has
relied on heuristic approaches, since the limiting behavior for infinitesimally
small rebalancing time steps was not sufficiently understood. In a
recent work \cite{LG26} by the authors, the exact limiting distribution
of volatility target is established under the assumption that the
volatility of the underlying risky asset is deterministic, which shows
a consistent bias between the volatility of the limiting distribution
and the target level even when $\Delta t\to0$. The structural limit
obtain in \cite[Theorem 1.5]{LG26} serves as a base line for the
pricing of financial derivatives written on volatility target indices.
However, it has certain limitations from a practitioner's point of
view. Firstly, it has been verified by empirical data that the volatility
of risky assets is not deterministic, which corresponds to the so-called
volatility smile widely observed in financial markets. Secondly, even
when the log-normal assumption is a good approximation to the reality,
the structural limit of volatility target may deviate from its discrete
rebalacing implementation. The purpose of this paper is to generalize
the structural limit of volatility target to a broad class of stochastic
volatility dynamics, and to derive a correction formula which compensates
the deviation from the limiting distribution of the volatility target
to its discrete rebalancing implementation. 

\subsection{\label{subsec:-1}Summary of Main Results}

Let $d\ge0$ and $\{\mathcal{F}_{t}\}_{t\ge0}$ be the filtration
generated by a $(d+1)$-dimensional Brownian motion $\{(W_{t},\tilde{W}_{t})\}_{t\ge0}=\{(W_{t},\tilde{W}_{t}^{1},\dots,\tilde{W}_{t}^{d})\}_{t\ge0}$
on a probability space $(\Omega,\mathcal{F}_{\infty},\mathbb{P})$.
We assume that $\{S_{t}\}_{t\ge0}$ is a risky asset with dynamics
given by
\begin{equation}
dS_{t}=\rho(t)S_{t}dt+\sigma_{t}S_{t}dW_{t},\label{eq:-87}
\end{equation}
where 
\begin{itemize}
\item $\rho(t)\in L^{\infty}(0,T)$ is a deterministic function of $t$
and
\begin{equation}
|\rho(t)|\le M_{\rho},\quad t\ge0\label{eq:-184}
\end{equation}
for some constant $M_{\rho}$.
\item $\{\sigma_{t}\}_{t\ge0}$ is a progressively measurable process satisfying
the condition
\begin{equation}
0<\sigma_{\ast}\le\sigma_{t}\le\sigma^{\ast}<\infty,\quad t\ge0\quad\mathrm{a.s.}\label{eq:-10}
\end{equation}
for some constants $\sigma_{\ast}$ and $\sigma^{\ast}$.
\end{itemize}
\begin{rem}
\label{rem:-3}(i) Examples for which $d=0$ include the Black--Scholes
model and the local volatility model for which $\sigma_{t}=\sigma(t,S_{t})$,
while stochastic volatility models such as the Heston model correspond
to the case $d\ge1$.

(ii) The boundedness condition (\ref{eq:-10}), also called \emph{uniform
ellipticity} in the study of partial differential equations, is not
a very restrictive assumption in practice. Indeed, one may replace
(\ref{eq:-87}) by $dS_{t}=\rho(t)S_{t}dt+\bar{\sigma}_{t}S_{t}dW_{t}$,
where $\bar{\sigma}_{t}=\max(\sigma_{\ast},\min(\sigma^{\ast},\sigma_{t}))$
with $\sigma_{\ast}>0$ and $\sigma^{\ast}>0$ being small and large
constants such as $\sigma_{\ast}=0.001$ and $\sigma^{\ast}=100$. 
\end{rem}

We first recall the following definition of volatility target index.
\begin{defn}
\label{def:-3}Let $r(t)$ be a deterministic function representing
a\emph{ risk free rate} and
\begin{equation}
|r(t)|\le M_{r},\quad t\ge0,\label{eq:-190}
\end{equation}
for some constant $M_{r}$. Let $T>0$ be a fixed time, $N\in\mathbb{N}_{+}$
be the number of time steps, and $\Delta t=T/N$ be the \emph{rebalancing
time step}. Denote $t_{n}=n\Delta t$, $n=0,1,\dots,N$. The \emph{volatility
target index} $\{I_{t}\}_{t\ge0}$ (with initial value $1$) is the
process defined by 
\begin{equation}
\begin{aligned}\dfrac{dI_{t}}{I_{t}} & =(1-\bar{\sigma}v_{n-1}^{-1/2})r(t)dt+\bar{\sigma}v_{n-1}^{-1/2}\frac{dS_{t}}{S_{t}},\quad t\in[t_{n-1},t_{n}),\\
I_{0} & =1,
\end{aligned}
\label{eq:-86}
\end{equation}
where $\bar{\sigma}>0$ is a constant called the \emph{target volatility},
and $v_{n}$ is the \emph{realized variance }defined by
\begin{equation}
\begin{aligned}v_{n} & =\lambda v_{n-1}+\frac{1-\lambda}{\Delta t}\Big(\frac{S_{t_{n}}}{S_{t_{n-1}}}-1\Big)^{2},\quad n\ge1,\end{aligned}
\label{eq:-2}
\end{equation}
with $v_{0}>0$ being the \emph{initial realized variance} and $\lambda\in(0,1)$
being the \emph{observation window parameter}. 
\end{defn}

\begin{rem}
\label{rem:}We should point out that the definition (\ref{eq:-86})
is slightly different from the one used in the industry, which is
defined by the following recursion
\[
\begin{aligned}\frac{I_{t_{n}}}{I_{t_{n-1}}} & =1+(1-\bar{\sigma}v_{n-1}^{-1/2})\int_{t_{n-1}}^{t_{n}}r(t)dt+\bar{\sigma}v_{n-1}^{-1/2}\Big(\frac{S_{t_{n}}}{S_{t_{n-1}}}-1\Big),\quad n\ge1,\\
I_{t_{0}} & =1,
\end{aligned}
\]
and setting $I_{t}=I_{t_{n-1}}$ for $t\in[t_{n-1},t_{n})$. It can
be shown that these two definitions converge to the same limiting
distribution as $\Delta t\to0$ (see \cite[Section 1]{LG26}). Therefore,
we will adopt the definition (\ref{eq:-86}) which is easier to work
with.
\end{rem}

When $\sigma_{t}$ is a deterministic function satisfying the condition
(\ref{eq:-10}), it is shown in \cite[Theorem 1.5]{LG26} that the
process $\{I_{t}\}_{t\ge0}$ converges in distribution to an It\^{o}
diffusion as $\Delta t\to0$. The following theorem extends the conclusion
of \cite[Theorem 1.5]{LG26} to any stochastic process $\sigma_{t}$
satisfying (\ref{eq:-10}). When pricing a financial derivative on
a basket containing a volatility target index as a constituent, it
is necessary to consider the joint distribution of $\{I_{t}\}_{t\ge0}$
and the Brownian motions driving all other risky assets. Note that,
by Theorem \ref{thm:}, an extra independent Brownian motion $\{W_{t}^{\prime}\}_{t\ge0}$
must be involved when the limiting distribution of the joint process
$\{(I_{t},W_{t},\tilde{W}_{t})\}_{t\ge0}$ is considered, and \cite[Theorem 1.5]{LG26}
can be regarded as a degenerate case of Theorem \ref{thm:} when the
limiting distribution of $\{I_{t}\}_{t\ge0}$ itself is considered.
\begin{thm}
\label{thm:}Suppose that $\{\sigma_{t}\}_{t\ge0}$ is a progressively
measurable process satisfying the condition (\ref{eq:-10}). Let $\{W_{t}^{\prime}\}_{t\ge0}$
be a Brownian motion independent of $\{(W_{t},\tilde{W}_{t})\}_{t\ge0}$.
For any fixed $\lambda\in(0,1)$, let $\{X_{t}\}_{t\ge0}$ be the
continuous process given by the stochastic differential equation
\begin{equation}
\begin{aligned}dX_{t} & =\big(r(t)+(\rho(t)-r(t))\sigma_{t}^{-1}\bar{\sigma}U(\lambda)\big)X_{t}dt+\bar{\sigma}V(\lambda)^{1/2}X_{t}dB_{t},\\
X_{0} & =1,
\end{aligned}
\label{eq:}
\end{equation}
where $\{B_{t}\}_{t\ge0}$ is the Brownian motion 
\begin{equation}
B_{t}=\frac{U(\lambda)}{V(\lambda)^{1/2}}W_{t}+\sqrt{1-\frac{U(\lambda)^{2}}{V(\lambda)}}W_{t}^{\prime},\label{eq:-35}
\end{equation}
and $U(\lambda)$ and $V(\lambda)$ are the functions given by
\begin{equation}
U(\lambda)=\sqrt{\frac{2}{\pi(1-\lambda)}}\int_{0}^{\infty}\prod_{k=0}^{\infty}(1+t^{2}\lambda^{k})^{-1/2}dt,\label{eq:-16}
\end{equation}
and
\begin{equation}
V(\lambda)=\frac{1}{2(1-\lambda)}\int_{0}^{\infty}\prod_{k=0}^{\infty}(1+t\lambda^{k})^{-1/2}dt.\label{eq:-3}
\end{equation}
 Then the joint process $\{(I_{t},W_{t},\tilde{W}_{t})\}_{t\ge0}$
converges in distribution to $\{(X_{t},W_{t},\tilde{W}_{t})\}_{t\ge0}$
as $\Delta t\to0$.
\end{thm}

\begin{rem}
\label{rem:-4}(i) As shown by (\ref{eq:-35}), the Brownian motion
$\{B_{t}\}_{t\ge0}$ driving the limiting process $\{X_{t}\}_{t\ge0}$
involves an extra independent Brownian motion $\{W_{t}^{\prime}\}_{t\ge0}$.
For values of $\lambda$ commonly used in practice, the variance $1-U(\lambda)^{2}/V(\lambda)$
in front of $W_{t}^{\prime}$ in (\ref{eq:-35}) is small; for example,
$1-U(0.9)^{2}/V(0.9)\approx0.026$.

(ii) The equation (\ref{eq:}) should be interpreted as a dynamics
supported on the product probability space $(\Omega\times\Omega^{\prime},\mathcal{F}_{\infty}\times\mathcal{F}_{\infty}^{\prime},\mathbb{P}\times\mathbb{P}^{\prime})$,
with $(\Omega^{\prime},\mathcal{F}_{\infty}^{\prime},\mathbb{P}^{\prime})$
being the probability space where the Brownian motion $\{W_{t}^{\prime}\}_{t\ge0}$
is defined. The limiting process $\{X_{t}\}_{t\ge0}$ can be defined
on any probability space which supports a $(d+2)$-dimensional Brownian
motion. For convenience, we construct it directly on a product space
with the original probability space $(\Omega,\mathcal{F}_{\infty},\mathbb{P})$. 

(iii) For the definition (\ref{eq:-35}), note that $U(\lambda)^{2}<V(\lambda)$
is a consequence of H\"{o}lder's inequality and probabilistic representations
of $U(\lambda)$ and $V(\lambda)$. In fact, let $\xi_{n}=(1-\lambda)\sum_{k=1}^{n}\lambda^{k}Z_{n}^{2}$
with $\{Z_{n}\}_{n\ge1}$ being an i.i.d. sequence of normal distributions.
Then $U(\lambda)=\lim_{n\to\infty}\mathbb{E}(\xi_{n})$ and $V(\lambda)=\lim_{n\to\infty}\mathbb{E}(\xi_{n}^{2})$
(see \cite[Proposition 2.7]{LG26}). Alternatively, $U(\lambda)^{2}<V(\lambda)$
can also be seen by the bounds given in \cite[Theorem 1.5]{LG26}.
In Lemma \ref{lem:-6}, we give a direct proof of $U(\lambda)^{2}<V(\lambda)$
without using their probabilistic representation or the delicate bounds
in \cite[Theorem 1.5]{LG26}.

(iv) It is shown in \cite[Theorem 1.5]{LG26} that $\lim_{\lambda\to1-}U(\lambda)=\lim_{\lambda\to1-}V(\lambda)=1$.
However, this does not imply the convergence in distribution of $\{I_{t}\}_{t\ge0}$
to the diffusion 
\[
dX_{t}=\big(r(t)+(\rho(t)-r(t))\sigma_{t}^{-1}\bar{\sigma}\big)X_{t}dt+\bar{\sigma}X_{t}dB_{t}
\]
with the target volatility as $(\Delta t,\lambda)\to(0,1)$. In fact,
the expected quadratic variation $\mathbb{E}(\langle\log I\rangle_{T})$
does not converge as $(\Delta t,\lambda)\to(0,1)$ unless $\sigma_{t}=v_{0}^{-1/2}$
for all $t$ (see \cite[Section 4]{LG26}). In Theorem \ref{thm:-4},
we will give a sufficient and necessary condition for $\mathbb{E}(\langle\log I\rangle_{T})\to\bar{\sigma}^{2}T$
as $\Delta t\to0$ and $\lambda\to1$ simultaneously.
\end{rem}

As we will see in Proposition \ref{prop:-1}, the convergence in distribution
of $\{I_{t}\}_{t\ge0}$ by Theorem \ref{thm:} cannot be improved
to a.s. pathwise convergence or even pathwise convergence in probability.
More precisely, let $\mathbb{R}^{[0,T]}$ be the space of real-valued
functions and equipped with the pointwise convergence topology; that
is, a sequence $\{f_{n}\}_{n}\subseteq\mathbb{R}^{[0,T]}$ converges
to $f\in\mathbb{R}^{[0,T]}$ in the space $\mathbb{R}^{[0,T]}$ if
and only if $\lim_{n\to\infty}f_{n}(t)=f(t)$ for all $t\in[0,T]$.
The following proposition shows that, as $\mathbb{R}^{[0,T]}$-valued
random variables, $\{I_{t}^{(N)}\}_{t\ge0}$ does not converge a.s.
to a limit, where we denote $\{I_{t}\}_{t\ge0}$ by $\{I_{t}^{(N)}\}_{t\ge0}$
to indicate the dependence on $N$. 
\begin{prop}
\label{prop:-1}There exist no subsequence $\{I_{t}^{(n)}\}_{t\ge0}$
of $\{I_{t}^{(N)}\}_{t\ge0}$ and no process $\{\tilde{X}_{t}\}_{t\ge0}$
on the probability space $(\Omega,\mathcal{F},\mathbb{P})$ such that
$\lim_{n\to\infty}I_{t}^{(n)}=\tilde{X}_{t}$ a.s. for all $t\in[0,T]$.
Consequently, $\{I_{t}^{(N)}\}_{t\ge0}$ does not converge a.s. as
$\mathbb{R}^{[0,T]}$-valued random variables, and does not converge
in probability as $C([0,T])$-valued random variables.
\end{prop}

Let $C^{1,1}([0,T])$ be the space of differentiable functions $f$
on $[0,T]$ such that $f^{\prime}$ is Lipschitz continuous, that
is, 
\[
\Vert f^{\prime}\Vert_{C^{0,1}}=\sup_{0\le s<t\le T}\frac{|f^{\prime}(t)-f^{\prime}(s)|}{t-s}<\infty.
\]
Proposition \ref{prop:} below gives a correction formula which gives
the first order approximation of the expected quadratic variation
of $\{\log I_{t}\}_{t\ge0}$ with respect to $\Delta t$, or more
precisely $(1-\lambda)^{-1}\Delta t$. The correction formula (\ref{eq:-7})
below can be regarded as a moment matching result. Moreover, it gives
a characterization of the competing effect between $\Delta t$ and
$\lambda$ when $\Delta t\to0$ and $\lambda\to1$ simultaneously
(see \cite[Section 4]{LG26} for details).
\begin{prop}
\label{prop:}Suppose that $\sigma_{t}=\sigma(t)$ is a deterministic
$C^{1,1}([0,T])$ function satisfying (\ref{eq:-10}). Let 
\begin{equation}
\delta(\Delta t,\lambda)=(1-\lambda)^{-1}\Delta t,\label{eq:-307}
\end{equation}
and $\lambda_{\ast}\in(0,1)$ be a fixed constant. Then, for any $\lambda\in[\lambda_{\ast},1)$,
we have 
\begin{equation}
\mathbb{E}[\langle\log I\rangle_{T}]=V(\lambda)\bar{\sigma}^{2}T+\Big[Q_{V}(\lambda,v_{0}\sigma(0)^{-2})+2\Big(\frac{V(\lambda)}{1-\lambda}-\lambda V^{\prime}(\lambda)\Big)\log\frac{\sigma(T)}{\sigma(0)}\Big]\bar{\sigma}^{2}\Delta t+R(T;\Delta t,\lambda),\label{eq:-22}
\end{equation}
with $R(\tau;\Delta t,\lambda)=\mathrm{O}(\delta(\Delta t,\lambda)^{2})$
in the sense that $|R(T;\Delta t,\lambda)|\le c_{\lambda_{\ast},\sigma,v_{0},T}\delta(\Delta t,\lambda)^{2}$,
where $c_{\lambda_{\ast},\sigma,v_{0},T}$ is a constant depending
only on $\lambda_{\ast}$, $\sigma_{\ast}$, $\sigma^{\ast}$, $\Vert\sigma^{\prime}\Vert_{L^{\infty}}$,
$\Vert\sigma^{\prime}\Vert_{C^{0,1}}$, $v_{0}$ and $T$, and $Q_{V}(\lambda,\theta)$
is the absolutely convergent series 
\begin{equation}
Q_{V}(\lambda,\theta)=\sum_{k=0}^{\infty}[V_{k}(\lambda,\theta)-V(\lambda)],\quad\theta>0,\label{eq:-4}
\end{equation}
where
\begin{align}
V_{k}(\lambda,\theta) & =\frac{1}{2(1-\lambda)}\int_{0}^{\infty}e^{-t\theta(1-\lambda)^{-1}\lambda^{k}/2}\prod_{j=0}^{k-1}(1+t\lambda^{j})^{-1/2}dt,\label{eq:-5}
\end{align}
with $\prod_{j=0}^{k-1}(1+t\lambda^{j})^{-1/2}=1$ for $k=0$. Moreover,
let
\begin{equation}
\tau_{\ast}=\lambda_{\ast}^{-1}\delta_{1}(\Delta t,\lambda)\log\delta_{1}(\Delta,\lambda)^{-1}.\label{eq:-305}
\end{equation}
Then, for any $\lambda\in[\lambda_{\ast},1)$ and any $\tau\in[\tau_{\ast},T]$,
we have 
\begin{equation}
\mathbb{E}[\langle\log I\rangle_{\tau}]=V(\lambda)\bar{\sigma}^{2}\tau+\Big[Q_{V}(\lambda,v_{0}\sigma(0)^{-2})+2\Big(\frac{V(\lambda)}{1-\lambda}-\lambda V^{\prime}(\lambda)\Big)\log\frac{\sigma(\tau)}{\sigma(0)}\Big]\bar{\sigma}^{2}\Delta t+R(\tau;\Delta t,\lambda),\label{eq:-7}
\end{equation}
with 
\begin{equation}
|R(\tau;\Delta t,\lambda)|\le c_{\lambda_{\ast},\sigma,v_{0},T}\delta(\Delta t,\lambda)^{2},\label{eq:-308}
\end{equation}
where $c_{\lambda_{\ast},\sigma,v_{0},T}$ is again a constant depending
only on $\lambda_{\ast}$, $\sigma_{\ast}$, $\sigma^{\ast}$, $\Vert\sigma^{\prime}\Vert_{L^{\infty}}$,
$\Vert\sigma^{\prime}\Vert_{C^{0,1}}$, $v_{0}$ and $T$.
\end{prop}

As a direct consequence of Proposition \ref{prop:} and Corollary
\ref{cor:-4} and Proposition \ref{prop:-9} in Section \ref{subsec:-3},
we obtain the following sufficient and necessary condition for the
expected quadratic variation to converge to the target variance as
$\Delta t\to0$ and $\lambda\to1$ simultaneously along a path. 
\begin{thm}
\label{thm:-4}Suppose that $\sigma_{t}=\sigma(t)$ is a deterministic
$C^{1,1}([0,T])$ function satisfying (\ref{eq:-10}). Let $\lambda_{\ast}\in(0,1)$
be a fixed constant, and $\delta(\Delta t,\lambda)$ be defined by
(\ref{eq:-307}). Then
\[
|Q_{V}(\lambda,v_{0}\sigma(0)^{-2})|+|(1-\lambda)^{-1}V(\lambda)-\lambda V^{\prime}(\lambda)|\le c_{\lambda_{\ast},\sigma,v_{0}}(1-\lambda)^{-1},
\]
where $c_{\lambda_{\ast},\sigma,v_{0}}$ is a constant depending only
on $\lambda_{\ast}$, $\sigma_{\ast}$, $\sigma^{\ast}$, $\Vert\sigma^{\prime}\Vert_{L^{\infty}}$,
$\Vert\sigma^{\prime}\Vert_{C^{0,1}}$ and $v_{0}$. Therefore, 
\[
\mathbb{E}[\langle\log I\rangle_{T}]\to\bar{\sigma}^{2}T
\]
 as $(\Delta t,\lambda)\to(0,1)$ along a path $\lambda=\lambda(\Delta t)$
if and only if $\delta(\Delta t,\lambda)\to0$ along the same path.
\end{thm}

\begin{rem}
\label{rem:-1}(i) By Proposition \ref{prop:-9}, the convergence
of (\ref{eq:-4}) is exponentially fast with a remainder estimate
given by $\sum_{k=n}^{\infty}|V_{k}(\lambda,\theta)-V(\lambda)|\le c_{0}(\theta+\theta^{-1})\lambda^{n-4}(1-\lambda)^{-1}$,
where $c_{0}$ is a universal constant.

(ii) In practice, the volatility term structure $\sigma(t)$ is usually
constructed using (monotone) cubic splines or smooth parametrization
such as $\sigma(t)=\sigma_{1}(1-e^{-\kappa t})/(\kappa t)+\sigma_{2}$
with $\sigma_{1},\sigma_{2},\kappa>0$. Therefore the smoothness assumption
of Proposition \ref{prop:} is usually satisfied.

(iii) A condition such as $\tau\ge\tau_{\ast}$ is necessary for the
asymptotic correction formula (\ref{eq:-7}) to be valid. This is
because the asymptotic behavior is a phenomenon when $\tau$ is much
larger than $\Delta t$. When $\tau$ is comparable to $\Delta t$,
or more precisely $\tau<\tau_{\ast}$, it can be seen from the proof
of Proposition \ref{prop:} that (\ref{eq:-7}) should be replaced
by 
\begin{equation}
\mathbb{E}[\langle\log I\rangle_{\tau}]=V(\lambda)\bar{\sigma}^{2}\tau+\Big[\sum_{0\le k<\tau/\Delta t}[V_{k}(\lambda,v_{0}\sigma(0)^{-2})-V(\lambda)]+2\Big(\frac{V(\lambda)}{1-\lambda}-\lambda V^{\prime}(\lambda)\Big)\log\frac{\sigma(\tau)}{\sigma(0)}\Big]\bar{\sigma}^{2}\Delta t+\text{O}(\delta(\Delta t,\lambda)^{2}).\label{eq:-306}
\end{equation}

(iv) Differentiating (\ref{eq:-7}) with respect to $\tau$ gives
the following instantaneous discrete correction formula 
\begin{equation}
\frac{d\mathbb{E}[\langle\log I\rangle_{\tau}]}{d\tau}=V(\lambda)\bar{\sigma}^{2}+2\Big(\frac{V(\lambda)}{1-\lambda}-\lambda V^{\prime}(\lambda)\Big)\frac{\sigma^{\prime}(\tau)}{\sigma(\tau)}\bar{\sigma}^{2}\Delta t+\text{O}(\delta(\Delta t,\lambda)^{2}),\label{eq:-309}
\end{equation}
for $\tau\ge\tau_{\ast}$. For $\tau<\tau_{\ast}$, differentiating
(\ref{eq:-306}) with respect to $\tau$ gives
\begin{equation}
\frac{d\mathbb{E}[\langle\log I\rangle_{\tau}]}{d\tau}=\bar{\sigma}^{2}\sum_{k=0}^{\infty}V_{k}(\lambda,v_{0}\sigma(0)^{-2})\mathds{1}_{(k,k+1]}(\tau/\Delta t)+2\Big(\frac{V(\lambda)}{1-\lambda}-\lambda V^{\prime}(\lambda)\Big)\frac{\sigma^{\prime}(\tau)}{\sigma(\tau)}\bar{\sigma}^{2}\Delta t+\text{O}(\delta(\Delta t,\lambda)^{2}).\label{eq:-310}
\end{equation}
Note that the sum over $k$ in (\ref{eq:-310}) is in fact a finite
sum over $0\le k<\tau_{\ast}/\Delta t$. 

(v) Using an argument similar to the proof of Proposition \ref{prop:},
together with the estimates in Lemma \ref{lem:-21}, Lemma \ref{lem}
and Lemma \ref{lem:-23} with $\gamma=1/2$, we have the following
correction formula for expected drift 
\begin{equation}
\begin{aligned} & \mathbb{E}(\log I_{\tau})+2^{-1}\mathbb{E}(\langle\log I\rangle_{\tau})-\int_{0}^{\tau}r(t)dt\\
 & =\bar{\sigma}U(\lambda)\int_{0}^{\tau}\frac{q(t)}{\sigma(t)}dt+\Big[Q_{U}(\lambda,v_{0}\sigma(0)^{-2})\frac{q(0)}{\sigma(0)}+\Big(\frac{U(\lambda)}{1-\lambda}-2\lambda U^{\prime}(\lambda)\Big)\int_{0}^{\tau}\frac{q(t)\sigma^{\prime}(t)}{\sigma(t)^{2}}dt\Big]\bar{\sigma}\Delta t+\text{O}(\delta(\Delta t,\lambda)^{2}),
\end{aligned}
\label{eq:-21}
\end{equation}
where
\begin{equation}
Q_{U}(\lambda,\theta)=\sum_{k=0}^{\infty}[U_{k}(\lambda,\theta)-U(\lambda)],\label{eq:-311}
\end{equation}
and
\begin{equation}
U_{k}(\lambda,\theta)=\sqrt{\frac{2}{\pi(1-\lambda)}}\int_{0}^{\infty}e^{-t^{2}\theta(1-\lambda)^{-1}\lambda^{k}/2}\prod_{j=0}^{k-1}(1+t^{2}\lambda^{j})^{-1/2}dt.\label{eq:-312}
\end{equation}
\end{rem}

\subsection{\label{subsec:-2}Organization of the Paper}

The rest of this paper will be organized as follows. In Section \ref{sec:-1},
we establish several expectation estimates, including estimates for
negative moments of the realized variance $v_{n}$ defined by (\ref{eq:-2}),
which will be needed in the proof of the main results. The key estimates
in Section \ref{sec:-1} are Lemma \ref{lem:-1} and Lemma \ref{lem:-3};
once these are obtained, the negative moment estimates for $v_{n}$
follow from arguments similar to \cite[Section 2.2]{LG26} with small
modification to adapt for the case where independence is not available. 

Section \ref{sec:-5} is devoted to the proof of Theorem \ref{thm:}.
The strategy of the proof is to pass from the structural limit for
constant $\{\sigma_{t}\}_{t\ge0}$ to that for general progressively
measurable $\{\sigma_{t}\}_{t\ge0}$ by a conditioning argument, together
with an estimate on the convergence rate independent of realizations
of $\{\sigma_{t}\}_{t\ge0}$. In Section \ref{subsec:-9}, we introduce
the concept of realized variance reset, which will be useful in the
conditioning argument in Section \ref{subsec:-8}. In Section \ref{subsec:-10},
we establish a continuity in coefficients result for $\{\log I_{t}\}_{t\ge0}$
(Proposition \ref{prop:-6}), which, together with Lemma \ref{lem:-7},
reduces the case for general progressively measurable $\{\sigma_{t}\}_{t\ge0}$
to that for Wiener functionals with explicit simple forms. In Section
\ref{subsec:-5}, we derive a Barry--Esseen type convergence rate
(Proposition \ref{prop:-4}) for marginal cumulative distribution
functions of $(\log I_{T},W_{T},\tilde{W}_{T})$ when $r(t)$, $\rho(t)$
and $\sigma_{t}$ are constants, as an application of \cite[Theorem 1]{Haeusler88}
for one-dimensional martingale difference sequences and a small generalization
of \cite[Theorem 2.4]{LG26}. In Section \ref{subsec:-6}, using properties
of Fourier transform and a truncation argument, we establish a result
for convergence rate of expectations (Proposition \ref{prop:-5})
for random vectors satisfying the marginal Barry--Esseen type bound
(\ref{eq:-207}) in Proposition \ref{prop:-4}. Compared to \cite[Theorem 1.5]{LG26}
for deterministic $\{\sigma_{t}\}_{t\ge0}$, the explicit convergence
rate of expectations for $(\log I_{T},W_{T},\tilde{W}_{T})$, as an
application of Proposition \ref{prop:-5}, will be crucial to the
conditioning argument in Section \ref{subsec:-8} when $\{\sigma_{t}\}_{t\ge0}$
is stochastic. In Section \ref{subsec:-8}, we give the proof of Theorem
\ref{thm:} by combining the ingredients established in the previous
subsections. We end Section \ref{sec:-5} by the proof of Proposition
\ref{prop:-1} on the non-pathwise convergence of $\{I_{t}\}_{t\ge0}$. 

Section \ref{sec:-2} is devoted to the derivation of Proposition
\ref{prop:} for the discrete correction formula. In Section \ref{subsec:-3},
we establish estimates related to the functions $V(\lambda)$ and
$V_{k}(\lambda,\theta)$ defined by (\ref{eq:-3}) and (\ref{eq:-5})
respectively. These estimates will be crucial to the order estimate
(\ref{eq:-7}) in terms of $\delta(\Delta t,\lambda)$ defined by
(\ref{eq:-307}). In Section \ref{subsec:}, we give the derivation
of the discrete correction formula (\ref{eq:-7}) under the assumption
that $\sigma_{t}=\sigma(t)$ is a $C^{1,1}([0,T])$ function. 

In Section \ref{sec:}, we present numerical results which support
the main results of this paper. More specifically, numerical evidence
for Theorem \ref{thm:} on the structural limit using the local volatility
model and the Heston stochastic volatility model and numerical results
for the approximation quality of the discrete correction formula are
given. 

\section{\label{sec:-1}Technical Estimates}

In this section, we establish some technical estimates which will
be needed in the proof of the main results. Throughout this section,
without otherwise stated, we will use the notation $c_{q_{1},\dots,q_{m}}$
to denote positive constants depending only on the quantities $q_{1},\dots,q_{m}$,
and the constant $c_{q_{1},\dots,q_{m}}$ may vary at different appearances. 
\begin{lem}
\label{lem:-1}Suppose that $\{\sigma_{t}\}_{t\ge0}$ is a progressively
measurable process satisfying the condition (\ref{eq:-10}). Then,
for any $0\le\tau_{1}<\tau_{2}<\infty$, $u>0$ and $h\in\mathbb{R}$,
\begin{equation}
\mathbb{E}\Big[\exp\Big(-u\Big(h+\int_{\tau_{1}}^{\tau_{2}}\sigma_{t}dW_{t}\Big)^{2}\Big)\Big|\mathcal{F}_{\tau_{1}}\Big]\le[1+2u(\sigma^{\ast})^{2}(\tau_{2}-\tau_{1})]^{-(\sigma_{\ast}/\sigma^{\ast})^{2}/2}\exp\Big(-\frac{uh^{2}}{1+2u(\sigma^{\ast})^{2}(\tau_{2}-\tau_{1})}\Big).\label{eq:-1}
\end{equation}
\end{lem}

\begin{proof}
We first prove (\ref{eq:-1}) for continuous processes $\{\sigma_{t}\}_{t\ge0}$.
Let $n$ be a positive integer, $\Delta\tau=(\tau_{2}-\tau_{1})/n$,
and $s_{k}=\tau_{1}+k\Delta\tau$, $0\le k\le n$. Let $M_{0}=h$,
$M_{n}=h+\sum_{k=1}^{n}\sigma_{s_{k-1}}(W_{s_{k}}-W_{s_{k-1}})$.
By (\ref{eq:-10}) and the continuity of $\sigma_{t}$, we have
\begin{equation}
\lim_{n\to\infty}\Big\Vert M_{n}-h-\int_{\tau_{1}}^{\tau_{2}}\sigma_{t}dW_{t}\Big\Vert_{L^{2}}=0.\label{eq:-36}
\end{equation}
By (\ref{eq:-36}), there exists a subsequence $\{M_{n_{k}}\}_{k\ge0}$
of $\{M_{n}\}_{n\ge0}$ such that $M_{n_{k}}\to h+\int_{\tau_{1}}^{\tau_{2}}\sigma_{t}dW_{t}$
a.s. By the dominated convergence theorem, 
\[
\mathbb{E}\Big[\exp\Big(-u\Big(h+\int_{\tau_{1}}^{\tau_{2}}\sigma_{t}dW_{t}\Big)^{2}\Big)\Big|\mathcal{F}_{\tau_{1}}\Big]=\lim_{k\to\infty}\mathbb{E}[\exp(-uM_{n_{k}}^{2})|\mathcal{F}_{\tau_{1}}].
\]
Therefore, it suffices to show that
\begin{equation}
\limsup_{n\to\infty}\mathbb{E}[\exp(-uM_{n}^{2})|\mathcal{F}_{\tau_{1}}]\le[1+2u(\sigma^{\ast})^{2}(\tau_{2}-\tau_{1})]^{-(\sigma_{\ast}/\sigma^{\ast})^{2}/2}\exp\Big(-\frac{uh^{2}}{1+2u(\sigma^{\ast})^{2}(\tau_{2}-\tau_{1})}\Big).\label{eq:-47}
\end{equation}

Note that, 
\begin{equation}
\mathbb{E}\big(\exp[-u(a+bZ)^{2}]\big)=(1+2ub^{2})^{-1/2}\exp\Big(-\frac{ua^{2}}{1+2ub^{2}}\Big),\label{eq:-54}
\end{equation}
for any $a,b\in\mathbb{R}$, where $Z$ is a standard normal distribution.
By (\ref{eq:-10}) and (\ref{eq:-54}), 
\[
\begin{aligned}\mathbb{E}[\exp(-uM_{n}^{2})|\mathcal{F}_{\tau_{1}}] & =\mathbb{E}\big[\mathbb{E}\big[\exp\big[-u\big(M_{n-1}+\sigma_{s_{n-1}}(W_{s_{n}}-W_{s_{n-1}})\big)^{2}\big]\big|\mathcal{F}_{s_{n-1}}\big]\big|\mathcal{F}_{\tau_{1}}\big]\\
 & =\mathbb{E}\Big[(1+2u\sigma_{s_{n-1}}^{2}\Delta\tau)^{-1/2}\exp\Big(-\frac{uM_{n-1}^{2}}{1+2u\sigma_{s_{n-1}}^{2}\Delta\tau}\Big)\Big|\mathcal{F}_{\tau_{1}}\Big]\\
 & \le(1+2u\sigma_{\ast}^{2}\Delta\tau)^{-1/2}\mathbb{E}\Big[\exp\Big(-\frac{uM_{n-1}^{2}}{1+2u(\sigma^{\ast})^{2}\Delta\tau}\Big)\Big|\mathcal{F}_{\tau_{1}}\Big].
\end{aligned}
\]
Let $T$ be the M\"{o}bius transformation
\begin{equation}
T(u)=\frac{u}{1+2u(\sigma^{\ast})^{2}\Delta\tau}.\label{eq:-55}
\end{equation}
Then
\begin{equation}
\mathbb{E}[\exp(-uM_{n}^{2})|\mathcal{F}_{\tau_{1}}]\le(1+2u\sigma_{\ast}^{2}\Delta\tau)^{-1/2}\mathbb{E}[\exp(-T(u)M_{n-1}^{2})|\mathcal{F}_{\tau_{1}}].\label{eq:-69}
\end{equation}
By (\ref{eq:-69}) and a simple induction,
\[
\begin{aligned}\mathbb{E}[\exp(-uM_{n}^{2})|\mathcal{F}_{\tau_{1}}] & \le(1+2u\sigma_{\ast}^{2}\Delta\tau)^{-1/2}\mathbb{E}[\exp(-T(u)M_{n-1}^{2})|\mathcal{F}_{\tau_{1}}]\\
 & \le(1+2u\sigma_{\ast}^{2}\Delta\tau)^{-1/2}(1+2T(u)\sigma_{\ast}^{2}\Delta\tau)^{-1/2}\mathbb{E}[\exp(-T^{2}(u)M_{n-2}^{2})|\mathcal{F}_{\tau_{1}}]\\
 & \le\cdots\\
 & \le\Big[\prod_{k=0}^{n-1}(1+2T^{k}(u)\sigma_{\ast}^{2}\Delta\tau)^{-1/2}\Big]\exp(-T^{n}(u)h^{2}),
\end{aligned}
\]
where $T^{0}$ is the identity function. It is easily seen by induction
that
\[
T^{k}(u)=\frac{u}{1+2ku(\sigma^{\ast})^{2}\Delta\tau},\quad k\ge0.
\]
Therefore, 
\begin{equation}
\mathbb{E}[\exp(-uM_{n}^{2})|\mathcal{F}_{\tau_{1}}]\le\Big[\prod_{k=0}^{n-1}\Big(1+\frac{2u\sigma_{\ast}^{2}\Delta\tau}{1+2uk(\sigma^{\ast})^{2}\Delta\tau}\Big)^{-1/2}\Big]\exp\Big(-\frac{uh^{2}}{1+2u(\sigma^{\ast})^{2}(\tau_{2}-\tau_{1})}\Big).\label{eq:-70}
\end{equation}

We now show that 
\begin{equation}
\lim_{n\to\infty}\prod_{k=0}^{n-1}\Big(1+\frac{2u\sigma_{\ast}^{2}\Delta\tau}{1+2uk(\sigma^{\ast})^{2}\Delta\tau}\Big)=[1+2u(\sigma^{\ast})^{2}(\tau_{2}-\tau_{1})]^{(\sigma_{\ast}/\sigma^{\ast})^{2}},\label{eq:-71}
\end{equation}
which, together with (\ref{eq:-70}), implies (\ref{eq:-47}). Let
$a=2u(\sigma^{\ast})^{2}$, $\theta=(\sigma_{\ast}/\sigma^{\ast})^{2}$,
and $f(x)=\log(1+ax)$. Then $f^{\prime\prime}(x)=-a^{2}(1+ax)^{-2}$.
Clearly, $|f^{\prime\prime}(x)|\le a^{2}$ for all $x\ge0$. Therefore,
by Taylor's theorem, 
\[
\begin{aligned}\sum_{k=0}^{n-1}\log\Big(1+\frac{2u\sigma_{\ast}^{2}\Delta\tau}{1+2uk(\sigma^{\ast})^{2}\Delta\tau}\Big) & =\sum_{k=0}^{n-1}\log\Big(1+\frac{a\theta\Delta\tau}{1+ak\Delta\tau}\Big)\\
 & =\sum_{k=0}^{n-1}\big[f\big((k+\theta)\Delta\tau\big)-f(k\Delta\tau)\big]=\theta\sum_{k=0}^{n-1}f^{\prime}(k\Delta\tau)\Delta\tau+\text{O}(\Delta\tau),
\end{aligned}
\]
which implies that
\[
\lim_{n\to\infty}\sum_{k=0}^{n-1}\log\Big(1+\frac{2u\sigma_{\ast}^{2}\Delta\tau}{1+2uk(\sigma^{\ast})^{2}\Delta\tau}\Big)=\theta\int_{0}^{\tau_{2}-\tau_{1}}f^{\prime}(x)dx=\theta\log[1+a(\tau_{2}-\tau_{1})].
\]
This proves (\ref{eq:-71}) and therefore completes the proof of (\ref{eq:-1})
for continuous $\{\sigma_{t}\}_{t\ge0}$.

For general progressively measurable processes $\{\sigma_{t}\}_{t\ge0}$
satisfying (\ref{eq:-10}). Let 
\begin{equation}
\sigma_{t}^{n}=\frac{1}{n}\int_{t-n^{-1}}^{t}\sigma_{s}ds,\label{eq:-136}
\end{equation}
where we adopt the convention that $\sigma_{t}=\sigma_{0}$ for $t<0$.
Clearly, $\{\sigma_{t}^{n}\}_{t\ge0}$ is a sequence of continuous
adapted processes satisfying (\ref{eq:-10}) with the same bounds
$\sigma_{\ast}$ and $\sigma^{\ast}$. By the Lebesgue differentiation
theorem (see \cite[Corollary 1, p. 5]{Stein70}) and the dominated
convergence theorem, $\mathbb{E}\big[\int_{0}^{T}(\sigma_{t}-\sigma_{t}^{n})^{2}dt\big]\to0$
for any $T\in(0,\infty)$. Consequently, $\int_{\tau_{1}}^{\tau_{2}}\sigma_{t}^{n}dW_{t}\to\int_{\tau_{1}}^{\tau_{2}}\sigma_{t}dW_{t}$
in $L^{2}$. By considering a subsequence of $\{\sigma^{n}\}_{n\ge1}$
if necessary, we may assume that $\int_{\tau_{1}}^{\tau_{2}}\sigma_{t}^{n}dW_{t}\to\int_{\tau_{1}}^{\tau_{2}}\sigma_{t}dW_{t}$
a.s. The inequality (\ref{eq:-1}) for $\{\sigma_{t}\}_{t\ge0}$ follows
readily from the same inequality for $\{\sigma_{t}^{n}\}_{t\ge0}$
and Fatou's lemma.
\end{proof}
\begin{rem}
\label{rem:-5}(i) When $\sigma_{\ast}=\sigma^{\ast}$, the process
$\sigma_{t}$ is a constant and $h+\int_{\tau_{1}}^{\tau_{2}}\sigma_{t}dW_{t}$
is normal distribution with mean $h$ and variance $(\sigma^{\ast})^{2}(\tau_{2}-\tau_{1})$.
It is easily seen that the equality in (\ref{eq:-1}) holds in this
case.

(ii) When $\{\sigma_{t}\}_{t\ge0}$ is the coefficient of a Markov
diffusion as in the local volatility model, an estimate similar to
(\ref{eq:-1}) can be given using the heat kernel estimate of linear
parabolic equations. More specifically, suppose that $\{\sigma_{t}\}_{t\ge0}$
is given by $\sigma_{t}=\sigma(t,Y_{t})$ with $\{Y_{t}\}_{t\ge0}$
being the solution of the stochastic differential equation
\[
dY_{t}=\sigma(t,Y_{t})dW_{t}.
\]
The transition probability density $p(t,x;\tau,y)=p(Y_{t}=x|Y_{\tau}=y)$
of $\{Y_{t}\}_{t\ge0}$ is given by the fundamental solution of the
Fokker--Planck equation
\[
\partial_{t}u(t,x)-\frac{1}{2}\sigma(t,x)^{2}\partial_{x}^{2}u(t,x)=0.
\]
By the heat kernel estimate for linear parabolic equations (see \cite[Theorem 7 and Theorem 8]{Aronson68}),
there exist positive constants $c_{1}$, $c_{2}$, $c_{3}$ and $c_{4}$
depending only on $\sigma_{\ast}$, $\sigma^{\ast}$ such that
\begin{equation}
\frac{c_{1}}{(t-\tau)^{1/2}}\exp\Big(-\frac{c_{2}|x-y|^{2}}{t-\tau}\Big)\le p(t,x;\tau,y)\le\frac{c_{3}}{(t-\tau)^{1/2}}\exp\Big(-\frac{c_{4}|x-y|^{2}}{t-\tau}\Big),\label{eq:-91}
\end{equation}
for all $x,y\in\mathbb{R}$ and $t>\tau$. It follows easily from
(\ref{eq:-91}) that
\begin{equation}
\mathbb{E}\big(\exp[-u(Y_{\tau_{2}}-Y_{\tau_{1}})^{2}]\big|Y_{\tau_{1}}=y\big)=\int_{-\infty}^{\infty}e^{-u(x-y)^{2}}p(\tau_{2},x;\tau_{1},y)dx\le c_{3}c_{4}^{-1/2}[1+c_{4}^{-1}u(\tau_{2}-\tau_{1})]^{-1/2}.\label{eq:-92}
\end{equation}

(iii) Let us give some comparison for the two estimates (\ref{eq:-1})
and (\ref{eq:-92}). The estimate (\ref{eq:-1}) applies to general
progressively measurable processes $\{\sigma_{t}\}_{t\ge0}$; while
(\ref{eq:-92}) requires that the coefficient $\sigma_{t}$ to be
Markov, and when this is satisfied, it gives the exact decay rate
$\text{O}(u^{-1/2})$. On the other hand, the heat kernel estimate
(\ref{eq:-91}) provides little information about the constant $c_{3}c_{4}^{-1/2}$
in (\ref{eq:-92}), which will be useful when considering products
of expectations of the form $\prod_{k=1}^{n}\mathbb{E}\big(\exp[-u(Y_{\tau_{k}}-Y_{\tau_{k-1}})^{2}]\big)$
for large $n$, where $\tau_{k}=kT/n$. More specifically, it follows
from (\ref{eq:-1}) that 
\[
\lim_{n\to\infty}\prod_{k=1}^{n}\mathbb{E}\big(\exp[-u(Y_{\tau_{k}}-Y_{\tau_{k-1}})^{2}]\big)\le\lim_{n\to\infty}\Big(1+\frac{2u(\sigma^{\ast})^{2}T}{n}\Big)^{-n(\sigma_{\ast}/\sigma^{\ast})^{2}/2}=e^{-u\sigma_{\ast}^{2}T}.
\]
However, the estimate (\ref{eq:-92}) does not yield any valid bound
for the above limit due to the presence of the factor $(c_{3}c_{4}^{-1/2})^{n}$.
\end{rem}

\begin{lem}
\label{lem:-3}Let $\{h_{t}\}_{t\ge0}$ and $\{\sigma_{t}\}_{t\ge0}$
be progressively measurable processes such that $\{\sigma_{t}\}_{t\ge0}$
satisfies the condition (\ref{eq:-10}), and that 
\begin{equation}
|h_{t}|\le M,\quad t\ge0\quad\mathrm{a.s.}\label{eq:-85}
\end{equation}
for some constant $M>0$. Then, for any $0\le\tau_{1}<\tau_{2}<\infty$,
$u>0$ and $p>1$,
\begin{equation}
\mathbb{E}\Big[\exp\Big(-u\Big(\int_{\tau_{1}}^{\tau_{2}}h_{t}dt+\int_{\tau_{1}}^{\tau_{2}}\sigma_{t}dW_{t}\Big)^{2}\Big)\Big|\mathcal{F}_{\tau_{1}}\Big]\le e^{(p-1)M^{2}\sigma_{\ast}^{-2}(\tau_{2}-\tau_{1})/2}[1+2p^{\prime}u(\sigma^{\ast})^{2}(\tau_{2}-\tau_{1})]^{-(\sigma_{\ast}/\sigma^{\ast})^{2}/(2p^{\prime})},\label{eq:-77}
\end{equation}
where $p^{\prime}=p/(p-1)$ is the conjugate exponent of $p$.
\end{lem}

\begin{proof}
We prove (\ref{eq:-77}) by change of measure. Let 
\begin{equation}
\mathcal{E}_{t}=\exp\Big(-\int_{0}^{t}\frac{h_{s}}{\sigma_{s}}dW_{s}-\frac{1}{2}\int_{0}^{t}\frac{h_{s}^{2}}{\sigma_{s}^{2}}ds\Big).\label{eq:-78}
\end{equation}
By the martingale convergence theorem, there exists an equivalent
probability measure $\mathbb{Q}$ on $(\Omega,\mathcal{F}_{\infty})$
such that 
\begin{equation}
\mathbb{E}^{\mathbb{P}}\Big(\frac{d\mathbb{Q}}{d\mathbb{P}}\Big|\mathcal{F}_{t}\Big)=\mathcal{E}_{t}.\label{eq:-79}
\end{equation}
Let $Y=\exp\big(-u\big(\int_{\tau_{1}}^{\tau_{2}}h_{t}dt+\int_{\tau_{1}}^{\tau_{2}}\sigma_{t}dW_{t}\big)^{2}\big)$.
Then
\begin{equation}
\begin{aligned}\mathbb{E}^{\mathbb{P}}(Y|\mathcal{F}_{\tau_{1}}) & =\mathbb{E}^{\mathbb{P}}\Big[Y\mathcal{E}_{\tau_{2}}^{-1}\mathbb{E}^{\mathbb{P}}\Big(\frac{d\mathbb{Q}}{d\mathbb{P}}\Big|\mathcal{F}_{\tau_{2}}\Big)\Big|\mathcal{F}_{\tau_{1}}\Big]\\
 & =\mathbb{E}^{\mathbb{P}}\Big(Y\mathcal{E}_{\tau_{2}}^{-1}\frac{d\mathbb{Q}}{d\mathbb{P}}\Big|\mathcal{F}_{\tau_{1}}\Big)\\
 & =\mathbb{E}^{\mathbb{Q}}\Big(Y\mathcal{E}_{\tau_{2}}^{-1}\Big|\mathcal{F}_{\tau_{1}}\Big)\mathbb{E}^{\mathbb{P}}\Big(\frac{d\mathbb{Q}}{d\mathbb{P}}\Big|\mathcal{F}_{\tau_{1}}\Big)\\
 & =\mathbb{E}^{\mathbb{Q}}\Big(Y\mathcal{E}_{\tau_{1}}\mathcal{E}_{\tau_{2}}^{-1}\Big|\mathcal{F}_{\tau_{1}}\Big),
\end{aligned}
\label{eq:-81}
\end{equation}
where the second equality is due to $Y\mathcal{E}_{\tau_{2}}^{-1}\in\mathcal{F}_{\tau_{2}}$,
and the third equality follows from Bayes' theorem. Let
\begin{equation}
B_{t}=W_{t}+\int_{0}^{t}\frac{h_{s}}{\sigma_{s}}ds.\label{eq:-80}
\end{equation}
Then
\[
Y\mathcal{E}_{\tau_{1}}\mathcal{E}_{\tau_{2}}^{-1}=\exp\Big(-u\Big(\int_{\tau_{1}}^{\tau_{2}}\sigma_{t}dB_{t}\Big)^{2}\Big)\exp\Big(\int_{\tau_{1}}^{\tau_{2}}\frac{h_{s}}{\sigma_{s}}dB_{s}-\frac{1}{2}\int_{\tau_{1}}^{\tau_{2}}\frac{h_{s}^{2}}{\sigma_{s}^{2}}ds\Big).
\]
Therefore, by (\ref{eq:-81}),
\begin{equation}
\begin{aligned}\mathbb{E}^{\mathbb{P}}(Y|\mathcal{F}_{\tau_{1}}) & \le\mathbb{E}^{\mathbb{Q}}\Big[\exp\Big(-p^{\prime}u\Big(\int_{\tau_{1}}^{\tau_{2}}\sigma_{t}dB_{t}\Big)^{2}\Big)\Big|\mathcal{F}_{\tau_{1}}\Big]^{1/p^{\prime}}\mathbb{E}^{\mathbb{Q}}\Big[\exp\Big(p\int_{\tau_{1}}^{\tau_{2}}\frac{h_{s}}{\sigma_{s}}dB_{s}-\frac{p}{2}\int_{\tau_{1}}^{\tau_{2}}\frac{h_{s}^{2}}{\sigma_{s}^{2}}ds\Big)\Big|\mathcal{F}_{\tau_{1}}\Big]^{1/p}.\end{aligned}
\label{eq:-84}
\end{equation}
By the Girsanov theorem, $\{B_{t}\}_{t\ge0}$ is a Brownian motion
under the probability measure $\mathbb{Q}$. It follows from Lemma
\ref{lem:-1} that
\begin{equation}
\mathbb{E}^{\mathbb{Q}}\Big[\exp\Big(-p^{\prime}u\Big(\int_{\tau_{1}}^{\tau_{2}}\sigma_{t}dB_{t}\Big)^{2}\Big)\Big|\mathcal{F}_{\tau_{1}}\Big]\le[1+2p^{\prime}u(\sigma^{\ast})^{2}(\tau_{2}-\tau_{1})]^{-(\sigma_{\ast}/\sigma^{\ast})^{2}/2}.\label{eq:-82}
\end{equation}
Note that $\exp\big(p\int_{\tau_{1}}^{\tau_{2}}h_{s}\sigma_{s}^{-1}dB_{s}-2^{-1}p^{2}\int_{\tau_{1}}^{\tau_{2}}h_{s}^{2}\sigma_{s}^{-2}ds\big)$
is a martingale under the probability measure $\mathbb{Q}$. By $|h_{t}\sigma_{t}^{-1}|\le M\sigma_{\ast}^{-1}$,
we deduce that
\begin{equation}
\begin{aligned} & \mathbb{E}^{\mathbb{Q}}\Big[\exp\Big(p\int_{\tau_{1}}^{\tau_{2}}\frac{h_{s}}{\sigma_{s}}dB_{s}-\frac{p}{2}\int_{\tau_{1}}^{\tau_{2}}\frac{h_{s}^{2}}{\sigma_{s}^{2}}ds\Big)\Big|\mathcal{F}_{\tau_{1}}\Big]\\
 & =\mathbb{E}^{\mathbb{Q}}\Big[\exp\Big(p\int_{\tau_{1}}^{\tau_{2}}\frac{h_{s}}{\sigma_{s}}dB_{s}-\frac{p^{2}}{2}\int_{\tau_{1}}^{\tau_{2}}\frac{h_{s}^{2}}{\sigma_{s}^{2}}ds\Big)\exp\Big(\frac{p(p-1)}{2}\int_{\tau_{1}}^{\tau_{2}}\frac{h_{s}^{2}}{\sigma_{s}^{2}}ds\Big)\Big|\mathcal{F}_{\tau_{1}}\Big]\\
 & \le e^{p(p-1)M^{2}\sigma_{\ast}^{-2}(\tau_{2}-\tau_{1})/2}\mathbb{E}^{\mathbb{Q}}\Big[\exp\Big(p\int_{\tau_{1}}^{\tau_{2}}\frac{h_{s}}{\sigma_{s}}dB_{s}-\frac{p^{2}}{2}\int_{\tau_{1}}^{\tau_{2}}\frac{h_{s}^{2}}{\sigma_{s}^{2}}ds\Big)\Big|\mathcal{F}_{\tau_{1}}\Big]\\
 & =e^{p(p-1)M^{2}\sigma_{\ast}^{-2}(\tau_{2}-\tau_{1})/2}.
\end{aligned}
\label{eq:-83}
\end{equation}
Combining (\ref{eq:-84}), (\ref{eq:-82}) and (\ref{eq:-83}) completes
the proof.
\end{proof}
We now establish several estimates for the realized variance $v_{n}$
defined in (\ref{eq:-2}), which will be needed in the proof of the
main results. Let 
\begin{align}
\zeta_{k} & =(\Delta t)^{-1/2}\int_{t_{k-1}}^{t_{k}}[\rho(t)-\sigma_{t}^{2}/2]dt,\label{eq:-148}\\
\xi_{k} & =(\Delta t)^{-1/2}\int_{t_{k-1}}^{t_{k}}\sigma_{t}dW_{t},\label{eq:-149}\\
\eta_{k} & =\Delta t^{-1/2}[e^{\Delta t^{1/2}(\zeta_{k}+\xi_{k})}-1],\label{eq:-150}
\end{align}
where $t_{k}=k\Delta t$, $0\le k\le N$. Clearly, $v_{n}$ can be
written as
\begin{equation}
v_{n}=\lambda^{n}v_{0}+(1-\lambda)\sum_{k=1}^{n}\lambda^{n-k}\eta_{k}^{2},\label{eq:-74}
\end{equation}
It will be convenient to consider the following approximation of $v_{n}$,
\begin{equation}
u_{n}=\lambda^{n}v_{0}+(1-\lambda)\sum_{k=1}^{n}\lambda^{n-k}\xi_{k}^{2}.\label{eq:-75}
\end{equation}

We start with some estimates for $\zeta_{k}$, $\xi_{k}$, and $\eta_{k}$.
\begin{lem}
\label{lem:-17}Let $\zeta_{k}$, $\xi_{k}$, $\eta_{k}$ be the random
variables defined in (\ref{eq:-148}), (\ref{eq:-149}) and (\ref{eq:-150})
respectively. Then
\begin{equation}
|\zeta_{k}|\le c_{\sigma^{\ast},M_{\rho}}\Delta t^{1/2},\quad\mathrm{a.s.}\label{eq:-102}
\end{equation}
and, for any $p\ge1$,
\begin{equation}
\mathbb{E}(|\xi_{k}|^{p})\le c_{p,\sigma^{\ast}}<\infty,\label{eq:-101}
\end{equation}
and
\begin{equation}
\mathbb{E}(|\eta_{k}|^{p})\le c_{p,\sigma^{\ast},M_{\rho},T}<\infty,\label{eq:-105}
\end{equation}
where $c_{\sigma^{\ast},M_{\rho}}$, $c_{p,\sigma^{\ast}}$ and $c_{p,\sigma^{\ast},M_{\rho},T}$
are constants depending only on the quantities in their subscripts.
\end{lem}

\begin{proof}
The inequality (\ref{eq:-102}) is obvious. The inequality (\ref{eq:-101})
follows easily from the Burkholder--Davis--Gundy inequality. 

We now prove (\ref{eq:-105}). For any $r\ge0$, By the fact that
$\exp\big(r\Delta t^{-1/2}\int_{0}^{t}\sigma_{s}dW_{s}-2^{-1}r^{2}\Delta t^{-1}\int_{0}^{t}\sigma_{s}^{2}ds\big)$
is a martingale, we have
\begin{align*}
\mathbb{E}(e^{r\xi_{k}}) & =\mathbb{E}\Big[\exp\Big(r\Delta t^{-1/2}\int_{t_{k-1}}^{t_{k}}\sigma_{s}dW_{s}-2^{-1}r^{2}\Delta t^{-1}\int_{t_{k-1}}^{t_{k}}\sigma_{s}^{2}ds\big)\exp\Big(2^{-1}r^{2}\Delta t^{-1}\int_{t_{k-1}}^{t_{k}}\sigma_{s}^{2}ds\big)\Big]\\
 & \le e^{(r\sigma^{\ast})^{2}/2}\mathbb{E}\Big[\exp\Big(r\Delta t^{-1/2}\int_{t_{k-1}}^{t_{k}}\sigma_{s}dW_{s}-2^{-1}r^{2}\Delta t^{-1}\int_{t_{k-1}}^{t_{k}}\sigma_{s}^{2}ds\big)\Big]\\
 & =e^{(r\sigma^{\ast})^{2}/2}.
\end{align*}
Similarly, we have $\mathbb{E}(e^{-r\xi_{k}})\le e^{(r\sigma^{\ast})^{2}/2}$.
Therefore,
\begin{equation}
\mathbb{E}(e^{r|\xi_{k}|})\le2e^{(r\sigma^{\ast})^{2}/2}\le c_{r,\sigma^{\ast},T},\quad r\ge0.\label{eq:-103}
\end{equation}
By $|e^{x}-1|\le|x|\sum_{k=0}^{\infty}|x|^{k}/(k+1)!\le|x|e^{|x|}$
and (\ref{eq:-102}) and (\ref{eq:-101}), we have
\[
\begin{aligned}\mathbb{E}(|\eta_{k}|^{p}) & =\Delta t^{-p/2}\mathbb{E}(|e^{\Delta t^{1/2}(\zeta_{k}+\xi_{k})}-1|^{p})\le\mathbb{E}(|\zeta_{k}+\xi_{k}|^{p}e^{p\Delta t^{1/2}|\zeta_{k}+\xi_{k}|})\\
 & \le\mathbb{E}(|\zeta_{k}+\xi_{k}|^{2p})^{1/2}\mathbb{E}(e^{2p\Delta t^{1/2}|\zeta_{k}+\xi_{k}|})^{1/2}\le c_{p,\sigma^{\ast},M_{\rho},T}\mathbb{E}(e^{2pT^{1/2}|\xi_{k}|})^{1/2},
\end{aligned}
\]
which, together with (\ref{eq:-103}) applied to $r=2pT^{1/2}$, gives
(\ref{eq:-105}).
\end{proof}
We will need the following negative moment estimates for $u_{n}$
and $v_{n}$, where the explicit dependence on the initial variance
$v_{0}$ will be needed for realized variance reset in section \ref{subsec:-9}.
\begin{lem}
\label{lem:-2}Let $u_{n}$ be the random variable defined in (\ref{eq:-75}).
Then, for any $\alpha\in\mathbb{R}$,
\begin{equation}
\sup_{1\le n\le N<\infty}\mathbb{E}(u_{n}^{-\alpha})\le c_{\alpha,\lambda,\sigma_{\ast},\sigma^{\ast}}(1+v_{0}^{-\alpha})<\infty,\label{eq:-76}
\end{equation}
where $c_{\alpha,\lambda,\sigma_{\ast},\sigma^{\ast}}$ is a constant
depending only on $\alpha$, $\lambda$, $\sigma_{\ast}$ and $\sigma^{\ast}$.
\end{lem}

\begin{proof}
For $\alpha<0$, it suffices to prove (\ref{eq:-76}) for $\alpha\le-1$;
the case for $-1<\alpha<0$ follows easily by interpolation. By H\"{o}lder's
inequality,
\[
u_{n}^{-\alpha}\le\lambda^{n}v_{0}^{-\alpha}+(1-\lambda)\sum_{k=1}^{n}\lambda^{n-k}|\xi_{k}|^{-2\alpha}.
\]
Combining (\ref{eq:-101}) and the above proves (\ref{eq:-76}) for
$\alpha<0$.

For $\alpha>0$, note that $x^{-\alpha}=\Gamma(\alpha)^{-1}\int_{0}^{\infty}e^{-xt}t^{\alpha-1}dt$
for $x>0$, where $\Gamma(z)$ is the Gamma function. By Fubini's
theorem,
\begin{equation}
\begin{aligned}\mathbb{E}(u_{n}^{-\alpha}) & =\frac{1}{\Gamma(\alpha)}\mathbb{E}\Big(\int_{0}^{\infty}e^{-u_{n}t}t^{\alpha-1}dt\Big)=\frac{1}{\Gamma(\alpha)}\int_{0}^{\infty}e^{-t\lambda^{n}v_{0}}t^{\alpha-1}\mathbb{E}\Big(\prod_{k=1}^{n}e^{-t(1-\lambda)\lambda^{n-k}\xi_{k}^{2}}\Big)dt,\end{aligned}
\label{eq:-94}
\end{equation}
By Lemma \ref{lem:-1},
\[
\begin{aligned}\mathbb{E}\Big(\prod_{k=1}^{n}e^{-t(1-\lambda)\lambda^{n-k}\xi_{k}^{2}}\Big) & =\mathbb{E}\Big(\prod_{k=1}^{n-1}e^{-t(1-\lambda)\lambda^{n-k}\xi_{k}^{2}}\mathbb{E}\big(e^{-t(1-\lambda)\xi_{n}^{2}}\big|\mathcal{F}_{t_{n-1}}\big)\Big)\\
 & \le[1+2t(1-\lambda)(\sigma^{\ast})^{2}]^{-(\sigma_{\ast}/\sigma^{\ast})^{2}/2}\mathbb{E}\Big(\prod_{k=1}^{n-1}e^{-t(1-\lambda)\lambda^{n-k}\xi_{k}^{2}}\Big)\\
 & \le\cdots\\
 & \le\prod_{k=1}^{n}[1+2t(1-\lambda)\lambda^{n-k}(\sigma^{\ast})^{2}]^{-\theta/2},
\end{aligned}
\]
where $\theta=(\sigma_{\ast}/\sigma^{\ast})^{2}$. Therefore,
\begin{equation}
\mathbb{E}(u_{n}^{-\alpha})\le\frac{1}{\Gamma(\alpha)}\int_{0}^{\infty}e^{-t\lambda^{n}v_{0}}t^{\alpha-1}\prod_{k=0}^{n-1}[1+2t(1-\lambda)\lambda^{k}(\sigma^{\ast})^{2}]^{-\theta/2}dt.\label{eq:-89}
\end{equation}
Let $m=\lfloor2\alpha\theta^{-1}\rfloor+1$. Then $m\theta/2>\alpha$,
and therefore, for any $n\ge m$ and $N\ge n$, we have
\begin{equation}
\mathbb{E}(u_{n}^{-\alpha})\le\frac{1}{\Gamma(\alpha)}\int_{0}^{\infty}t^{\alpha-1}[1+2t(1-\lambda)\lambda^{m-1}(\sigma^{\ast})^{2}]^{-\theta m/2}dt<\infty.\label{eq:-90}
\end{equation}
Clearly, by (\ref{eq:-94}),
\begin{equation}
\mathbb{E}(u_{n}^{-\alpha})\le\frac{1}{\Gamma(\alpha)}\int_{0}^{\infty}e^{-t\lambda^{m}v_{0}}t^{\alpha-1}dt=\lambda^{-\alpha m}v_{0}^{-\alpha},\label{eq:-73}
\end{equation}
for any $n<m$ and $N\ge n$. The uniform bound (\ref{eq:-76}) follows
readily from (\ref{eq:-90}) and (\ref{eq:-73}), since $\theta$
and $m$ are constants depending only on $\alpha$, $\sigma_{\ast}$
and $\sigma^{\ast}$.
\end{proof}
\begin{lem}
\label{lem:-4}Let $v_{n}$ be the random variable defined in (\ref{eq:-74}).
Then, for any $\alpha\in\mathbb{R}$,
\begin{equation}
\sup_{1\le n\le N<\infty}\mathbb{E}(v_{n}^{-\alpha})\le c_{\alpha,\lambda,\sigma_{\ast},\sigma^{\ast},M_{\rho},T}(1+v_{0}^{-\alpha})<\infty,\label{eq:-93}
\end{equation}
where $c_{\alpha,\lambda,\sigma_{\ast},\sigma^{\ast},M_{\rho},T}$
is a constant depending only on $\alpha$, $\lambda$, $\sigma_{\ast}$,
$\sigma^{\ast}$, $M_{\rho}$ and $T$.
\end{lem}

\begin{proof}
The proof is similar to that of Lemma \ref{lem:-2}. Let $\zeta_{k}$,
$\xi_{k}$ and $\eta_{k}$ be the random variables defined in (\ref{eq:-148}),
(\ref{eq:-149}) and (\ref{eq:-150}) respectively. 

For $\alpha<0$, it suffices to prove (\ref{eq:-93}) for $\alpha\le-1$;
the case for $-1<\alpha<0$ follows easily by interpolation. By H\"{o}lder's
inequality,
\[
v_{n}^{-\alpha}\le\lambda^{n}v_{0}^{-\alpha}+(1-\lambda)\sum_{k=1}^{n}\lambda^{n-k}|\eta_{k}|^{-2\alpha}.
\]
Therefore, the inequality (\ref{eq:-93}) for $\alpha<0$ follows
from (\ref{eq:-105}) immediately. 

For $\alpha>0$, similar to (\ref{eq:-94}), we have
\begin{equation}
\begin{aligned}\mathbb{E}(v_{n}^{-\alpha}) & =\frac{1}{\Gamma(\alpha)}\int_{0}^{\infty}e^{-t\lambda^{n}v_{0}}t^{\alpha-1}\mathbb{E}\Big(\prod_{k=1}^{n}e^{-t(1-\lambda)\lambda^{n-k}\eta_{k}^{2}}\Big)dt.\end{aligned}
\label{eq:-96}
\end{equation}
It is easily seen that $|e^{x}-1|\ge\min(1,|x|)/2$ for all $x\in\mathbb{R}$.
Therefore, $|\eta_{k}|\ge\Delta t^{-1/2}\min(1,\Delta t^{1/2}|\zeta_{k}+\xi_{k}|)/2$.
Consequently,
\begin{equation}
\begin{aligned}e^{-t(1-\lambda)\lambda^{n-k}\eta_{k}^{2}} & \le\max\big(e^{-t(1-\lambda)\lambda^{n-k}\Delta t^{-1}/4},e^{-t(1-\lambda)\lambda^{n-k}(\zeta_{k}+\xi_{k})^{2}/4}\big)\\
 & \le e^{-t(1-\lambda)\lambda^{n-k}T^{-1}/4}+e^{-t(1-\lambda)\lambda^{n-k}(\zeta_{k}+\xi_{k})^{2}/4}\\
 & \le[1+t\mu(1-\lambda)\lambda^{n-k}]^{-\theta}+e^{-t(1-\lambda)\lambda^{n-k}(\zeta_{k}+\xi_{k})^{2}/4},
\end{aligned}
\label{eq:-72}
\end{equation}
where $\theta=(\sigma_{\ast}/\sigma^{\ast})^{2}$ and $\mu=\min[(\sigma^{\ast})^{2},(\sigma^{\ast}/\sigma_{\ast})^{2}T^{-1}]$.
Let $M=M_{\rho}+(\sigma^{\ast})^{2}/2$. By (\ref{eq:-72}) and Lemma
\ref{lem:-3} with $p=2$,
\[
\begin{aligned}\mathbb{E}\big(e^{-t(1-\lambda)\lambda^{n-k}\eta_{k}^{2}}\big|\mathcal{F}_{t_{k-1}}\big) & \le[1+t\mu(1-\lambda)\lambda^{n-k}]^{-\theta/4}+e^{M^{2}\sigma_{\ast}^{-2}\Delta t/2}[1+t(1-\lambda)\lambda^{n-k}(\sigma^{\ast})^{2}]^{-\theta/4}\\
 & \le2e^{M^{2}\sigma_{\ast}^{-2}\Delta t/2}[1+t\mu(1-\lambda)\lambda^{n-k}]^{-\theta/4}.
\end{aligned}
\]
Let $m=\lfloor4\alpha\theta^{-1}\rfloor+1$. Then, for any $n\ge m$,
\[
\begin{aligned}\mathbb{E}\Big(\prod_{k=1}^{n}e^{-t(1-\lambda)\lambda^{n-k}\eta_{k}^{2}}\Big) & =\mathbb{E}\Big(\prod_{k=1}^{n-1}e^{-t(1-\lambda)\lambda^{n-k}\eta_{k}^{2}}\mathbb{E}\big(e^{-t(1-\lambda)\eta_{n}^{2}}\big|\mathcal{F}_{t_{n-1}}\big)\Big)\\
 & \le2e^{M^{2}\sigma_{\ast}^{-2}\Delta t/2}[1+t\mu(1-\lambda)]^{-\theta/4}\mathbb{E}\Big(\prod_{k=1}^{n-1}e^{-t(1-\lambda)\lambda^{n-k}\eta_{k}^{2}}\Big)\\
 & \le\cdots\\
 & \le2^{m}e^{M^{2}\sigma_{\ast}^{-2}m\Delta t/2}\prod_{j=0}^{m-1}[1+t\mu(1-\lambda)\lambda^{j}]^{-\theta/4}\mathbb{E}\Big(\prod_{k=1}^{n-m}e^{-t(1-\lambda)\lambda^{n-k}\eta_{k}^{2}}\Big)\\
 & \le2^{m}e^{M^{2}\sigma_{\ast}^{-2}T/2}[1+t\mu(1-\lambda)\lambda^{m-1}]^{-m\theta/4}.
\end{aligned}
\]
Since $m\theta/4>\alpha$, for any $n\ge m$ and $N\ge n$, we have
\begin{equation}
\mathbb{E}(v_{n}^{-\alpha})\le\frac{2^{m}e^{M^{2}\sigma_{\ast}^{-2}T/2}}{\Gamma(\alpha)}\int_{0}^{\infty}t^{\alpha-1}[1+t\mu(1-\lambda)\lambda^{m-1}]^{-m\theta/4}dt<\infty,\label{eq:-88}
\end{equation}
Clearly, by (\ref{eq:-96}),
\begin{equation}
\mathbb{E}(v_{n}^{-\alpha})\le\frac{1}{\Gamma(\alpha)}\int_{0}^{\infty}e^{-t\lambda^{m}v_{0}}t^{\alpha-1}dt=\lambda^{-\alpha m}v_{0}^{-\alpha},\label{eq:-95}
\end{equation}
for any $n<m$ and $N\ge n$. The uniform bound (\ref{eq:-93}) follows
from (\ref{eq:-88}) and (\ref{eq:-95}), since $\theta$, $\mu$,
$m$ and $M$ are constants depending only on $\alpha$, $\sigma_{\ast}$,
$\sigma^{\ast}$, $M_{\rho}$ and $T$.
\end{proof}
The following Komogorov continuity property of $\{\log I_{t}\}_{t\ge0}$
is a simple corollary of Lemma \ref{lem:-4}. Its proof is the same
as \cite[Lemma 2.9]{LG26} once the uniform bound (\ref{eq:-93})
is obtained.
\begin{cor}
\label{cor:}Let $0\le s\le t\le T$ and $1\le p<\infty$. Then 
\begin{equation}
\mathbb{E}\big(|\log I_{t}-\log I_{s}|^{p}\big)^{1/p}\le c_{p,\lambda,v_{0},\sigma_{\ast},\sigma^{\ast},M_{r},M_{\rho},T}|t-s|^{1/2},\label{eq:-115}
\end{equation}
where $c_{p,\lambda,v_{0},\sigma_{\ast},\sigma^{\ast},M_{r},M_{\rho},T}$
is a constant depending only on $p$, $\lambda$, $v_{0}$, $\sigma_{\ast}$,
$\sigma^{\ast}$, $M_{r}$, $M_{\rho}$ and $T$.
\end{cor}

\begin{lem}
\label{lem:-5}Let $u_{n}$ and $v_{n}$ be the random variables defined
in (\ref{eq:-75}) and (\ref{eq:-74}) respectively. Then, for any
$\alpha>0$ and any $p>0$, 
\begin{equation}
\mathbb{E}(|u_{n}^{-\alpha}-v_{n}^{-\alpha}|^{p})^{1/p}\le c_{\alpha,p,\lambda,v_{0},\sigma_{\ast},\sigma^{\ast},M_{\rho},T}\Delta t^{1/2},\label{eq:-97}
\end{equation}
where $c_{\alpha,p,\lambda,v_{0},\sigma_{\ast},\sigma^{\ast},M_{\rho},T}$
is a constant depending only on $\alpha$, $p$, $v_{0}$, $\sigma_{\ast}$,
$\sigma^{\ast}$, $M_{\rho}$ and $T$.
\end{lem}

\begin{proof}
We may assume $p\ge1$. The case $0<p<1$ follows easily from H\"{o}lder's
inequality. Let $f(x)=x^{-\alpha}$. By Taylor's theorem and the monotonicity
of $f^{\prime}(x)=-\alpha x^{-\alpha-1}$, we have
\[
|f(u_{n})-f(v_{n})|\le\big(|f^{\prime}(u_{n})|+|f^{\prime}(v_{n})|\big)|u_{n}-v_{n}|.
\]
By H\"{o}lder's inequality,
\[
\mathbb{E}(|u_{n}^{-\alpha}-v_{n}^{-\alpha}|^{p})\le\big[\mathbb{E}(|f^{\prime}(u_{n})|^{2p})^{1/2}+\mathbb{E}(|f^{\prime}(v_{n})|^{2p})^{1/2}\big]\mathbb{E}(|u_{n}-v_{n}|^{2p})^{1/2}.
\]
By Lemma \ref{lem:-2} and Lemma \ref{lem:-4}, 
\[
\sup_{1\le n\le N<\infty}\big[\mathbb{E}(|f^{\prime}(u_{n})|^{2p})^{1/2}+\mathbb{E}(|f^{\prime}(v_{n})|^{2p})^{1/2}\big]\le c_{\alpha,p,\lambda,v_{0},\sigma_{\ast},\sigma^{\ast},M_{\rho},T}.
\]
Therefore, it suffices to show that
\begin{equation}
\mathbb{E}(|u_{n}-v_{n}|^{2p})^{1/(2p)}\le c_{p,\lambda,v_{0},\sigma_{\ast},\sigma^{\ast},M_{\rho},T}\Delta t^{1/2}.\label{eq:-99}
\end{equation}

Let $\xi_{k}$, $\zeta_{k}$ and $\eta_{k}$ be the random variables
defined in (\ref{eq:-149}), (\ref{eq:-148}) and (\ref{eq:-150})
respectively. Then 
\begin{equation}
u_{n}-v_{n}=(1-\lambda)\sum_{k=1}^{n}\lambda^{n-k}(\xi_{k}^{2}-\eta_{k}^{2}).\label{eq:-100}
\end{equation}
We show that 
\begin{equation}
\mathbb{E}(|\xi_{k}-\eta_{k}|^{r})^{1/r}\le c_{r,\sigma^{\ast},M_{\rho},T}\Delta t^{1/2}.\label{eq:-104}
\end{equation}
Once (\ref{eq:-104}) is obtained, by H\"{o}lder's inequality and
the estimates (\ref{eq:-101}), (\ref{eq:-105}) and (\ref{eq:-104}),
we have
\[
\begin{aligned}\mathbb{E}(|u_{n}-v_{n}|^{2p}) & \le(1-\lambda)\sum_{k=1}^{n}\lambda^{n-k}\mathbb{E}(|\xi_{k}^{2}-\eta_{k}^{2}|^{2p})\\
 & \le(1-\lambda)\sum_{k=1}^{n}\lambda^{n-k}\mathbb{E}(|\xi_{k}+\eta_{k}|^{4p})^{1/2}\mathbb{E}(|\xi_{k}-\eta_{k}|^{4p})^{1/2}\\
 & \le c_{p,\sigma^{\ast},M_{\rho},T}(1-\lambda)\sum_{k=1}^{n}\lambda^{n-k}\Delta t^{p}\\
 & \le c_{p,\sigma^{\ast},M_{\rho},T}\Delta t^{p},
\end{aligned}
\]
which proves (\ref{eq:-99}) and therefore completes the proof.

We now turn to the proof of (\ref{eq:-105}) and (\ref{eq:-104}).
To see (\ref{eq:-104}), by $|e^{x}-1-x|\le|x|^{2}\sum_{k=0}^{\infty}|x|^{k}/(k+2)!\le|x|^{2}e^{|x|}$
and (\ref{eq:-102}), we have 
\[
\begin{aligned}|\xi_{k}-\eta_{k}| & =\Delta t^{-1/2}\big|e^{\Delta t^{1/2}(\zeta_{k}+\xi_{k})}-1-\Delta t^{1/2}\xi_{k}\big|\\
 & \le\Delta t^{-1/2}\big|e^{\Delta t^{1/2}(\zeta_{k}+\xi_{k})}-1-\Delta t^{1/2}(\zeta_{k}+\xi_{k})\big|+|\zeta_{k}|\\
 & \le\Delta t^{1/2}|\zeta_{k}+\xi_{k}|^{2}e^{\Delta t^{1/2}|\zeta_{k}+\xi_{k}|}+c_{\sigma^{\ast},\rho}\Delta t^{1/2}\\
 & \le c_{\sigma^{\ast},M_{\rho},T}(1+|\zeta_{k}+\xi_{k}|^{2}e^{\Delta t^{1/2}|\xi_{k}|})\Delta t^{1/2}.
\end{aligned}
\]
Therefore,
\[
\begin{aligned}\mathbb{E}(|\xi_{k}-\eta_{k}|^{r}) & \le c_{r,\sigma^{\ast},M_{\rho},T}\Delta t^{r/2}[1+\mathbb{E}(|\zeta_{k}+\xi_{k}|^{2r}e^{r\Delta t^{1/2}|\xi_{k}|})]\\
 & \le c_{r,\sigma^{\ast},M_{\rho},T}\Delta t^{r/2}[1+\mathbb{E}(|\zeta_{k}+\xi_{k}|^{4r})^{1/2}\mathbb{E}(e^{2rT^{1/2}|\xi_{k}|})^{1/2}],
\end{aligned}
\]
which, together with (\ref{eq:-102}), (\ref{eq:-101}) and (\ref{eq:-103}),
proves (\ref{eq:-104}). 
\end{proof}
The following estimate will be used repeatedly. We skip its proof
since it is a simple application of Doob's maximal inequality and
the Burkholder--Davis--Gundy inequality.
\begin{lem}
\label{lem:-9}Let 
\[
Y_{t}=\int_{0}^{t}a_{s}ds+\int_{0}^{t}b_{s}dW_{s},\quad t\in[0,T],
\]
for some progressively measurable processes $\{a_{t}\}_{t\ge0}$ and
$\{b_{t}\}_{t\ge0}$. Then, for any $p\ge1$,
\[
\mathbb{E}\Big(\sup_{t\in[0,T]}|Y_{t}|^{p}\Big)^{1/p}\le c_{p,T}\Big[\Big(\int_{0}^{T}\mathbb{E}(|a_{t}|^{p})dt\Big)^{1/p}+\Big(\int_{0}^{T}\mathbb{E}(|b_{t}|^{p})dt\Big)^{1/p}\Big],
\]
where $c_{p,T}$ is a constant depending on $p$ and $T$.
\end{lem}

The following proposition reduces the study of convergence of $\{I_{t}\}_{t\ge0}$
as $\Delta t\to0$ to that of the process obtained by replacing the
realized variance $v_{n}$ by $u_{n}$, where $u_{n}$ is the random
variable defined in (\ref{eq:-75}). The process with realized variance
given by $u_{n}$ is more convenient for the computation of the drift
and diffusion coefficients of the limiting process.
\begin{prop}
\label{prop:-3}Let $\{J_{t}\}_{t\ge0}$ be the process defined by
\begin{equation}
\begin{aligned}\dfrac{dJ_{t}}{J_{t}} & =(1-\bar{\sigma}u_{n-1}^{-1/2})r(t)dt+\bar{\sigma}u_{n-1}^{-1/2}\frac{dS_{t}}{S_{t}},\quad t\in[t_{n-1},t_{n}),\\
J_{0} & =1,
\end{aligned}
\label{eq:-109}
\end{equation}
where $u_{n}$ is given by (\ref{eq:-75}). Then
\begin{equation}
\lim_{\Delta t\to0}\mathbb{E}\Big(\sup_{t\in[0,T]}|\log I_{t}-\log J_{t}|^{2}\Big)^{1/2}=0.\label{eq:-110}
\end{equation}
Therefore, $\{(I_{t},W_{t},\tilde{W}_{t})\}_{t\ge0}$ converges in
distribution if and only if $\{(J_{t},W_{t},\tilde{W}_{t})\}_{t\ge0}$
converges in distribution to the same limiting process.
\end{prop}

\begin{proof}
By It\"{o}'s formula, for $t\in[t_{n-1},t_{n})$,
\[
d(\log I_{t}-\log J_{t})=[\bar{\sigma}(r(t)-\rho(t))a_{n-1}+\bar{\sigma}^{2}\sigma_{t}^{2}b_{n-1}/2]dt+\bar{\sigma}a_{n-1}\sigma_{t}dW_{t},
\]
where $a_{n-1}=u_{n-1}^{-1/2}-v_{n-1}^{-1/2}$ and $b_{n-1}=u_{n-1}^{-1}-v_{n-1}^{-1}$.
By Lemma \ref{lem:-9} and the boundedness of $\rho(t)$, $r(t)$
and $\sigma_{t}$,
\begin{equation}
\begin{aligned}\mathbb{E}\Big(\sup_{t\in[0,T]}|\log I_{t}-\log J_{t}|^{2}\Big)^{1/2} & \le c_{M_{r},M_{\rho},\sigma^{\ast},\bar{\sigma}}\Big[\Big(\sum_{n=1}^{N}\mathbb{E}(|a_{n-1}|^{2})\Delta t\Big)^{1/2}+\Big(\sum_{n=1}^{N}\mathbb{E}(|b_{n-1}|^{2})\Delta t\Big)^{1/2}\Big]\end{aligned}
\label{eq:-112}
\end{equation}
By Lemma \ref{lem:-5},
\begin{equation}
\sum_{n=1}^{N}\mathbb{E}(a_{n-1}^{2})\le c_{\lambda,v_{0},\sigma_{\ast},\sigma^{\ast},M_{\rho},T}N\Delta t=c_{\lambda,v_{0},\sigma_{\ast},\sigma^{\ast},M_{\rho},T}.\label{eq:-114}
\end{equation}
and similarly,
\begin{equation}
\sum_{n=1}^{N}\mathbb{E}(b_{n-1}^{2})\le c_{\lambda,v_{0},\sigma_{\ast},\sigma^{\ast},M_{\rho},T}N\Delta t=c_{\lambda,v_{0},\sigma_{\ast},\sigma^{\ast},M_{\rho},T},\label{eq:-113}
\end{equation}
Combining (\ref{eq:-112}), (\ref{eq:-113}) and (\ref{eq:-114})
completes the proof.
\end{proof}
We end this section with the following lemma which, when applied to
$f(t)=\prod_{k=0}^{\infty}(1+t\lambda^{k})^{-1/2}$, gives that $U(\lambda)^{2}<V(\lambda)$.
\begin{lem}
\label{lem:-6}Suppose that $f(x)$ is a positive Borel function on
$[0,\infty)$ such that
\begin{equation}
f(x+y)\ge f(x)f(y).\label{eq:-106}
\end{equation}
Then
\begin{equation}
\Big(\int_{0}^{\infty}f(x^{2})dx\Big)^{2}\le\frac{\pi}{2}\int_{0}^{\infty}f(x)dx.\label{eq:-107}
\end{equation}
The equality of (\ref{eq:-107}) holds if and only if $f(x+y)=f(x)f(y)$
a.e.
\end{lem}

\begin{proof}
By change of variables,
\[
\begin{aligned}\Big(\int_{0}^{\infty}f(x^{2})dx\Big)^{2} & =\int_{0}^{\infty}\int_{0}^{\infty}f(x^{2})f(y^{2})dxdy=\int_{0}^{\pi/2}\int_{0}^{\infty}f(r^{2}\cos^{2}\theta)f(r^{2}\sin^{2}\theta)rdrd\theta.\end{aligned}
\]
By (\ref{eq:-106}), $f(r^{2}\cos^{2}\theta)f(r^{2}\sin^{2}\theta)\le f(r^{2})$.
Therefore,
\begin{equation}
\Big(\int_{0}^{\infty}f(x^{2})dx\Big)^{2}\le\int_{0}^{\pi/2}\int_{0}^{\infty}f(r^{2})rdrd\theta=\frac{\pi}{2}\int_{0}^{\infty}f(x)dx.\label{eq:-108}
\end{equation}
Clearly, the equality of (\ref{eq:-108}) holds if and only if $f(x+y)=f(x)f(y)$
a.e.
\end{proof}

\section{\label{sec:-5}The Structural Limit Theorem}

\subsection{\label{subsec:-9}Volatility Targeting with Realized Variance Reset}

In this subsection, we introduce the realized variance reset for the
volatility target index $\{I_{t}\}_{t\ge0}$, which will be useful
in reducing convergence of finite dimensional distribution to that
of one dimensional distribution.
\begin{lem}
\label{lem:-10}Let $\{I_{t}\}_{t\ge0}$ and $\{\tilde{I}_{t}\}_{t\ge0}$
be two volatility target indices with the same initial level $I_{0}$,
the same observation window parameter $\lambda$, and different initial
realized variances $v_{0}$ and $\tilde{v}_{0}$ respectively. Then
\begin{equation}
\mathbb{E}\Big(\sup_{t\in[0,T]}|\log I_{t}-\log\tilde{I}_{t}|^{2}\Big)^{1/2}\le c_{\lambda,M_{\rho},M_{r},\sigma_{\ast},\sigma^{\ast},T}(1+v_{0}^{-2}+\tilde{v}_{0}^{-2})|v_{0}-\tilde{v}_{0}|\Delta t^{1/2},\label{eq:-147}
\end{equation}
where $c_{\lambda,M_{\rho},M_{r},\sigma_{\ast},\sigma^{\ast},T}$
is a constant depending only on $\lambda$, $M_{\rho}$, $M_{r}$,
$\sigma_{\ast}$, $\sigma^{\ast}$ and $T$.
\end{lem}

\begin{proof}
Throughout the proof, $c_{q_{1},\dots,q_{m}}$ will denote constants
depending only on the quantities $q_{1},\dots,q_{m}$, and it may
be different at different appearances. Let $\{\tilde{v}_{n}\}_{n\ge0}$
be realized variance process corresponding to the initial variance
$\tilde{v}_{0}$; that is, 
\begin{equation}
\tilde{v}_{n}=\lambda^{n}\tilde{v}_{0}+\frac{1-\lambda}{\Delta t}\sum_{k=1}^{n}\lambda^{n-k}\Big(\frac{S_{t_{n}}}{S_{t_{n-1}}}-1\Big)^{2},\quad1\le n\le N.\label{eq:-142}
\end{equation}
Then $v_{n}-\tilde{v}_{n}=\lambda^{n}(v_{0}-\tilde{v}_{0})$. By It\"{o}'s
formula,
\begin{equation}
\log I_{t}-\log\tilde{I}_{t}=[\bar{\sigma}(\rho(t)-r(t))a_{n-1}-\bar{\sigma}^{2}\sigma_{t}^{2}b_{n-1}/2]dt+\bar{\sigma}\sigma_{t}a_{n-1}dW_{t},\label{eq:-143}
\end{equation}
for $t\in[t_{n-1},t_{n})$, where $a_{n}=v_{n}^{-1/2}-\tilde{v}_{n}^{-1/2}$
and $b_{n}=v_{n}^{-1}-\tilde{v}_{n}^{-1}$. By Lemma \ref{lem:-9}
and the boundedness of $\rho(t)$, $r(t)$ and $\sigma_{t}$,
\begin{equation}
\begin{aligned}\mathbb{E}\Big(\sup_{t\in[0,T]}|\log I_{t}-\log\tilde{I}_{t}|^{2}\Big)^{1/2} & \le c_{M_{\rho},M_{r},\sigma^{\ast},T}\Big[\Big(\sum_{n=1}^{N}\mathbb{E}(a_{n-1}^{2})\Delta t\Big)^{1/2}+\Big(\sum_{n=1}^{N}\mathbb{E}(b_{n-1}^{2})\Delta t\Big)^{1/2}\Big].\end{aligned}
\label{eq:-144}
\end{equation}
For any $\alpha>0$, by Taylor's theorem,
\[
\begin{aligned}|v_{n}^{-\alpha}-\tilde{v}_{n}^{-\alpha}| & \le\alpha(v_{n}^{-\alpha-1}+\tilde{v}_{n}^{-\alpha-1})|v_{n}-\tilde{v}_{n}|=\alpha|v_{0}-\tilde{v}_{0}|\lambda^{n}(v_{n}^{-\alpha-1}+\tilde{v}_{n}^{-\alpha-1})|.\end{aligned}
\]
Therefore,
\[
\mathbb{E}(|v_{n}^{-\alpha}-\tilde{v}_{n}^{-\alpha}|^{2})^{1/2}\le\alpha|v_{0}-\tilde{v}_{0}|\lambda^{n}[\mathbb{E}(v_{n}^{-2\alpha-2})^{1/2}+\mathbb{E}(\tilde{v}_{n}^{-2\alpha-2})^{1/2}],
\]
which, together with Lemma \ref{lem:-4}, implies that 
\[
\mathbb{E}(|v_{n}^{-\alpha}-\tilde{v}_{n}^{-\alpha}|^{2})^{1/2}\le c_{\alpha,\lambda,\sigma_{\ast},\sigma^{\ast},M_{\rho},T}|v_{0}-\tilde{v}_{0}|(1+v_{0}^{-\alpha-1}+\tilde{v}_{0}^{-\alpha-1})\lambda^{n}.
\]
Therefore,
\begin{equation}
\mathbb{E}(a_{n-1}^{2})^{1/2}\le c_{\lambda,\sigma_{\ast},\sigma^{\ast},M_{\rho},T}|v_{0}-\tilde{v}_{0}|(1+v_{0}^{-3/2}+\tilde{v}_{0}^{-3/2})\lambda^{2n},\label{eq:-145}
\end{equation}
and
\begin{equation}
\mathbb{E}(b_{n-1}^{2})^{1/2}\le c_{\lambda,\sigma_{\ast},\sigma^{\ast},M_{\rho},T}|v_{0}-\tilde{v}_{0}|(1+v_{0}^{-2}+\tilde{v}_{0}^{-2})\lambda^{2n}.\label{eq:-146}
\end{equation}
Combining (\ref{eq:-144}), (\ref{eq:-145}) and (\ref{eq:-146})
completes the proof.
\end{proof}
\begin{defn}
\label{def:}Let $0<\tau_{1}<\cdots<\tau_{m}<T$ be fixed constants,
and 
\begin{equation}
n(\tau_{k})=\inf\{n\ge0:t_{n}\ge\tau_{k}\},\quad1\le k\le m.\label{eq:-116}
\end{equation}
The \emph{volatility target index with realized variance reset at
times $\tau_{1},\dots,\tau_{m}$} is the continuous process $\{\tilde{I}_{t}\}_{t\ge0}$
defined by 

\begin{equation}
\begin{aligned}\dfrac{d\tilde{I}_{t}}{\tilde{I}_{t}} & =(1-\bar{\sigma}\tilde{v}_{n-1}^{-1/2})r(t)dt+\bar{\sigma}\tilde{v}_{n-1}^{-1/2}\frac{dS_{t}}{S_{t}},\quad t\in[t_{n-1},t_{n}),\\
\tilde{I}_{0} & =1,
\end{aligned}
\label{eq:-98}
\end{equation}
where the realized variance process is defined by
\begin{equation}
\tilde{v}_{n}=\lambda\tilde{v}_{n-1}+\frac{1-\lambda}{\Delta t}\Big(\frac{S_{t_{n}}}{S_{t_{n-1}}}-1\Big)^{2},\quad1\le n\le N,\;n\not\in\{n(\tau_{1}),\dots n(\tau_{m})\},\label{eq:-139}
\end{equation}
and 
\begin{equation}
\tilde{v}_{0}=\tilde{v}_{n(\tau_{1})}=\dots=\tilde{v}_{n(\tau_{m})}=v_{0}.\label{eq:-140}
\end{equation}
\end{defn}

The following lemma allows to remove the dependence on the realized
variance $v_{n}$ with $t_{n-1}<s\le t_{n}$ when considering the
conditional limiting distribution of $\log I_{t}-\log I_{s}$ with
respect to $\mathcal{F}_{s}$.
\begin{lem}
\label{lem:-8}Let $0<\tau_{1}<\cdots<\tau_{m}<T$ be fixed constants,
and $\{\tilde{I}_{t}\}_{t\ge0}$ be the volatility target index with
realized variance reset at times \emph{$\tau_{1},\dots,\tau_{m}$.
Then
\begin{equation}
\mathbb{E}\Big(\sup_{t\in[0,T]}|\log I_{t}-\log\tilde{I}_{t}|^{2}\Big)^{1/2}\le c_{m,\lambda,v_{0},M_{\rho},M_{r},\sigma_{\ast},\sigma^{\ast},T}\Delta t^{1/2},\label{eq:-141}
\end{equation}
where $c_{m,\lambda,v_{0},M_{\rho},M_{r},\sigma_{\ast},\sigma^{\ast},T}$
is a constant depending only on $m$, $\lambda$, $v_{0}$, $M_{\rho}$,
$M_{r}$, $\sigma_{\ast}$, $\sigma^{\ast}$ and $T$. Therefore,
\begin{equation}
\lim_{\Delta t\to0}\mathbb{E}\Big(\sup_{t\in[0,T]}|\log I_{t}-\log\tilde{I}_{t}|^{2}\Big)^{1/2}=0.\label{eq:-155}
\end{equation}
}
\end{lem}

\begin{proof}
It suffices to prove (\ref{eq:-141}) for the case $m=1$; for general
$m\ge1$, it follows easily from the $m=1$ case and resetting the
realized variance of $\tilde{I}_{t}$. Let $\tau=\tau_{1}$ be the
reset time and $n(\tau)=\inf\{n\ge0:t_{n}\ge\tau\}$. Clearly, $I_{t}=\tilde{I}_{t}$
for $t\in[0,t_{n(\tau)}]$. Consequently,
\begin{equation}
\mathbb{E}\Big(\sup_{t\in[0,T]}|\log I_{t}-\log\tilde{I}_{t}|^{2}\Big)=\mathbb{E}\Big[\mathbb{E}\Big(\sup_{t\in[t_{n(\tau)},T]}|\log I_{t}-\log\tilde{I}_{t}|^{2}\Big|\mathcal{F}_{t_{n(\tau)}}\Big)\Big].\label{eq:-152}
\end{equation}
Let $J_{t}=I_{t}/I_{t_{n}(\tau)}$ and $\tilde{J}_{t}=\tilde{I}_{t}/\tilde{I}_{t_{n}(\tau)}$.
When conditional on $\mathcal{F}_{t_{n(\tau)}}$, $\{J_{t}\}_{t\ge t_{n(\tau)}}$
and $\{\tilde{J}_{t}\}_{t\ge t_{n(\tau)}}$ are volatility target
indices with initial realized variances $v_{n(\tau)}$ and $v_{0}$
respectively. By Lemma \ref{lem:-10}, 
\[
\mathbb{E}\Big(\sup_{t\in[t_{n(\tau)},T]}|\log J_{t}-\log\tilde{J}_{t}|^{2}\Big|\mathcal{F}_{t_{n(\tau)}}\Big)\le c_{\lambda,M_{\rho},M_{r},\sigma_{\ast},\sigma^{\ast},T}(1+v_{0}^{-4}+v_{n(\tau)}^{-4})|v_{0}-v_{n(\tau)}|^{2}\Delta t.
\]
Note that $I_{t_{n}(\tau)}=\tilde{I}_{t_{n}(\tau)}$, by (\ref{eq:-152})
and the above,
\begin{equation}
\begin{aligned}\mathbb{E}\Big(\sup_{t\in[0,T]}|\log I_{t}-\log\tilde{I}_{t}|^{2}\Big) & \le c_{\lambda,M_{\rho},M_{r},\sigma_{\ast},\sigma^{\ast},T}\Delta t\mathbb{E}[(1+v_{0}^{-2}+v_{n(\tau)}^{-4})|v_{0}-v_{n(\tau)}|^{2}]\\
 & \le c_{\lambda,M_{\rho},M_{r},\sigma_{\ast},\sigma^{\ast},T}\Delta t[1+v_{0}^{-2}+\mathbb{E}(v_{n(\tau)}^{-8})^{1/2}][v_{0}^{2}+\mathbb{E}(v_{n(\tau)})^{4})^{1/2}].
\end{aligned}
\label{eq:-153}
\end{equation}
By Lemma \ref{lem:-4} and (\ref{eq:-153}), we deduce that
\[
\mathbb{E}\Big(\sup_{t\in[0,T]}|\log I_{t}-\log\tilde{I}_{t}|^{2}\Big)^{1/2}\le c_{\lambda,v_{0},M_{\rho},M_{r},\sigma_{\ast},\sigma^{\ast},T}\Delta t^{1/2}.
\]
This completes the proof.
\end{proof}

\subsection{\label{subsec:-10}Continuity in $L^{2}(\Omega;C([0,T]))$ with Respect
to Coefficients}

We start with the continuity of the variables $\zeta_{k}$, $\xi_{k}$,
$\eta_{k}$ defined by (\ref{eq:-148}), (\ref{eq:-149}) and (\ref{eq:-150})
with respect to $\rho(t)$ and $\{\sigma_{t}\}_{t\ge0}$.
\begin{lem}
\label{lem:-18}Let $\zeta_{k}$, $\xi_{k}$, $\eta_{k}$ be the random
variables defined in (\ref{eq:-148}), (\ref{eq:-149}) and (\ref{eq:-150})
respectively. Suppose that $\tilde{\rho}(t)$ is a deterministic function
with $|\tilde{\rho}(t)|\le M_{\rho}$, and that $\{\tilde{\sigma}_{t}\}_{t\ge0}$
is a progressively measurable process with $\sigma_{\ast}\le\tilde{\sigma}_{t}\le\sigma^{\ast}$.
Let $\tilde{\zeta}_{k}$, $\tilde{\xi}_{k}$, $\tilde{\eta}_{k}$
be the random variables defined by replacing $(\rho(t),\sigma_{t})$
by $(\tilde{\rho}(t),\tilde{\sigma}_{t})$ in (\ref{eq:-148}), (\ref{eq:-149})
and (\ref{eq:-150}). 

(i) For any $1\le p<\infty$,
\begin{align}
\mathbb{E}(|\zeta_{k}-\tilde{\zeta}_{k}|^{2p}) & \le c_{p,\sigma^{\ast},M_{\rho}}\Delta t^{p-1}\Big(\int_{t_{k-1}}^{t_{k}}|\rho(t)-\tilde{\rho}(t)|^{2}dt+\int_{t_{k-1}}^{t_{k}}\mathbb{E}(|\sigma_{t}-\tilde{\sigma}_{t}|^{2})dt\Big),\label{eq:-111}\\
\mathbb{E}(|\xi_{k}-\tilde{\xi}_{k}|^{2p}) & \le c_{p,\sigma^{\ast}}\Delta t^{-1}\int_{t_{k-1}}^{t_{k}}\mathbb{E}(|\sigma_{t}-\tilde{\sigma}_{t}|^{2})dt,\label{eq:-191}\\
\mathbb{E}(|\eta_{k}-\tilde{\eta}_{k}|^{2p}) & \le c_{p,\sigma^{\ast},M_{\rho},T}\Delta t^{-1/2}\Big(\int_{t_{k-1}}^{t_{k}}|\rho(t)-\tilde{\rho}(t)|^{2}dt+\int_{t_{k-1}}^{t_{k}}\mathbb{E}(|\sigma_{t}-\tilde{\sigma}_{t}|^{2})dt\Big)^{1/2},\label{eq:-192}
\end{align}
where $c_{p,\sigma^{\ast},M_{\rho}}$, $c_{p,\sigma^{\ast}}$ and
$c_{p,\sigma^{\ast},M_{\rho},T}$ are constants depending only on
the quantities in their subscripts.

(ii) Let $v_{n}$ be the realized variance defined in (\ref{eq:-2}),
or equivalently (\ref{eq:-74}). Let $\tilde{v}_{n}$ be the random
variable obtained by replacing $\eta_{k}$ by $\tilde{\eta}_{k}$
in (\ref{eq:-74}); that is,
\[
\tilde{v}_{n}=\lambda^{n}v_{0}+(1-\lambda)\sum_{k=1}^{n}\lambda^{n-k}\tilde{\eta}_{k}^{2}.
\]
Then, for any $1\le p<\infty$ and $\alpha\in\mathbb{R}$, 
\begin{equation}
\sum_{n=1}^{N}\mathbb{E}(|v_{n}^{-\alpha}-\tilde{v}_{n}^{-\alpha}|^{2p})^{1/p}\le c_{\alpha,p,\sigma_{\ast},\sigma^{\ast},M_{\rho},v_{0},T}\Delta t^{-1}\big(\Vert\rho-\tilde{\rho}\Vert_{L^{2}}+\Vert\sigma-\tilde{\sigma}\Vert_{L^{2}(\Omega;L^{2}([0,T]))}\big)^{1/(4p)},\label{eq:-151}
\end{equation}
where $c_{\alpha,p,\sigma^{\ast},M_{\rho},v_{0},T}$ is a constant
depending only on $\alpha$, $p$, $\sigma_{\ast}$, $\sigma^{\ast}$,
$M_{\rho}$, $v_{0}$ and $T$, and the norm $\Vert\cdot\Vert_{L^{2}(\Omega;L^{2}([0,T]))}$
is given by 
\[
\Vert\sigma\Vert_{L^{2}(\Omega;L^{2}([0,T]))}=\mathbb{E}\Big(\int_{0}^{T}\sigma_{t}^{2}dt\Big)^{1/2}.
\]
\end{lem}

\begin{proof}
(i) By H\"{o}lder's inequality and Minkowski's inequality,
\[
\begin{aligned}|\zeta_{k}-\tilde{\zeta}_{k}| & \le\Big(\int_{t_{k-1}}^{t_{k}}|\rho(t)-\tilde{\rho}(t)|^{2}dt\Big)^{1/2}+\Big(\int_{t_{k-1}}^{t_{k}}|\sigma_{t}^{2}-\tilde{\sigma}_{t}^{2}|^{2}dt\Big)\\
 & \le c_{\sigma^{\ast}}\Big(\int_{t_{k-1}}^{t_{k}}|\rho(t)-\tilde{\rho}(t)|^{2}dt+\int_{t_{k-1}}^{t_{k}}|\sigma_{t}-\tilde{\sigma}_{t}|^{2}dt\Big)^{1/2}.
\end{aligned}
\]
Therefore,
\[
\begin{aligned}\mathbb{E}(|\zeta_{k}-\tilde{\zeta}_{k}|^{2p}) & \le c_{p,\sigma^{\ast}}\mathbb{E}\Big[\Big(\int_{t_{k-1}}^{t_{k}}|\rho(t)-\tilde{\rho}(t)|^{2}dt+\int_{t_{k-1}}^{t_{k}}|\sigma_{t}-\tilde{\sigma}_{t}|^{2}dt\Big)^{p}\Big]\\
 & \le c_{p,\sigma^{\ast}}[M_{\rho}^{2}\Delta t+(\sigma^{\ast})^{2}\Delta t]^{p-1}\mathbb{E}\Big(\int_{t_{k-1}}^{t_{k}}|\rho(t)-\tilde{\rho}(t)|^{2}dt+\int_{t_{k-1}}^{t_{k}}|\sigma_{t}-\tilde{\sigma}_{t}|^{2}dt\Big)\\
 & \le c_{p,\sigma^{\ast},M_{\rho}}\Delta t^{p-1}\mathbb{E}\Big(\int_{t_{k-1}}^{t_{k}}|\rho(t)-\tilde{\rho}(t)|^{2}dt+\int_{t_{k-1}}^{t_{k}}|\sigma_{t}-\tilde{\sigma}_{t}|^{2}dt\Big).
\end{aligned}
\]
This proves (\ref{eq:-111}).

By the Burkholder--Davis--Gundy inequality,
\[
\begin{aligned}\mathbb{E}(|\xi_{k}-\tilde{\xi}_{k}|^{2p}) & \le c_{p}\Delta t^{-p}\mathbb{E}\Big[\Big(\int_{t_{k-1}}^{t_{k}}|\sigma_{t}-\tilde{\sigma}_{t}|^{2}dt\Big)^{p}\Big]\\
 & \le c_{p}\Delta t^{-p}[(\sigma^{\ast})^{2}\Delta t]^{p-1}\mathbb{E}\Big(\int_{t_{k-1}}^{t_{k}}|\sigma_{t}-\tilde{\sigma}_{t}|^{2}dt\Big)\\
 & \le c_{p,\sigma^{\ast}}\Delta t^{-1}\int_{t_{k-1}}^{t_{k}}\mathbb{E}(|\sigma_{t}-\tilde{\sigma}_{t}|^{2})dt.
\end{aligned}
\]
This proves (\ref{eq:-191}).

For (\ref{eq:-192}), let $f(x)=\Delta t^{-1/2}(e^{\Delta t^{1/2}x}-1)$.
Then $f^{\prime}(x)>0$ is a monotone function, and $f^{\prime}=1+\Delta t^{1/2}f$.
By Taylor's theorem,
\[
\begin{aligned}|\eta_{k}-\tilde{\eta_{k}}| & =|f(\zeta_{k}+\xi_{k})-f(\tilde{\zeta}_{k}+\tilde{\xi}_{k})|\\
 & \le[f^{\prime}(\zeta_{k}+\xi_{k})+f^{\prime}(\tilde{\zeta}_{k}+\tilde{\xi}_{k})](|\zeta_{k}-\tilde{\zeta}_{k}|+|\xi_{k}-\tilde{\xi}_{k}|)\\
 & =[2+\Delta t^{1/2}(\eta_{k}+\tilde{\eta}_{k})](|\zeta_{k}-\tilde{\zeta}_{k}|+|\xi_{k}-\tilde{\xi}_{k}|)\\
 & \le c_{T}(1+\eta_{k}+\tilde{\eta}_{k})(|\zeta_{k}-\tilde{\zeta}_{k}|+|\xi_{k}-\tilde{\xi}_{k}|).
\end{aligned}
\]
By Lemma \ref{lem:-17}, 
\[
\sup_{1\le k\le N<\infty}\mathbb{E}(|\eta_{k}|^{p})\le c_{p,\sigma^{\ast},M_{\rho},T},\quad\sup_{1\le k\le N<\infty}\mathbb{E}(|\eta_{k}|^{p})\le c_{p,\sigma^{\ast},M_{\rho},T}.
\]
Therefore, by H\"{o}lder's inequality, 
\[
\mathbb{E}(|\eta_{k}-\tilde{\eta}_{k}|^{2p})\le c_{p,\sigma^{\ast},M_{\rho},T}[\mathbb{E}(|\zeta_{k}-\tilde{\zeta}_{k}|^{4p})^{1/2}+\mathbb{E}(|\xi_{k}-\tilde{\xi}_{k}|^{4p})^{1/2}].
\]
The inequality (\ref{eq:-192}) follows readily from (\ref{eq:-111}),
(\ref{eq:-191}) and the above.

(ii) We first prove (\ref{eq:-151}) for $\alpha=-1$. Clearly, $v_{n}-\tilde{v}_{n}=(1-\lambda)\sum_{k=1}^{n}\lambda^{n-k}(\eta_{k}^{2}-\tilde{\eta}_{k}^{2})$.
By H\"{o}lder's inequality,
\[
|v_{n}-\tilde{v}_{n}|^{2p}\le c_{p}(1-\lambda)\sum_{k=1}^{n}\lambda^{n-k}|\eta_{k}+\tilde{\eta}_{k}|^{2p}|\eta_{k}-\tilde{\eta}_{k}|^{2p}.
\]
By Lemma \ref{lem:-17}, we have
\[
\sup_{k}\mathbb{E}(|\eta_{k}|^{4p})\le c_{p,\sigma^{\ast},M_{\rho},T},\quad\sup_{k}\mathbb{E}(|\tilde{\eta}_{k}|^{4p})\le c_{p,\sigma^{\ast},M_{\rho},T},
\]
for a constant $c_{p,\sigma^{\ast},M_{\rho},T}<\infty$ depending
only on $p$, $\sigma^{\ast}$, $M_{\rho}$ and $T$. Therefore, by
H\"{o}lder's inequality,
\[
\begin{aligned}\mathbb{E}(|v_{n}-\tilde{v}_{n}|^{2p})^{1/p} & \le c_{p,\sigma^{\ast},M_{\rho},T}(1-\lambda)\Big(\sum_{k=1}^{n}\lambda^{n-k}\mathbb{E}(|\eta_{k}-\tilde{\eta}_{k}|^{4p})^{1/2}\Big)^{1/p}\\
 & \le c_{p,\sigma^{\ast},M_{\rho},T}(1-\lambda)\sum_{k=1}^{n}\lambda^{(n-k)/p}\mathbb{E}(|\eta_{k}-\tilde{\eta}_{k}|^{4p})^{1/(2p)},
\end{aligned}
\]
where the last inequality is due to $p\ge1$. Consequently,
\[
\begin{aligned}\sum_{n=1}^{N}\mathbb{E}(|v_{n}-\tilde{v}_{n}|^{2p})^{1/p} & \le c_{p,\sigma^{\ast},M_{\rho},T}(1-\lambda)\sum_{n=1}^{N}\sum_{k=1}^{n}\lambda^{(n-k)/p}\mathbb{E}(|\eta_{k}-\tilde{\eta}_{k}|^{4p})^{1/(2p)}\\
 & \le c_{p,\sigma^{\ast},M_{\rho},T}(1-\lambda)\sum_{k=1}^{N}\sum_{n=k}^{N}\lambda^{(n-k)/p}\mathbb{E}(|\eta_{k}-\tilde{\eta}_{k}|^{4p})^{1/(2p)}\\
 & \le c_{\lambda,p,\sigma^{\ast},M_{\rho},T}\sum_{k=1}^{N}\mathbb{E}(|\eta_{k}-\tilde{\eta}_{k}|^{4p})^{1/(2p)}\\
 & \le c_{\lambda,p,\sigma^{\ast},M_{\rho},T}N^{1-1/(4p)}\Big(\sum_{k=1}^{N}\mathbb{E}(|\eta_{k}-\tilde{\eta}_{k}|^{4p})^{2}\Big)^{1/(4p)}.
\end{aligned}
\]
By (\ref{eq:-192}),
\[
\sum_{k=1}^{N}\mathbb{E}(|\eta_{k}-\tilde{\eta}_{k}|^{4p})^{2}\le c_{p,\sigma^{\ast},M_{\rho},T}\Delta t^{-1}\big(\Vert\rho-\tilde{\rho}\Vert_{L^{2}}^{2}+\Vert\sigma-\tilde{\sigma}\Vert_{L^{2}(\Omega;L^{2}([0,T]))}^{2}\big).
\]
Therefore,
\begin{equation}
\begin{aligned}\sum_{n=1}^{N}\mathbb{E}(|v_{n}-\tilde{v}_{n}|^{2p})^{1/p} & \le c_{\lambda,p,\sigma^{\ast},M_{\rho},T}N^{1-1/(4p)}\Delta t^{-1/(4p)}\big(\Vert\rho-\tilde{\rho}\Vert_{L^{2}}+\Vert\sigma-\tilde{\sigma}\Vert_{L^{2}(\Omega;L^{2}([0,T]))}\big)^{1/(2p)}\\
 & \le c_{\lambda,p,\sigma^{\ast},M_{\rho},T}\Delta t^{-1}\big(\Vert\rho-\tilde{\rho}\Vert_{L^{2}}+\Vert\sigma-\tilde{\sigma}\Vert_{L^{2}(\Omega;L^{2}([0,T]))}\big)^{1/(2p)}.
\end{aligned}
\label{eq:-194}
\end{equation}
This proves (\ref{eq:-151}) for $\alpha=-1$.

For general $\alpha\in\mathbb{R}$, let $h(x)=x^{-\alpha}$. By the
monotonicity of $h^{\prime}(x)=-\alpha x^{-\alpha-1}$ and Taylor's
theorem, 
\begin{equation}
|v_{n}^{-\alpha}-\tilde{v}_{n}^{-\alpha}|\le\alpha(v_{n}^{-\alpha-1}+\tilde{v}_{n}^{-\alpha-1})|v_{n}-\tilde{v}_{n}|.\label{eq:-193}
\end{equation}
By Lemma \ref{lem:-4},
\[
\sup_{n}\mathbb{E}(v_{n}^{-4p(\alpha+1)})\le c_{\alpha,p,\lambda,\sigma_{\ast},\sigma^{\ast},M_{\rho},v_{0},T},\quad\sup_{n}\mathbb{E}(\tilde{v}_{n}^{-4p(\alpha+1)})\le c_{\alpha,p,\lambda,\sigma_{\ast},\sigma^{\ast},M_{\rho},v_{0},T},
\]
where $c_{\alpha,\lambda,\sigma_{\ast},\sigma^{\ast},M_{\rho},v_{0},T}<\infty$
is a constant depending only on $\alpha$, $p$, $\lambda$, $\sigma_{\ast}$,
$\sigma^{\ast}$, $M_{\rho}$, $v_{0}$ and $T$. Therefore, by (\ref{eq:-193})
and H\"{o}lder's inequality, 
\[
\mathbb{E}(|v_{n}^{-\alpha}-\tilde{v}_{n}^{-\alpha}|^{2p})^{1/p}\le c_{\alpha,p,\lambda,\sigma_{\ast},\sigma^{\ast},M_{\rho},v_{0},T}\mathbb{E}(|v_{n}-\tilde{v}_{n}|^{4p})^{1/(2p)}.
\]
Combining the above (\ref{eq:-194}) gives (\ref{eq:-151}) for general
$\alpha\in\mathbb{R}$.
\end{proof}
To indicate the dependence on the coefficients $r(t)$, $\rho(t)$,
$\{\sigma_{t}\}_{t\ge0}$, we denote the volatility target index $I=\{I_{t}\}_{t\ge0}$
by $I(r,\rho,\sigma)=\{I(r,\rho,\sigma)_{t}\}_{t\ge0}$. We are now
in a position to prove the following continuity of $I(r,\rho,\sigma)$
with respect to $(r,\rho,\sigma)$. 
\begin{prop}
\label{prop:-6}Let $\tilde{r}(t)$, $\tilde{\rho}(t)$ be deterministic
functions satisfying (\ref{eq:-184}) and (\ref{eq:-190}) respectively,
and $\{\tilde{\sigma}_{t}\}_{t\ge0}$ be a progressively measurable
process (on the same filtered probability space $(\Omega,\{\mathcal{F}_{t}\}_{t\ge0},\mathbb{P})$)
satisfying (\ref{eq:-10}). Let $\{\tilde{I}_{t}\}_{t\ge0}=\{I(\tilde{r},\tilde{\rho},\tilde{\sigma})_{t}\}_{t\ge0}$
be the volatility target index with coefficients $\tilde{r}(t)$,
$\tilde{\rho}(t)$, $\{\tilde{\sigma}_{t}\}_{t\ge0}$. Then 
\begin{equation}
\mathbb{E}\Big(\sup_{t\in[0,T]}|\log I_{t}-\log\tilde{I}_{t}|^{2}\Big)\le c_{\lambda,\bar{\sigma},\sigma_{\ast},\sigma^{\ast},M_{\rho},M_{r},v_{0},T}\big(\Vert r-\tilde{r}\Vert_{L^{2}}+\Vert\rho-\tilde{\rho}\Vert_{L^{2}}+\Vert\sigma-\tilde{\sigma}\Vert_{L^{2}(\Omega;L^{2}([0,T]))}\big)^{1/4},\label{eq:-195}
\end{equation}
where $c_{\lambda,\bar{\sigma},\sigma_{\ast},\sigma^{\ast},M_{\rho},M_{r},v_{0},T}$
is a constant depending only on $\lambda$, $\bar{\sigma}$, $\sigma_{\ast}$,
$\sigma^{\ast}$, $M_{\rho}$, $M_{r}$, $v_{0}$ and $T$.
\end{prop}

\begin{rem}
\label{rem:-7}We should point out that the exponent $1/4$ on the
right-hand side of (\ref{eq:-195}) is not optimal. This can be seen
by a scaling argument; the optimal exponent should be $2$ (see \cite[p. 28]{LG26}
for the case when $\sigma_{t}$ is deterministic). 
\end{rem}

\begin{proof}[Proof of Proposition \ref{prop:-6}]
Let $q=\rho-r$ and $\tilde{q}=\tilde{\rho}-\tilde{r}$. For $t\in[t_{n-1},t_{n})$,
by It\^{o}'s formula, 
\[
\begin{aligned}\log I_{t}-\log\tilde{I}_{t} & =(r(t)-\tilde{r}(t))dt-\bar{\sigma}(v_{n-1}^{-1/2}q(t)-\tilde{v}_{n-1}^{-1/2}\tilde{q}(t))dt\\
 & \quad-\frac{\bar{\sigma}^{2}}{2}(v_{n-1}^{-1}\sigma_{t}^{2}-\tilde{v}_{n-1}^{-1}\tilde{\sigma}_{t}^{2})dt+\bar{\sigma}(v_{n-1}^{-1/2}\sigma_{t}-\tilde{v}_{n-1}^{-1/2}\tilde{\sigma}_{t})dW_{t},
\end{aligned}
\]
where $\tilde{v}_{n}$ is the realized variance corresponding to $(\tilde{r},\tilde{\rho},\tilde{\sigma})$.
Therefore, by Doob's maximal inequality,
\begin{equation}
\begin{aligned}\mathbb{E}\Big(\sup_{t\in[0,T]}|\log I_{t}-\log\tilde{I}_{t}|^{2}\Big)\le c_{\bar{\sigma}}(\Vert r-\tilde{r}\Vert_{L^{2}}^{2}+E_{1}+E_{2}+E_{3}),\end{aligned}
\label{eq:-196}
\end{equation}
where
\[
\begin{aligned}E_{1} & =\mathbb{E}\Big[\Big(\sum_{n=1}^{N}\int_{t_{n-1}}^{t_{n}}|v_{n-1}^{-1/2}q(t)-\tilde{v}_{n-1}^{-1/2}\tilde{q}(t)|dt\Big)^{2}\Big],\\
E_{2} & =\mathbb{E}\Big[\Big(\sum_{n=1}^{N}\int_{t_{n-1}}^{t_{n}}|v_{n-1}^{-1}\sigma_{t}^{2}-\tilde{v}_{n-1}^{-1}\tilde{\sigma}_{t}^{2}|dt\Big)^{2}\Big],\\
E_{3} & =\mathbb{E}\Big[\Big(\sum_{n=1}^{N}\int_{t_{n-1}}^{t_{n}}|v_{n-1}^{-1/2}\sigma_{t}-\tilde{v}_{n-1}^{-1/2}\tilde{\sigma}_{t}|dW_{t}\Big)^{2}\Big].
\end{aligned}
\]

For $E_{1}$, we have
\begin{equation}
E_{1}\le2\mathbb{E}\Big[\Big(\sum_{n=1}^{N}\int_{t_{n-1}}^{t_{n}}|(v_{n-1}^{-1/2}-\tilde{v}_{n-1}^{-1/2})q(t)|dt\Big)^{2}\Big]+2\mathbb{E}\Big[\Big(\sum_{n=1}^{N}\int_{t_{n-1}}^{t_{n}}|\tilde{v}_{n-1}^{-1/2}(q(t)-\tilde{q}(t))|dt\Big)^{2}\Big].\label{eq:-199}
\end{equation}
For the first term on the right-hand side, we have
\begin{equation}
\begin{aligned}\mathbb{E}\Big[\Big(\sum_{n=1}^{N}\int_{t_{n-1}}^{t_{n}}|(v_{n-1}^{-1/2}-\tilde{v}_{n-1}^{-1/2})q(t)|dt\Big)^{2}\Big] & \le c_{M_{\rho},M_{r}}\Delta t^{2}\mathbb{E}\Big[\Big(\sum_{n=1}^{N}|v_{n-1}^{-1/2}-\tilde{v}_{n-1}^{-1/2}|\Big)^{2}\Big]\\
 & \le c_{M_{\rho},M_{r}}\Delta t^{2}\mathbb{E}\Big(N\sum_{n=1}^{N}|v_{n-1}^{-1/2}-\tilde{v}_{n-1}^{-1/2}|^{2}\Big)\\
 & \le c_{M_{\rho},M_{r},T}\Delta t\sum_{n=1}^{N}\mathbb{E}(|v_{n-1}^{-1/2}-\tilde{v}_{n-1}^{-1/2}|^{2}).
\end{aligned}
\label{eq:-198}
\end{equation}
By (\ref{eq:-151}) with $p=1$ and $\alpha=1/2$, we see that
\[
\mathbb{E}\Big[\Big(\sum_{n=1}^{N}\int_{t_{n-1}}^{t_{n}}|(v_{n-1}^{-1/2}-\tilde{v}_{n-1}^{-1/2})q(t)|dt\Big)^{2}\Big]\le c_{\sigma_{\ast},\sigma^{\ast},M_{\rho},M_{r},v_{0},T}\big(\Vert\rho-\tilde{\rho}\Vert_{L^{2}}+\Vert\sigma-\tilde{\sigma}\Vert_{L^{2}(\Omega;L^{2}([0,T]))}\big)^{1/4}.
\]
For the second term on the right-hand side of (\ref{eq:-199}), by
H\"{o}lder's inequality, we have
\begin{equation}
\begin{aligned}\mathbb{E}\Big[\Big(\sum_{n=1}^{N}\int_{t_{n-1}}^{t_{n}}|\tilde{v}_{n-1}^{-1/2}(q(t)-\tilde{q}(t))|dt\Big)^{2}\Big] & \le\mathbb{E}\Big[\Big(\sum_{n=1}^{N}\tilde{v}_{n-1}^{-1}\Big)\Big(\sum_{n=1}^{N}\Big(\int_{t_{n-1}}^{t_{n}}|q(t)-\tilde{q}(t)|dt\Big)^{2}\Big)\Big]\\
 & \le\mathbb{E}\Big[\Big(\sum_{n=1}^{N}\tilde{v}_{n-1}^{-1}\Big)\Big(\Delta t\int_{0}^{T}|q(t)-\tilde{q}(t)|^{2}dt\Big)\Big]\\
 & =\Delta t\sum_{n=1}^{N}\mathbb{E}\Big(\tilde{v}_{n-1}^{-1}\int_{0}^{T}|q(t)-\tilde{q}(t)|^{2}dt\Big).
\end{aligned}
\label{eq:-201}
\end{equation}
By Lemma \ref{lem:-4}, $\sup_{n}\mathbb{E}(\tilde{v}_{n}^{-1})\le c_{\lambda,\sigma_{\ast},\sigma^{\ast},M_{\rho},v_{0},T}<\infty$,
and we deduce that
\[
\mathbb{E}\Big[\Big(\sum_{n=1}^{N}\int_{t_{n-1}}^{t_{n}}|\tilde{v}_{n-1}^{-1/2}(q(t)-\tilde{q}(t))|dt\Big)^{2}\Big]\le c_{\lambda,\sigma_{\ast},\sigma^{\ast},M_{\rho},v_{0},T}\Vert q-\tilde{q}\Vert_{L^{2}}^{2}\le c_{\lambda,\sigma_{\ast},\sigma^{\ast},M_{\rho},v_{0},T}\Vert q-\tilde{q}\Vert_{L^{2}}^{1/4},
\]
where the last inequality is due to $\Vert q-\tilde{q}\Vert_{L^{2}}\le2T^{1/2}(M_{\rho}+M_{r})$.
Therefore, 
\begin{equation}
E_{1}\le c_{\sigma_{\ast},\sigma^{\ast},M_{\rho},M_{r},v_{0},T}\big(\Vert\rho-\tilde{\rho}\Vert_{L^{2}}+\Vert\sigma-\tilde{\sigma}\Vert_{L^{2}(\Omega;L^{2}([0,T]))}\big)^{1/4}.\label{eq:-197}
\end{equation}

The expectation $E_{2}$ can be estimated similarly. We have
\begin{equation}
E_{2}\le2\mathbb{E}\Big[\Big(\sum_{n=1}^{N}\int_{t_{n-1}}^{t_{n}}|v_{n-1}^{-1}-\tilde{v}_{n-1}^{-1}|\sigma_{t}^{2}dt\Big)^{2}\Big]+2\mathbb{E}\Big[\Big(\sum_{n=1}^{N}\int_{t_{n-1}}^{t_{n}}\tilde{v}_{n-1}^{-1}|\sigma_{t}^{2}-\tilde{\sigma}_{t}^{2}|dt\Big)^{2}\Big].\label{eq:-200}
\end{equation}
For the first term, similar to (\ref{eq:-198}), we have
\[
\mathbb{E}\Big[\Big(\sum_{n=1}^{N}\int_{t_{n-1}}^{t_{n}}|v_{n-1}^{-1}-\tilde{v}_{n-1}^{-1}|\sigma_{t}^{2}dt\Big)^{2}\Big]\le c_{\sigma^{\ast},T}\Delta t\sum_{n=1}^{N}\mathbb{E}(|v_{n-1}^{-1}-\tilde{v}_{n-1}^{-1}|^{2}).
\]
By (\ref{eq:-151}) with $p=1$ and $\alpha=1$, we obtain that
\[
\mathbb{E}\Big[\Big(\sum_{n=1}^{N}\int_{t_{n-1}}^{t_{n}}|v_{n-1}^{-1}-\tilde{v}_{n-1}^{-1}|\sigma_{t}^{2}dt\Big)^{2}\Big]\le c_{\sigma_{\ast},\sigma^{\ast},M_{\rho},v_{0},T}\big(\Vert\rho-\tilde{\rho}\Vert_{L^{2}}+\Vert\sigma-\tilde{\sigma}\Vert_{L^{2}(\Omega;L^{2}([0,T]))}\big)^{1/4}.
\]
For the second term on the right-hand side of (\ref{eq:-199}), similar
to (\ref{eq:-201}), we have
\[
\begin{aligned}\mathbb{E}\Big[\Big(\sum_{n=1}^{N}\int_{t_{n-1}}^{t_{n}}\tilde{v}_{n-1}^{-1}|\sigma_{t}^{2}-\tilde{\sigma}_{t}^{2}|dt\Big)^{2}\Big] & \le\Delta t\sum_{n=1}^{N}\mathbb{E}\Big(\tilde{v}_{n-1}^{-2}\int_{0}^{T}|\sigma_{t}^{2}-\tilde{\sigma}_{t}^{2}|^{2}dt\Big).\end{aligned}
\]
Note that $|\sigma_{t}^{2}-\tilde{\sigma}_{t}^{2}|\le2(\sigma^{\ast})^{3/2}|\sigma_{t}-\tilde{\sigma}_{t}|^{1/2}$.
Therefore, 
\[
\begin{aligned}\mathbb{E}\Big[\Big(\sum_{n=1}^{N}\int_{t_{n-1}}^{t_{n}}\tilde{v}_{n-1}^{-1}|\sigma_{t}^{2}-\tilde{\sigma}_{t}^{2}|dt\Big)^{2}\Big] & \le c_{\sigma^{\ast}}\Delta t\sum_{n=1}^{N}\mathbb{E}\Big(\tilde{v}_{n-1}^{-2}\int_{0}^{T}|\sigma_{t}-\tilde{\sigma}_{t}|^{1/2}dt\Big)\\
 & \le c_{\sigma^{\ast},T}\Delta t\sum_{n=1}^{N}\mathbb{E}\Big[\tilde{v}_{n-1}^{-2}\Big(\int_{0}^{T}|\sigma_{t}-\tilde{\sigma}_{t}|^{2}dt\Big)^{1/2}\Big]\\
 & \le c_{\sigma^{\ast}}\Delta t\sum_{n=1}^{N}\mathbb{E}(\tilde{v}_{n-1}^{-4})^{1/2}\Vert\sigma-\tilde{\sigma}\Vert_{L^{2}(\Omega;L^{2}([0,T]))}\\
 & \le c_{\sigma^{\ast}}\Delta t\sum_{n=1}^{N}\mathbb{E}(\tilde{v}_{n-1}^{-4})^{1/2}\Vert\sigma-\tilde{\sigma}\Vert_{L^{2}(\Omega;L^{2}([0,T]))}^{1/4},
\end{aligned}
\]
where the last inequality is due to the boundedness of $\sigma_{t}$
and $\tilde{\sigma}_{t}$. By Lemma \ref{lem:-4}, $\sup_{n}\mathbb{E}(\tilde{v}_{n}^{-4})\le c_{\lambda,\sigma_{\ast},\sigma^{\ast},M_{\rho},v_{0},T}<\infty$,
and we deduce that
\[
\begin{aligned}\mathbb{E}\Big[\Big(\sum_{n=1}^{N}\int_{t_{n-1}}^{t_{n}}\tilde{v}_{n-1}^{-1}|\sigma_{t}^{2}-\tilde{\sigma}_{t}^{2}|dt\Big)^{2}\Big] & \le c_{\lambda,\sigma_{\ast},\sigma^{\ast},M_{\rho},v_{0},T}\Vert\sigma-\tilde{\sigma}\Vert_{L^{2}(\Omega;L^{2}([0,T]))}^{1/4}.\end{aligned}
\]
This proves that 
\begin{equation}
E_{2}\le c_{\sigma_{\ast},\sigma^{\ast},M_{\rho},v_{0},T}\big(\Vert\rho-\tilde{\rho}\Vert_{L^{2}}+\Vert\sigma-\tilde{\sigma}\Vert_{L^{2}(\Omega;L^{2}([0,T]))}\big)^{1/4}.\label{eq:-202}
\end{equation}

For $E_{3}$, by the Burkholder--Davis--Gundy inequality,
\[
\begin{aligned}E_{3} & \le\mathbb{E}\Big(\sum_{n=1}^{N}\int_{t_{n-1}}^{t_{n}}|v_{n-1}^{-1/2}\sigma_{t}-\tilde{v}_{n-1}^{-1/2}\tilde{\sigma}_{t}|^{2}dt\Big)\\
 & \le2\sum_{n=1}^{N}\int_{t_{n-1}}^{t_{n}}\mathbb{E}(|v_{n-1}^{-1/2}-\tilde{v}_{n-1}^{-1/2}|^{2}\sigma_{t}^{2})dt+2\sum_{n=1}^{N}\int_{t_{n-1}}^{t_{n}}\mathbb{E}(\tilde{v}_{n-1}^{-1}|\sigma_{t}-\tilde{\sigma}_{t}|^{2})dt\\
 & \le c_{\sigma^{\ast}}\Delta t\sum_{n=1}^{N}\mathbb{E}(|v_{n-1}^{-1/2}-\tilde{v}_{n-1}^{-1/2}|^{2})+2\sum_{n=1}^{N}\mathbb{E}(\tilde{v}_{n-1}^{-2})^{1/2}\int_{t_{n-1}}^{t_{n}}\mathbb{E}(|\sigma_{t}-\tilde{\sigma}_{t}|^{4})^{1/2}dt.
\end{aligned}
\]
For the first term in the last inequality, by (\ref{eq:-151}) with
$p=1$ and $\alpha=1$, we have
\[
\Delta t\sum_{n=1}^{N}\mathbb{E}(|v_{n-1}^{-1/2}-\tilde{v}_{n-1}^{-1/2}|^{2})\le c_{\sigma_{\ast},\sigma^{\ast},M_{\rho},v_{0},T}\big(\Vert\rho-\tilde{\rho}\Vert_{L^{2}}+\Vert\sigma-\tilde{\sigma}\Vert_{L^{2}(\Omega;L^{2}([0,T]))}\big)^{1/4}.
\]
For the second term, by Lemma \ref{lem:-4}, $\sup_{n}\mathbb{E}(\tilde{v}_{n}^{-2})\le c_{\lambda,\sigma_{\ast},\sigma^{\ast},M_{\rho},v_{0},T}<\infty$.
Consequently,
\[
\begin{aligned}\sum_{n=1}^{N}\mathbb{E}(\tilde{v}_{n-1}^{-2})^{1/2}\int_{t_{n-1}}^{t_{n}}\mathbb{E}(|\sigma_{t}-\tilde{\sigma}_{t}|^{4})^{1/2}dt & \le c_{\lambda,\sigma_{\ast},\sigma^{\ast},M_{\rho},v_{0},T}\int_{0}^{T}\mathbb{E}(|\sigma_{t}-\tilde{\sigma}_{t}|^{4})^{1/2}dt\\
 & \le c_{\lambda,\sigma_{\ast},\sigma^{\ast},M_{\rho},v_{0},T}\Big(\int_{0}^{T}\mathbb{E}(|\sigma_{t}-\tilde{\sigma}_{t}|^{4})dt\Big)^{1/2}\\
 & \le c_{\lambda,\sigma_{\ast},\sigma^{\ast},M_{\rho},v_{0},T}\Vert\sigma-\tilde{\sigma}\Vert_{L^{2}(\Omega;L^{2}([0,T]))}.
\end{aligned}
\]
Therefore,
\begin{equation}
E_{3}\le c_{\sigma_{\ast},\sigma^{\ast},M_{\rho},v_{0},T}\big(\Vert\rho-\tilde{\rho}\Vert_{L^{2}}+\Vert\sigma-\tilde{\sigma}\Vert_{L^{2}(\Omega;L^{2}([0,T]))}\big)^{1/4}.\label{eq:-203}
\end{equation}

Combining (\ref{eq:-199}), (\ref{eq:-202}) and (\ref{eq:-203})
completes the proof.
\end{proof}
Similarly, we have the continuity of the limiting process $\{X_{t}\}_{t}$
with respect to the coefficients.
\begin{prop}
\label{prop:-7}Let $\{X_{t}\}_{t\ge0}$ be the continuous process
defined by (\ref{eq:}). Suppose that $\tilde{r}(t)$ and $\tilde{\rho}(t)$
are deterministic functions satisfying (\ref{eq:-184}) and (\ref{eq:-190})
respectively, and that $\{\tilde{\sigma}_{t}\}_{t\ge0}$ is a progressively
measurable process (on the same filtered probability space $(\Omega,\{\mathcal{F}_{t}\}_{t\ge0},\mathbb{P})$)
satisfying (\ref{eq:-10}). Let $\{\tilde{X}_{t}\}_{t\ge0}$ be the
process obtained by replacing $(\rho(t),r(t),\sigma_{t})$ by $(\tilde{\rho}(t),\tilde{r}(t),\tilde{\sigma}_{t})$
in (\ref{eq:}). Then 
\[
\mathbb{E}\Big(\sup_{t\in[0,T]}|\log X_{t}-\log\tilde{X}_{t}|^{2}\Big)\le c_{\bar{\sigma},\lambda,M_{\rho},M_{r},\sigma_{\ast},T}\big(\Vert r-\tilde{r}\Vert_{L^{2}}+\Vert\rho-\tilde{\rho}\Vert_{L^{2}}+\Vert\sigma-\tilde{\sigma}\Vert_{L^{2}(\Omega;L^{2}([0,T]))}\big)^{2},
\]
where $c_{\bar{\sigma},\lambda,M_{\rho},M_{r},\sigma_{\ast},T}$ is
a constant depending only on $\bar{\sigma}$, $\lambda$, $M_{\rho}$,
$M_{r}$, $\sigma_{\ast}$ and $T$.
\end{prop}

\begin{proof}
Let $q(t)=\rho(t)-r(t)$ and $\tilde{q}(t)=\tilde{\rho}(t)-\tilde{r}(t)$.
By It\^{o}'s formula,
\[
d(\log X_{t}-\log\tilde{X}_{t})=\big(r(t)-\tilde{r}(t)+(q(t)\sigma_{t}^{-1}-\tilde{q}(t)\tilde{\sigma}_{t}^{-1})\bar{\sigma}U(\lambda)\big)dt.
\]
Therefore,
\[
\begin{aligned}\mathbb{E}\Big(\sup_{t\in[0,T]}|\log X_{t}-\log\tilde{X}_{t}|^{2}\Big) & \le c_{\bar{\sigma},\lambda}\Big[\Vert r-\tilde{r}\Vert_{L^{2}}^{2}+\mathbb{E}\Big(\Big(\int_{0}^{T}|q(t)\sigma_{t}^{-1}-\tilde{q}(t)\tilde{\sigma}_{t}^{-1}|dt\Big)^{2}\Big)\Big]\\
 & \le c_{\bar{\sigma},\lambda,T}\Big[\Vert r-\tilde{r}\Vert_{L^{2}}^{2}+\int_{0}^{T}\mathbb{E}(|q(t)\sigma_{t}^{-1}-\tilde{q}(t)\tilde{\sigma}_{t}^{-1}|^{2})dt\Big].
\end{aligned}
\]
By $|q|\le M_{\rho}+M_{r}$, $|\tilde{q}|\le M_{\rho}+M_{r}$ and
$|\sigma_{t}^{-1}-\tilde{\sigma}_{t}^{-1}|\le\sigma_{\ast}^{2}|\sigma_{t}-\tilde{\sigma}_{t}|$,
we have 
\[
\begin{aligned}\mathbb{E}(|q(t)\sigma_{t}^{-1}-\tilde{q}(t)\tilde{\sigma}_{t}^{-1}|^{2}) & \le2\mathbb{E}(|q(t)-\tilde{q}(t)|^{2}\sigma_{t}^{-2})+2\mathbb{E}(\tilde{q}(t)^{2}|\sigma_{t}^{-1}-\tilde{\sigma}_{t}^{-1}|^{2})\\
 & \le c_{M_{\rho},M_{r},\sigma_{\ast}}[|q(t)-\tilde{q}(t)|^{2}+\mathbb{E}(|\sigma_{t}^{-1}-\tilde{\sigma}_{t}^{-1}|^{2})]\\
 & \le c_{M_{\rho},M_{r},\sigma_{\ast}}[|q(t)-\tilde{q}(t)|^{2}+\mathbb{E}(|\sigma_{t}-\tilde{\sigma}_{t}|^{2})].
\end{aligned}
\]
Therefore,
\[
\mathbb{E}\Big(\sup_{t\in[0,T]}|\log X_{t}-\log\tilde{X}_{t}|^{2}\Big)\le c_{\bar{\sigma},\lambda,M_{\rho},M_{r},\sigma_{\ast},T}\big(\Vert r-\tilde{r}\Vert_{L^{2}}^{2}+\Vert q-\tilde{q}\Vert_{L^{2}}^{2}+\Vert\sigma-\tilde{\sigma}\Vert_{L^{2}(\Omega;L^{2}([0,T]))}\big).
\]
This completes the proof.
\end{proof}

\subsection{\label{subsec:-5}Convergence of Terminal Distribution under Constant
Coefficients}

In this subsection, we prove the convergence of the terminal distribution
$(I_{T},W_{T},\tilde{W}_{T})$ under the assumption that $\rho(t)=\rho$,
$r(t)=r$ and $\sigma_{t}=\sigma$ are constants. To obtain the joint
limiting distribution of $(I_{T},W_{T},\tilde{W}_{T})$ as well as
its convergence rate in terms of cumulative distribution functions,
we will need an extension of the limit theorems in \cite{LG26}. The
following Barry--Esseen type bound for one-dimensional martingale
difference sequences is obtained in \cite[Theorem 1]{Haeusler88}\footnote{To the best of our knowledge, there is no existing results on Barry--Esseen
type bounds for multi-dimensional martingale difference sequences
which is applicable to the sequences being studied in this paper. }. 
\begin{thm}
\label{thm:-2}Let $\{\mathcal{G}_{k}\}_{0\le k\le n}$ be a filtration
on a probability space, and $\{X_{k}\}_{1\le k\le n}$ be a sequence
of adapted square integrable random variables such that $\mathbb{E}(X_{k}|\mathcal{G}_{k-1})=0$
for $1\le n$. Let 
\begin{align*}
L_{n,r} & =\sum_{k=1}^{n}\mathbb{E}(|X_{k}|^{2+2r}),\\
N_{n,r} & =\mathbb{E}\Big(\Big|\sum_{k=1}^{n}\mathbb{E}(X_{k}^{2}|\mathcal{G}_{k-1})-1\Big|^{1+r}\Big).
\end{align*}
Then
\[
\sup_{x\in\mathbb{R}}|\mathbb{P}(S_{n}\le x)-\Phi(x)|\le c_{r}(L_{n,r}+N_{n,r})^{1/(3+2r)},
\]
where $c_{r}$ is a constant depending only on $r$, and $\Phi(x)$
is the cumulative distribution function of $\mathcal{N}(0,1)$.
\end{thm}

We start with the following lemma, which will be needed in the proof
of Theorem \ref{thm:-1} and Theorem \ref{thm:-3}.
\begin{lem}
\label{lem:-19}Let $\{\lambda_{n}\}_{n\ge0}$ be a sequence of positive
numbers satisfying 
\begin{equation}
\sum_{n=0}^{\infty}\lambda_{n}=1,\label{eq:-204}
\end{equation}
and 
\begin{equation}
\lambda_{n}\le K_{\lambda}\gamma^{n},\label{eq:-179}
\end{equation}
for some constants $\gamma\in(0,1)$ and $K_{\lambda}>0$. Let $\phi(t)$
be a continuous function on $[0,\infty)$ satisfying $\phi(0)=1$,
$0<\phi(t)\le1$ and 
\begin{equation}
1-t^{\alpha}\le\phi(t)\le K_{\phi}(1+t)^{-\beta},\label{eq:-44}
\end{equation}
for some constants $0<\alpha\le1$, $\beta\ge0$, and $K_{\phi}>0$.
Suppose that $\mu$ is a positive Borel measure on $[0,\infty)$ such
that
\begin{equation}
\int_{0}^{\infty}(1+t)^{-m}\mu(dt)\le K_{\mu},\label{eq:-39}
\end{equation}
for some $m\ge0$ and $K_{\mu}<\infty$. Then, 
\begin{equation}
\int_{0}^{\infty}\prod_{k=0}^{\infty}\phi(\lambda_{k}t)\mu(dt)\le c_{\lambda,m,K_{\mu},K_{\phi},\alpha,\beta}<\infty,\label{eq:-178}
\end{equation}
and
\begin{equation}
\Big|\int_{0}^{\infty}e^{-r_{n}t}\prod_{k=0}^{n-1}\phi(\lambda_{k}t)\mu(dt)-\int_{0}^{\infty}\prod_{k=0}^{\infty}\phi(\lambda_{k}t)\mu(dt)\Big|\le c_{\lambda,m,K_{\mu},K_{\phi},\alpha,\beta}\gamma^{\alpha n},\label{eq:-188}
\end{equation}
for all $n\ge0$, where $r_{n}=\sum_{k=n}^{\infty}\lambda_{k}$, and
$c_{\lambda,m,K_{\mu},K_{\phi},\alpha,\beta}$ is a constant depending
only on $\{\lambda_{n}\}_{n}$, $m$, $K_{\mu}$, $K_{\phi}$, $\alpha$
and $\beta$. 
\end{lem}

\begin{proof}
Throughout the proof, we will denote by $c_{q_{1},\dots,q_{n}}$ a
constant depending only on $q_{1},\dots,q_{n}$, and its value may
vary at different appearances. Let
\[
P_{n}(t)=\prod_{k=0}^{n-1}\phi(\lambda_{k}t),\quad P_{\infty}(t)=\prod_{k=0}^{\infty}\phi(\lambda_{k}t).
\]

To see (\ref{eq:-178}), let $n_{0}=\lfloor m/\beta\rfloor+1$. Then
$P_{\infty}(t)\le P_{n_{0}}(t)\le c_{\lambda,K_{\phi},\beta}(1+t)^{-m}$,
which implies (\ref{eq:-178}) readily. To see (\ref{eq:-188}), let
\[
\tilde{P}_{n}(t)=P_{n}(t)^{-1}P_{\infty}(t)=\prod_{k=0}^{\infty}\phi(\lambda_{k+n}t).
\]
Note that 
\begin{equation}
\tilde{P}_{n}(t)=\prod_{k=0}^{\infty}\phi(\lambda_{k+n}t)=\phi(\lambda_{n}t)\prod_{k=1}^{\infty}\phi(\lambda_{k+n}t)=\phi(\lambda_{n}t)\tilde{P}_{n+1}(t).\label{eq:-182}
\end{equation}
By (\ref{eq:-182}) and $0<\phi\le1$, we see that $\tilde{P}_{n}(t)$
is an increasing sequence and $\tilde{P}_{n}(t)\le1$, which implies
the existence of $\tilde{P}_{\infty}(t)=\lim_{n\to\infty}\tilde{P}_{n}(t)$.
By the continuity of $\phi$, setting $n\to\infty$ in (\ref{eq:-182})
gives $\tilde{P}_{\infty}(t)=1$. Consequently, 
\[
P_{n}(t)-P_{\infty}(t)=P_{n}(t)(1-\tilde{P}_{n}(t))=P_{n}(t)\sum_{k=n}^{\infty}(\tilde{P}_{k+1}(t)-\tilde{P}_{k}(t)).
\]
By (\ref{eq:-44}) and (\ref{eq:-182}),
\[
\begin{aligned}P_{n}(t)-P_{\infty}(t) & =P_{n}(t)\sum_{k=n}^{\infty}(1-\phi(\lambda_{k}t))\tilde{P}_{k+1}(t)\le P_{n}(t)\sum_{k=n}^{\infty}(1-\phi(\lambda_{k}t))\le t^{\alpha}P_{n}(t)\sum_{k=n}^{\infty}\lambda_{k}^{\alpha}.\end{aligned}
\]
By the above and $\lambda_{n}\le K_{\lambda}\gamma^{n}$, we deduce
that
\begin{equation}
P_{n}(t)-P_{\infty}(t)\le K_{\lambda}^{\alpha}(1-\gamma^{\alpha})^{-1}\gamma^{\alpha n}t^{\alpha}P_{n}(t).\label{eq:-183}
\end{equation}
Let $n_{\alpha}=\lfloor(\alpha+m)/\beta\rfloor+1$. Then $\beta n_{0}-\alpha>m$.
For any $n\ge n_{0}$, by (\ref{eq:-44}), $t^{\alpha}P_{n}(t)\le t^{\alpha}P_{n_{\alpha}}(t)\le c_{\lambda,K_{\phi},\beta}(1+t)^{-m}$.
Therefore, by (\ref{eq:-183}) and the above,
\[
\begin{aligned} & \int_{0}^{\infty}e^{-r_{n}t}[P_{n}(t)-P_{\infty}(t)]\mu(dt)\\
 & \le K_{\lambda}^{\alpha}(1-\gamma^{\alpha})^{-1}\lambda_{n_{\alpha}}^{-\beta n_{\alpha}}K_{\phi}^{n_{\alpha}}\gamma^{\alpha n}\int_{0}^{\infty}(1+t)^{-m}\mu(dt)\le c_{\lambda,m,K_{\mu},K_{\phi},\alpha,\beta}\gamma^{\alpha n}.
\end{aligned}
\]
where, in the last inequality, the dependence on $K_{\lambda}$, $(1-\gamma^{\alpha})^{-1}$
and $\lambda_{n_{0}}$ are expressed as dependence on $\{\lambda_{n}\}_{n}$.
Let $A=\gamma^{-\alpha n_{\alpha}}\sup_{t\ge0}e^{-r_{n_{\alpha}}t}(1+t)^{m}<\infty$.
Then, for $n<n_{\alpha}$,
\[
\int_{0}^{\infty}e^{-r_{n}t}[P_{n}(t)-P_{\infty}(t)]\mu(dt)\le\int_{0}^{\infty}e^{-r_{n_{0}}t}\mu(dt)\le A\gamma^{\alpha n}\int_{0}^{\infty}(1+t)^{-m}\mu(dt).
\]
Therefore,
\begin{equation}
\int_{0}^{\infty}e^{-r_{n}t}[P_{n}(t)-P_{\infty}(t)]\mu(dt)\le c_{\lambda,m,K_{\mu},K_{\phi},\alpha,\beta}\gamma^{\alpha n},\quad n\ge0.\label{eq:-187}
\end{equation}
Let $n_{1}=\lfloor(1+m)/\beta\rfloor+1$. Then $tP_{\infty}(t)\le tP_{n_{1}}(t)\le c_{\lambda,K_{\phi},\beta}(1+t)^{-m}$.
Therefore, by $r_{k}=\sum_{k=n}^{\infty}\lambda_{k}\le K_{\lambda}(1-\gamma)^{-1}\gamma^{n}$,
we deduce that
\begin{equation}
\begin{aligned}\int_{0}^{\infty}(1 & -e^{-r_{n}t})P_{\infty}(t)\mu(dt)\le r_{n}\int_{0}^{\infty}tP_{\infty}(t)\mu(dt)\le c_{\lambda,m,K_{\mu},K_{\phi},\beta}\gamma^{n}.\end{aligned}
\label{eq:-185}
\end{equation}
Combining (\ref{eq:-187}) and (\ref{eq:-185}) gives (\ref{eq:-188}). 
\end{proof}
Recall that a function $F(t)$ on $[0,\infty)$ is called \emph{completely
monotone} if $F\in C^{\infty}$ and $(-1)^{n}F^{(n)}(t)\ge0$ for
all $n\in\mathbb{N}$ and $t\ge0$. By Bernstein's theorem (see \cite[Theorem 1.4]{SSV12}),
$F(t)$ is completely monotone if and only if there exists a positive
Borel measure $\mu_{F}$, called the \emph{representing measure} of
$F$, such that
\begin{equation}
F(t)=\int_{0}^{\infty}e^{-ts}\mu_{F}(ds),\quad t\ge0.\label{eq:-38}
\end{equation}
We will need a slight extension of \cite[Theorem 2.2]{LG26}, where
the measure $f(t)dt$ in \cite[Theorem 2.2]{LG26} is replaced by
the representing measure $\mu_{F}(dt)$.
\begin{thm}
\label{thm:-1}Let $\{\lambda_{n}\}_{n\ge0}$ be a sequence of positive
numbers satisfying (\ref{eq:-204}) and (\ref{eq:-179}). Suppose
that $\{\xi_{n}\}_{n\ge1}$ is a sequence of i.i.d. positive random
variables, and let
\begin{equation}
\omega_{n}=r_{n}\omega_{0}+\sum_{k=1}^{n}\lambda_{n-k}\xi_{k},\quad n\ge1,\label{eq:-37}
\end{equation}
where $\omega_{0}>0$ is a constant and $r_{n}=\sum_{k=n}^{\infty}\lambda_{k}$.
Suppose that the Laplace transform $\phi(t)=\mathbb{E}(e^{-t\xi_{1}})$
of $\xi_{1}$ satisfies the property (\ref{eq:-44}). Let $F(t)$
be a completely monotone function on $[0,\infty)$ with representing
measure $\mu_{F}$ satisfying (\ref{eq:-39}). Then
\begin{equation}
\lim_{n\to\infty}\frac{1}{n}\sum_{k=0}^{n-1}F(\omega_{k})=\int_{0}^{\infty}\prod_{k=0}^{\infty}\phi(\lambda_{k}t)\mu_{F}(dt)\quad\mathrm{a.s.},\label{eq:-42}
\end{equation}
and
\begin{equation}
\mathbb{E}\Big[\Big(\frac{1}{n}\sum_{k=0}^{n-1}F(\omega_{k})-\int_{0}^{\infty}\prod_{k=0}^{\infty}\phi(\lambda_{k}t)\mu_{F}(dt)\Big)^{2}\Big]\le c_{\lambda,m,K_{\mu},K_{\phi},\alpha,\beta,\omega_{0}}n^{-1},\label{eq:-175}
\end{equation}
where $c_{\lambda,m,K_{\mu},K_{\phi},\alpha,\beta,\omega_{0}}$ is
a constant depending only on $\{\lambda_{n}\}_{n}$, $m$, $K_{\mu}$,
$K_{\phi}$, $\alpha$, $\beta$ and $\omega_{0}$. 
\end{thm}

\begin{proof}
The proof of (\ref{eq:-42}) is the same as that of \cite[Theorem 2.2]{LG26}.
To see (\ref{eq:-175}), let $\psi_{n}=n^{-1}\sum_{k=0}^{n-1}F(\omega_{k})$.
It is shown in \cite[p. 9]{LG26} that
\begin{equation}
\text{Var}(\psi_{n})\le c_{\lambda,m,K_{\mu},K_{\phi},\alpha,\beta,\omega_{0}}n^{-1},\label{eq:-177}
\end{equation}
By $F(x)=\int_{0}^{\infty}e^{-tx}\mu_{F}(dt)$ and Fubini's theorem,
\begin{equation}
\mathbb{E}[F(\omega_{k})]=\mathbb{E}\Big(\int_{0}^{\infty}e^{-\omega_{k}t}\mu_{F}(dt)\Big)=\int_{0}^{\infty}e^{-r_{k}\omega_{0}t}\prod_{j=0}^{k-1}\phi(\lambda_{j}t)\mu_{F}(dt).\label{eq:-181}
\end{equation}
Therefore, by (\ref{eq:-188}), we have
\[
\begin{aligned}\Big|\mathbb{E}[F(\omega_{k})]-\int_{0}^{\infty}\prod_{j=0}^{\infty}\phi(\lambda_{j}t)\mu_{F}(dt)\Big| & \le c_{\lambda,m,K_{\mu},K_{\phi},\alpha,\beta,\omega_{0}}\gamma^{\alpha k}.\end{aligned}
\]
Consequently,
\[
\Big|\mathbb{E}(\psi_{n})-\int_{0}^{\infty}\prod_{j=0}^{\infty}\phi(\lambda_{j}t)\mu_{F}(dt)\Big|\le c_{\lambda,m,K_{\mu},K_{\phi},\alpha,\beta,\omega_{0}}n^{-1}\sum_{k=1}^{n}\gamma^{\alpha k}\le c_{\lambda,m,K_{\mu},K_{\phi},\alpha,\beta,\omega_{0}}n^{-1},
\]
which, together with (\ref{eq:-177}), proves (\ref{eq:-175}).
\end{proof}
\begin{rem}
\label{rem:-6}Clearly, if $F(t)$ and $G(t)$ are completely monotone
functions satisfying (\ref{eq:-39}), then $F(t)G(t)$ and $F(t)+G(t)$
are also completely monotone functions satisfying (\ref{eq:-39})
with possibly different $m$ (see Lemma \ref{lem:-16}).
\end{rem}

\begin{thm}
\label{thm:-3} Let $\{\lambda_{n}\}_{n\ge0}$, $\{\xi_{n}\}_{n\ge1}$,
$\{\omega_{n}\}_{n\ge0}$ be the same as in Theorem \ref{thm:-1}.
Suppose that the Laplace transform $\phi(t)=\mathbb{E}(e^{-t\xi_{1}})$
of $\xi_{1}$ satisfies the property (\ref{eq:-44}). Let $F(t)$
be a completely monotone function on $[0,\infty)$ with representing
measure $\mu_{F}$ such that satisfying (\ref{eq:-39}).Let $\{\eta_{n}\}_{n\ge1}$
be a sequence of random variables such that
\begin{equation}
\mathbb{E}(\eta_{n}|\mathcal{G}_{n-1})=0,\quad n\ge1\quad\mathrm{a.s.}\label{eq:-40}
\end{equation}
and 
\begin{equation}
\mathbb{E}(\eta_{n}^{2}|\mathcal{G}_{n-1})=1,\quad n\ge1\quad\mathrm{a.s.},\label{eq:-41}
\end{equation}
where $\mathcal{G}_{0}=\emptyset$ and $\mathcal{G}_{n}=\sigma(\{(\xi_{k},\eta_{k}):1\le k\le n\})$
be the filtration generated by $\{(\xi_{n},\eta_{n})\}_{n\ge1}$.
Suppose that
\begin{equation}
K_{\eta}=\sup_{n\ge1}\mathbb{E}(|\eta_{n}|^{p})<\infty,\label{eq:-176}
\end{equation}
for some $p>2$. Let 
\begin{equation}
V_{F}=\int_{0}^{\infty}P_{\infty}(t)(\mu_{F}\ast\mu_{F})(dt),\label{eq:-45}
\end{equation}
and 
\begin{equation}
S_{n}=\frac{1}{\sqrt{nV_{F}}}\sum_{k=1}^{n}F(\omega_{k-1})\eta_{k}.\label{eq:-46}
\end{equation}
Then
\begin{equation}
\sup_{x\in\mathbb{R}}|\mathbb{P}(S_{n}\le x)-\Phi(x)|\le c_{\lambda,m,K_{\mu},K_{\eta},K_{\phi},\alpha,\beta,\omega_{0},p}n^{-\kappa},\label{eq:-180}
\end{equation}
where $\kappa=\min\big(\frac{1}{5},\frac{p-2}{2+2p}\big)$ and $c_{\lambda,m,K_{\mu},K_{\eta},K_{\phi},\alpha,\beta,\omega_{0},p}$
is a constant depending only on $p$, $\{\lambda_{n}\}_{n}$, $m$,
$K_{\mu}$, $K_{\eta}$, $K_{\phi}$, $\alpha$, $\beta$, $\omega_{0}$
and $p$. Consequently, 
\begin{equation}
\lim_{n\to\infty}\frac{1}{\sqrt{n}}\sum_{k=1}^{n}F(\omega_{k-1})\eta_{k}=\mathcal{N}(0,V_{F}),\label{eq:-43}
\end{equation}
in distribution, where $\mathcal{N}(\mu,\sigma^{2})$ is the normal
distribution with mean $\mu$ and variance $\sigma^{2}$.
\end{thm}

\begin{proof}
We only need to prove (\ref{eq:-180}). Let $X_{k,n}=(nV_{F})^{-1/2}F(\omega_{k-1})\eta_{k}$.
By (\ref{eq:-41}), $\mathbb{E}(X_{k,n}^{2}|\mathcal{G}_{k-1})=(nV_{F})^{-1}F(\omega_{k-1})^{2}.$
Note that $V_{F}=\int_{0}^{\infty}P_{\infty}(t)\mu_{F^{2}}(dt)$.
Applying (\ref{eq:-175}) to $F^{2}$ gives
\[
\mathbb{E}\Big(\Big|\frac{1}{n}\sum_{k=1}^{n}F(\omega_{k-1})^{2}-V_{F}\Big|^{2}\Big)\le c_{\lambda,m,K_{\mu},K_{\phi},\alpha,\beta,\omega_{0}}n^{-1}.
\]
Let $r=\min(1,p/2-1)$. By the above and H\"{o}lder's inequality,
we have 
\begin{equation}
\mathbb{E}\Big(\Big|\sum_{k=1}^{n}\mathbb{E}(X_{k,n}^{2}|\mathcal{G}_{k-1})-1\Big|^{1+r}\Big)\le c_{\lambda,m,K_{\mu},K_{\phi},\alpha,\beta,\omega_{0}}n^{-(1+r)/2}.\label{eq:-68}
\end{equation}
By (\ref{eq:-176}),
\begin{equation}
\mathbb{E}(|X_{k,n}|^{2+2r})\le K_{\eta}(nV_{F})^{-1-r}\mathbb{E}[F(\omega_{k-1})^{2+2r}].\label{eq:-205}
\end{equation}
By (\ref{eq:-178}) and (\ref{eq:-188}), we have
\begin{equation}
\sup_{k}\mathbb{E}[F(\omega_{k})]\le c_{\lambda,m,K_{\mu},K_{\phi},\alpha,\beta,\omega_{0}}<\infty.\label{eq:-186}
\end{equation}
Let $n_{0}=\lfloor2r\rfloor+3$. Applying (\ref{eq:-186}) to $F^{n_{0}}$
gives 
\begin{equation}
\sup_{k}\mathbb{E}[F(\omega_{k})^{n_{0}}]\le c_{\lambda,m,K_{\mu},K_{\phi},\alpha,\beta,\omega_{0},n_{0}}<\infty.\label{eq:-189}
\end{equation}
It follows from (\ref{eq:-186}), (\ref{eq:-189}) and H\"{o}lder's
inequality that $\sup_{k}\mathbb{E}[F(\omega_{k})^{2+2r}]\le c_{\lambda,m,K_{\mu},K_{\phi},\alpha,\beta,\omega_{0},p}$.
Therefore, by (\ref{eq:-205}),
\begin{equation}
\sum_{k=1}^{n}\mathbb{E}(|X_{k,n}|^{2+2r})\le c_{\lambda,m,K_{\mu},K_{\eta},K_{\phi},\alpha,\beta,\omega_{0},p}n^{-r}.\label{eq:-206}
\end{equation}
Therefore, by (\ref{eq:-68}), (\ref{eq:-206}) and Theorem \ref{thm:-2},
\[
\begin{aligned}\sup_{x\in\mathbb{R}}|\mathbb{P}(S_{n}\le x)-\Phi(x)| & \le c_{\lambda,m,K_{\mu},K_{\eta},K_{\phi},\alpha,\beta,\omega_{0},p}(n^{-r}+n^{-(1+r)/2})^{1/(3+2r)}\\
 & \le c_{\lambda,m,K_{\mu},K_{\eta},K_{\phi},\alpha,\beta,\omega_{0},p}n^{-r/(3+2r)}=c_{\lambda,m,K_{\mu},K_{\eta},K_{\phi},\alpha,\beta,\omega_{0},p}n^{-\kappa}.
\end{aligned}
\]
This completes the proof.
\end{proof}
\begin{prop}
\label{prop:-4}Suppose that $r(t)=r$, $\rho(t)=\rho$ and $\sigma_{t}=\sigma$
are constants. Then, as $\Delta t\to0$, the random vector $(\log I_{T},W_{T},\tilde{W}_{T})$
converges in distribution to the $(d+2)$-dimensional normal distribution
$\mathcal{N}(\mu T,TLL^{\mathrm{T}})$ with 
\begin{equation}
\mu=\big(r+(\rho-r)\sigma^{-1}\bar{\sigma}U(\lambda)-\bar{\sigma}^{2}V(\lambda)/2,0,\dots,0\big),\label{eq:-118}
\end{equation}
and 
\begin{equation}
LL^{\mathrm{T}}=\left[\begin{array}{ccc}
\bar{\sigma}^{2}V(\lambda) & \bar{\sigma}U(\lambda)\\
\bar{\sigma}U(\lambda) & 1\\
 &  & \mathrm{I}_{d}
\end{array}\right].\label{eq:-126}
\end{equation}
Therefore, $(\log I_{T},W_{T},\tilde{W}_{T})\to(\log X_{T},W_{T},\tilde{W}_{T})$
in distribution as $\Delta t\to0$. Moreover, for any $a=(a_{1},a_{2},\dots,a_{d+2})\in\mathbb{R}^{d+2}$,
\begin{equation}
\sup_{x\in\mathbb{R}}|\mathbb{P}(a\cdot Y_{T}\le x)-\mathbb{P}(a\cdot\bar{Y}_{T}\le x)|\le c_{\lambda,v_{0},T}\Delta t^{1/5},\label{eq:-207}
\end{equation}
where $Y_{T}=(\log I_{T},W_{T},\tilde{W}_{T})$ and $\bar{Y}_{T}=(\log X_{T},W_{T},\tilde{W}_{T})$. 
\end{prop}

\begin{proof}
By Proposition \ref{prop:-3}, it suffices to show that $(\log J_{T},W_{T},\tilde{W}_{T})\to\mathcal{N}(\mu T,TLL^{\mathrm{T}})$
in distribution, where $\{J_{t}\}_{t\ge0}$ is defined by (\ref{eq:-109}).
By It\"{o}'s formula,
\begin{equation}
\log J_{T}=A_{N}+\bar{\sigma}\sqrt{\frac{T}{N}}\sum_{n=1}^{N}(\sigma^{-2}u_{n-1}^{(N)})^{-1/2}Z_{n},\label{eq:-117}
\end{equation}
where $Z_{n}=\Delta t^{-1/2}(W_{t_{n}}-W_{t_{n-1}})$, $u_{n}=\lambda^{n}v_{0}+\sigma^{2}(1-\lambda)\sum_{k=1}^{n}\lambda^{n-k}Z_{k}^{2}$,
and
\begin{equation}
A_{N}=rT+(\rho-r)\sigma^{-1}\bar{\sigma}\frac{T}{N}\sum_{n=1}^{N}(\sigma^{-2}u_{n-1})^{-1/2}-\frac{\bar{\sigma}^{2}}{2}\frac{T}{N}\sum_{n=1}^{N}(\sigma^{-2}u_{n-1})^{-1},\label{eq:-134}
\end{equation}
Therefore, for any $a_{1},a_{2},\dots,a_{d+2}\in\mathbb{R}$,
\begin{equation}
\begin{aligned}a_{1}\log J_{T}+a_{2}W_{T}+\sum_{k=1}^{d}a_{k+2}\tilde{W}_{T}^{k} & =a_{1}A_{N}+T^{1/2}M_{N}\end{aligned}
,\label{eq:-120}
\end{equation}
where
\begin{equation}
M_{N}=\sqrt{\frac{T}{N}}\sum_{n=1}^{N}\Big([a_{1}\bar{\sigma}(\sigma^{-2}u_{n-1})^{-1/2}+a_{2}]Z_{n}+\sum_{k=1}^{d}a_{k+2}\tilde{Z}_{n}^{k}\Big),\label{eq:-133}
\end{equation}
and $\tilde{Z}_{n}^{k}=\Delta t^{-1/2}(\tilde{W}_{t_{n}}-\tilde{W}_{t_{n-1}})$. 

For the bounded variation part $A_{N}$, let $\mathcal{L}[f]$ be
the Laplace transform of a function $f$, and $F(t)=t^{-1/2}$. Then
$d\mu_{F}/dt=\mathcal{L}^{-1}[F](t)=(\pi t)^{-1/2}$. Applying Theorem
\ref{thm:-1} (i) to $\lambda_{n}=(1-\lambda)\lambda^{n}$, $\xi_{n}=Z_{n}^{2}$,
$\phi(t)=\mathbb{E}(e^{-t\xi_{1}})=(1+2t)^{-1/2}$ and $F(t)=t^{-1/2}$
gives
\begin{equation}
\lim_{N\to\infty}\frac{1}{N}\sum_{n=1}^{N}(\sigma^{-2}u_{n-1})^{-1/2}=\int_{0}^{\infty}(\pi t)^{-1/2}\prod_{k=0}^{\infty}(1+2t(1-\lambda)\lambda^{k})^{-1/2}dt=U(\lambda)\quad a.s.,\label{eq:-122}
\end{equation}
where the last equality follows from a simple change of variable.
Note that 
\[
\frac{d\mu_{F^{2}}}{dt}=\frac{d(\mu_{F}\ast\mu_{F})}{dt}=\int_{0}^{t}[\pi(t-s)]^{-1/2}(\pi s)^{-1/2}ds=1.
\]
Applying Theorem \ref{thm:-1} (i) to $\lambda_{n}=(1-\lambda)\lambda^{n}$,
$\xi_{n}=Z_{n}^{2}$, $\phi(t)=\mathbb{E}(e^{-t\xi_{1}})=(1+2t)^{-1/2}$
and $F(t)^{2}=t^{-1}$ gives
\begin{equation}
\lim_{N\to\infty}\frac{1}{N}\sum_{n=1}^{N}(\sigma^{-2}u_{n-1})^{-1}=\int_{0}^{\infty}\prod_{k=0}^{\infty}(1+2t(1-\lambda)\lambda^{k})^{-1/2}dt=V(\lambda)\quad a.s.\label{eq:-123}
\end{equation}
Therefore, by (\ref{eq:-122}) and (\ref{eq:-123}),
\begin{equation}
\begin{aligned}\lim_{N\to\infty}A_{N} & =[r+(\rho-r)\sigma^{-1}\bar{\sigma}U(\lambda)-\bar{\sigma}^{2}V(\lambda)/2]T\quad a.s.\end{aligned}
\label{eq:-135}
\end{equation}

For the martingale part $M_{N}$, let
\begin{equation}
G(t)=\Big((a_{1}\bar{\sigma}F(t)+a_{2})^{2}+\sum_{k=1}^{d}a_{k+2}^{2}\Big)^{1/2},\label{eq:-124}
\end{equation}
and
\begin{equation}
\eta_{n}=\frac{a_{1}\bar{\sigma}(\sigma^{-2}u_{n-1})^{-1/2}+a_{2}}{G(\sigma^{-2}u_{n-1})}Z_{n}+\sum_{k=1}^{d}\frac{a_{k+2}}{G(\sigma^{-2}u_{n-1})}\tilde{Z}_{n}^{k},\label{eq:-119}
\end{equation}
Then
\begin{equation}
M_{N}=\sqrt{\frac{T}{N}}\sum_{n=1}^{N}G(\sigma^{-2}u_{n-1})\eta_{n}.\label{eq:-125}
\end{equation}
Clearly, $G^{2}$ is a completely monotone function. Since $h(x)=x^{1/2}$
is a positive function and $h^{\prime}(x)$ is completely monotone,
we see that $G=h\circ G^{2}$ is also completely monotone. Moreover,
\begin{equation}
\begin{aligned}\mu_{G^{2}} & =\mathcal{L}^{-1}\Big[(a_{1}\bar{\sigma}F+a_{2})^{2}+\sum_{k=1}^{d}a_{k+2}^{2}\Big]\\
 & =\mathcal{L}^{-1}[a_{1}\bar{\sigma}F+a_{2}]\ast\mathcal{L}^{-1}[a_{1}\bar{\sigma}F+a_{2}]+\sum_{k=1}^{d}a_{k+2}^{2}\delta\\
 & =(a_{1}\bar{\sigma}\mu_{F}+a_{2}\delta)\ast(a_{1}\bar{\sigma}\mu_{F}+a_{2}\delta)+\sum_{k=1}^{d}a_{k+2}^{2}\delta\\
 & =a_{1}^{2}\bar{\sigma}^{2}\mu_{F^{2}}+2a_{1}a_{2}\bar{\sigma}\mu_{F}+\sum_{k=0}^{d}a_{k+2}^{2}\delta,
\end{aligned}
\label{eq:-208}
\end{equation}
where $\delta$ is the Dirac measure. It follows from (\ref{eq:-208})
that $\mu_{G}\ast\mu_{G}$ satisfies the condition (\ref{eq:-39}).
Therefore, by Lemma \ref{lem:-20}, $\mu_{G}$ also satisfies the
condition (\ref{eq:-39}). Applying Theorem \ref{thm:-3} to $\lambda_{n}=(1-\lambda)\lambda^{n}$,
$\xi_{n}=Z_{n}^{2}$, $\phi(t)=\mathbb{E}(e^{-t\xi_{1}})=(1+2t)^{-1/2}$
and $\eta_{n}$ gives
\begin{equation}
\lim_{N\to\infty}\frac{1}{\sqrt{N}}\sum_{n=1}^{N}G(\sigma^{-2}u_{n-1})\eta_{n}=\mathcal{N}\Big(0,\int_{0}^{\infty}\prod_{k=0}^{\infty}(1+2t(1-\lambda)\lambda^{k})^{-1/2}\mu_{G^{2}}(dt)\Big),\label{eq:-128}
\end{equation}
in distribution. Therefore, by $d\mu_{F}/dt=(\pi t)^{-1/2}$ and $d\mu_{F^{2}}/dt=1$,
\[
\begin{aligned}\int_{0}^{\infty}\prod_{k=0}^{\infty}(1+2t(1-\lambda)\lambda^{k})^{-1/2}\mu_{G^{2}}(dt) & =a_{1}^{2}\bar{\sigma}^{2}V(\lambda)+2a_{1}a_{2}\bar{\sigma}U(\lambda)+\sum_{k=0}^{d}a_{k+2}^{2}\end{aligned}
,
\]
which, together with (\ref{eq:-128}), shows that 
\begin{equation}
\lim_{N\to\infty}M_{N}=\mathcal{N}\Big(0,\Big(a_{1}^{2}\bar{\sigma}^{2}V(\lambda)+2a_{1}a_{2}\bar{\sigma}U(\lambda)+\sum_{k=0}^{d}a_{k+2}^{2}\Big)T\Big),\label{eq:-130}
\end{equation}
in distribution. Combining (\ref{eq:-120}), (\ref{eq:-135}) and
(\ref{eq:-130}) gives that
\begin{equation}
\begin{aligned} & \lim_{N\to\infty}\Big(a_{1}\log J_{T}+a_{2}W_{T}+\sum_{k=1}^{d}a_{k+2}\tilde{W}_{T}^{k}\Big)\\
 & =\mathcal{N}\Big(a_{1}(r+(\rho-r)\sigma^{-1}\bar{\sigma}U(\lambda)-\bar{\sigma}^{2}V(\lambda)/2)T,\Big(a_{1}^{2}\bar{\sigma}^{2}V(\lambda)+2a_{1}a_{2}\bar{\sigma}U(\lambda)+\sum_{k=0}^{d}a_{k+2}^{2}\Big)T\Big)
\end{aligned}
\label{eq:-131}
\end{equation}
in distribution for any $a_{1},a_{2},\dots,a_{d+2}\in\mathbb{R}$.
Therefore, $(\log J_{T},W_{T},\tilde{W}_{T})$ converges in distribution
to the $(d+2)$-dimensional normal distribution $\mathcal{N}(\mu T,TLL^{\mathrm{T}})$
with mean and covariance given by (\ref{eq:-118}) and (\ref{eq:-126})
respectively. It is easily seen that $(\log X_{T},W_{T},\tilde{W}_{T})$
has the same normal distribution $\mathcal{N}(\mu T,TLL^{\mathrm{T}})$.
Therefore, 
\begin{equation}
\lim_{N\to\infty}(\log J_{T},W_{T},\tilde{W}_{T})=(\log X_{T},W_{T},\tilde{W}_{T})\label{eq:-132}
\end{equation}
in distribution. This proves the convergence of the terminal joint
distribution.

We now prove (\ref{eq:-207}). Clearly, $\eta_{n}$ is a standard
normal distribution when conditional on $\mathcal{F}_{n-1}$. Therefore,
$\mathbb{E}(|\eta_{n}|^{4}|\mathcal{F}_{n-1})=3$ and consequently
$\mathbb{E}(|\eta_{n}|^{4})=3$. For any $a\in\mathbb{R}^{d+2}$ with
$|a|=1$, by Theorem \ref{thm:-3} with $p=4$, we obtain that
\[
\sup_{x\in\mathbb{R}}|\mathbb{P}(|La|^{-1}M_{N}\le x)-\Phi(x)|\le c_{\lambda,v_{0},T}N^{-1/5},
\]
or equivalently, 
\begin{equation}
\sup_{x\in\mathbb{R}}|\mathbb{P}(M_{N}\le x)-\mathbb{P}(La\cdot Z\le x)|\le c_{\lambda,v_{0},T}N^{-1/5},\label{eq:-209}
\end{equation}
where $Z$ is a $(d+2)$-dimensional standard normal distribution,
and $c_{\lambda,v_{0}}$ is a constant depending only on $\lambda$,
$v_{0}$ and $T$. Note that $c_{\lambda,v_{0},T}$ does not depend
on $a\in\mathbb{R}^{d+2}$ due to the boundedness of the variance
$|La|^{2}$. By (\ref{eq:-175}), (\ref{eq:-209}) and Lemma \ref{lem:-11},
we deduce that
\[
\sup_{x\in\mathbb{R}}|\mathbb{P}(a\cdot Y_{T}\le x)-\mathbb{P}(\mu+La\cdot Z\le x)|\le c_{\lambda,v_{0},T}N^{-1/5}.
\]
Therefore, (\ref{eq:-207}) holds for all $a\in\mathbb{R}^{d+2}$
with $|a|=1$. For general $a\in\mathbb{R}^{d+2}$, let $\bar{a}=a/|a|$
and $\bar{x}=x/|a|$. Applying (\ref{eq:-207}) to $\bar{a}$ and
$\bar{x}$ gives 
\[
|\mathbb{P}(\bar{a}\cdot Y_{T}\le\bar{x})-\mathbb{P}(\bar{a}\cdot\bar{Y}_{T}\le\bar{x})|\le c_{\lambda,v_{0},T}N^{-1/5}.
\]
That is $|\mathbb{P}(a\cdot Y_{T}\le x)-\mathbb{P}(a\cdot\bar{Y}_{T}\le x)|\le c_{\lambda,v_{0},T}N^{-1/5}$.
This completes the proof of (\ref{eq:-207}).
\end{proof}

\subsection{\label{subsec:-6}Convergence Rate of Expectations}
\begin{lem}
\label{lem:-13}Let $h(u)$ be a bounded differentiable function on
$\mathbb{R}$ such that $\Vert h^{\prime}\Vert_{L^{1}}<\infty$. Suppose
that $\{Y_{n}\}_{n}$ and $Y$ satisfy the property 
\begin{equation}
\sup_{a\in\mathbb{R}^{d}}\sup_{u\in\mathbb{R}}|\mathbb{P}(a\cdot Y_{n}\ge u)-\mathbb{P}(a\cdot Y\ge u)|\le c_{Y}n^{-\alpha}.\label{eq:-160}
\end{equation}
Then 
\begin{equation}
\sup_{a\in\mathbb{R}^{d}}|\mathbb{E}(h(a\cdot Y_{n}))-\mathbb{E}(h(a\cdot Y))|\le c_{Y}\Vert h^{\prime}\Vert_{L^{1}}n^{-\alpha}.\label{eq:-162}
\end{equation}
\end{lem}

\begin{proof}
For any random variable $\xi$, it is easily seen by integration by
parts that
\[
\mathbb{E}(h(\xi))=h(0)+\int_{0}^{\infty}h^{\prime}(u)\mathbb{P}(\xi\ge u)du-\int_{-\infty}^{0}h^{\prime}(u)[1-\mathbb{P}(\xi\ge u)]du.
\]
Therefore,
\[
\mathbb{E}(h(a\cdot Y_{n}))-\mathbb{E}(h(a\cdot Y))=\int_{-\infty}^{\infty}h^{\prime}(u)[\mathbb{P}(a\cdot Y_{n}\ge u)-\mathbb{P}(a\cdot Y\ge u)]du,
\]
which implies (\ref{eq:-162}).
\end{proof}
\begin{lem}
\label{lem:-14}Suppose that $\{Y_{n}\}_{n}$ and $Y$ satisfy the
property (\ref{eq:-160}). If $Y$ has $d$-dimensional normal distribution
$\mathcal{N}(0,LL^{\mathrm{T}})$ and 
\[
|L^{\mathrm{T}}a|^{2}\le\kappa^{2}|a|^{2},\quad a\in\mathbb{R}^{d}.
\]
Then, for any $a\in\mathbb{R}^{d}$ and $u\ge\kappa|a|$,
\begin{equation}
\mathbb{P}(|a\cdot Y|>u)\le e^{-(\kappa|a|)^{-2}u^{2}/2},\label{eq:-164}
\end{equation}
and
\begin{equation}
\mathbb{P}(|a\cdot Y_{n}|>u)\le e^{-(\kappa|a|)^{-2}u^{2}/2}+2c_{Y}n^{-\alpha}.\label{eq:-165}
\end{equation}
\end{lem}

\begin{proof}
Let $\sigma^{2}=|L^{\text{T}}a|^{2}$. Then $a\cdot X$ has a normal
distribution $\mathcal{N}(0,\sigma^{2})$. For $u\ge\kappa|a|\ge\sigma$,
we have
\[
\begin{aligned}\mathbb{P}(a\cdot Y>u) & =\frac{1}{\sqrt{2\pi}}\int_{\sigma^{-1}u}^{\infty}e^{-x^{2}/2}dx\le\frac{1}{\sqrt{2\pi}}\int_{\sigma^{-1}u}^{\infty}xe^{-x^{2}/2}dx\le\frac{1}{2}e^{-(\kappa|a|)^{-2}u^{2}/2}.\end{aligned}
\]
Similarly, $\mathbb{P}(a\cdot X<-u)\le2^{-1}e^{-(\kappa|a|)^{-2}u^{2}/2}$.
This proves (\ref{eq:-164}). By(\ref{eq:-160}), 
\[
\mathbb{P}(a\cdot Y_{n}\ge u)\le\mathbb{P}(a\cdot Y\ge u)+c_{Y}n^{-\alpha}\le\frac{1}{2}e^{-(\kappa|a|)^{-2}u^{2}/2}+c_{Y}n^{-\alpha}.
\]
Similarly, $\mathbb{P}(a\cdot Y_{n}<-u)\le2^{-1}e^{-(\kappa|a|)^{-2}u^{2}/2}+c_{Y}n^{-\alpha}$.
This proves (\ref{eq:-165}).
\end{proof}
\begin{lem}
\label{lem:-15}Let $g(u)$ be a differentiable function on $\mathbb{R}$
such that $|g|\le1$ and $|g^{\prime}|\le1$. Then, for any $R>0$,
there exists a differentiable function $h(u)$ on $\mathbb{R}$ such
that $h(u)=g(u)$ on $|u|\le R$, $h(u)=0$ on $|u|\ge R+1$, and
$|h|\le1$ and $|h^{\prime}|\le5$.
\end{lem}

\begin{proof}
Let $\phi$ be a smooth function on $\mathbb{R}$ such that $\phi(u)=0$
on $u\le0$, $\phi(u)=1$ on $u\ge1$, and $0\le\phi\le1$ and $|\phi^{\prime}|\le2$.
Let 
\[
h(u)=g(u)\phi(u+R+1)[1-\phi(u-R)].
\]
Then $h(u)=g(u)$ on $|u|\le R$, $h(u)=0$ on $|u|\ge R+1$, and
$|h|\le1$. The bound $|h^{\prime}|\le5$ follows easily from the
chain rule and the bounds $|g|\le1$, $|g^{\prime}|\le1$, $0\le\phi\le1$
and $|\phi^{\prime}|\le2$.
\end{proof}
For a smooth function $f$ on $\mathbb{R}^{d}$, $k\in\mathbb{N}$
and $1\le p<\infty$, let $\Vert f\Vert_{W^{k,p}(\mathbb{R}^{d})}$
be the Sobolev norm defined by
\begin{equation}
\Vert f\Vert_{W^{k,p}(\mathbb{R}^{d})}=\Big(\sum_{|\gamma|\le k}\Vert\partial^{\gamma}f\Vert_{L^{p}(\mathbb{R}^{d})}^{p}\Big)^{1/p},\label{eq:-24}
\end{equation}
where $\partial^{\gamma}f=\partial_{x_{1}}^{\gamma_{1}}\cdots\partial_{x_{d}}^{\gamma_{d}}f$
for any multi-index $\gamma=(\gamma_{1},\dots,\gamma_{d})\in\mathbb{N}^{d}$.
When $p=\infty$, the Sobolev norm $\Vert f\Vert_{W^{k,\infty}(\mathbb{R}^{d})}$
is defined by 
\begin{equation}
\Vert f\Vert_{W^{k,\infty}(\mathbb{R}^{d})}=\sum_{|\gamma|\le k}\Vert\partial^{\gamma}f\Vert_{L^{\infty}(\mathbb{R}^{d})}.\label{eq:-25}
\end{equation}
For $1\le p<\infty$, the Sobolev space $W^{k,p}(\mathbb{R}^{d})$
is the completion of $C^{\infty}(\mathbb{R}^{d})$ with respect to
the norm $\Vert\cdot\Vert_{W^{k,p}(\mathbb{R}^{d})}$. If $0\le r\le k-d/p$
for some $r\in\mathbb{N}$, by the Sobolev embedding theorem, $W^{k,p}(\mathbb{R}^{d})\subseteq C^{r}(\mathbb{R}^{d})$
in the sense that any $f\in W^{k,p}(\mathbb{R}^{d})$ equals a.e.
to a function in $C^{r}(\mathbb{R}^{d})$. To avoid confusion in discussion
of convergence in distribution for random variables, we use the notation
$f\in C(\mathbb{R}^{d})\cap W^{k,p}(\mathbb{R}^{d})$ to indicate
that the continuous version of $f\in W^{k,p}(\mathbb{R}^{d})$ is
being considered.
\begin{prop}
\label{prop:-5}Let $\{Y_{n}\}_{n}$ be a sequence of $d$-dimensional
random vectors, and $Y$ be a $d$-dimensional random vector satisfying
the property
\begin{equation}
\sup_{a\in\mathbb{R}^{d}}\sup_{u\in\mathbb{R}}|\mathbb{P}(a\cdot Y_{n}\ge u)-\mathbb{P}(a\cdot Y\ge u)|\le c_{Y}n^{-\alpha}.\label{eq:-166}
\end{equation}
Suppose that $Y$ has $d$-dimensional normal distribution $\mathcal{N}(\mu,LL^{\mathrm{T}})$
with $\mu\in\mathbb{R}^{d}$ and $LL^{\mathrm{T}}$ satisfying
\begin{equation}
|L^{\mathrm{T}}a|^{2}\le\kappa^{2}|a|^{2},\quad a\in\mathbb{R}^{d}.\label{eq:-163}
\end{equation}
Then, for any $f\in C(\mathbb{R}^{d})\cap W^{d+2,1}(\mathbb{R}^{d})$
and $n\ge e^{\alpha^{-1}/2}$,
\begin{equation}
|\mathbb{E}(f(Y_{n}))-\mathbb{E}(f(Y))|\le c_{Y,\kappa,\alpha,d}\Vert f\Vert_{W^{d+2,1}}n^{-\alpha}(\log n)^{1/2},\label{eq:-167}
\end{equation}
where $c_{Y,\kappa,\alpha,d}$ is a constant depending only on the
constants $c_{Y}$, $\alpha$, $\kappa$ and the dimension $d$.
\end{prop}

\begin{proof}
Note that (\ref{eq:-166}) and the Sobolev norm $\Vert\cdot\Vert_{W^{d+2,1}}$
are translation invariant. We may assume $\mu=0$ by replacing $Y_{n}$
and $Y$ by $\tilde{Y}_{n}=Y_{n}-\mu$ and $\tilde{Y}=Y-\mu$ respectively.
Let 
\[
\hat{f}(\xi)=\int_{\mathbb{R}^{d}}f(x)e^{-\text{i}2\pi\xi\cdot x}dx
\]
be the Fourier transform of $f$. Since $\widehat{\partial^{\gamma}f}(\xi)=(\text{i}2\pi\xi)^{\gamma}\hat{f}(\xi)$
and the Fourier transform is a bounded linear operator from $L^{1}$
to $L^{\infty}$, we have $\sup_{\xi}|(\text{i}2\pi\xi)^{\gamma}\hat{f}(\xi)|\le\Vert\partial^{\gamma}f\Vert_{L^{1}}$
for all $|\gamma|\le d+2$. Consequently, 
\begin{equation}
\sup_{\xi\in\mathbb{R}^{d}}(1+|\xi|)^{d+2}|\hat{f}(\xi)|\le c_{d}\Vert f\Vert_{W^{d+2,1}},\label{eq:-168}
\end{equation}
where $c_{d}$ is a constant depending only on $d$. This implies
that $\hat{f}\in L^{1}$, and therefore, by the Fourier inversion
theorem (see \cite[Corollary 1.21, p. 11]{SW71}), 
\begin{equation}
f(x)=\int_{\mathbb{R}^{d}}\hat{f}(\xi)e^{\text{i}2\pi\xi\cdot x}d\xi,\label{eq:-169}
\end{equation}
for a.e. $x\in\mathbb{R}^{d}$. By $\hat{f}\in L^{1}$ and the dominated
convergence theorem, the right-hand side of (\ref{eq:-169}) is continuous
in $x$. Therefore, (\ref{eq:-169}) holds for all $x\in\mathbb{R}^{d}$.
By $\hat{f}\in L^{1}$ and Fubini's theorem, 
\begin{equation}
\mathbb{E}(f(Y_{n}))-\mathbb{E}(f(Y))=\int_{\mathbb{R}^{d}}\hat{f}(\xi)[\mathbb{E}(e^{\text{i}2\pi\xi\cdot Y_{n}})-\mathbb{E}(e^{\text{i}2\pi\xi\cdot Y_{n}})]d\xi.\label{eq:-170}
\end{equation}

Let $g(u)=(2\pi)^{-1}e^{\text{i}2\pi u}$. Then $|g|\le1$ and $|g^{\prime}|\le1$.
Let $\rho\ge\kappa$ and $R=\rho|\xi|$. By Lemma \ref{lem:-15},
there exists a differentiable function $h(u)$ such that $h(u)=g(u)$
on $|u|\le R$, $h(u)=0$ on $|u|\ge R+1$, and $|h|\le1$ and $|h^{\prime}|\le5$.
Then,
\[
\mathbb{E}(g(\xi\cdot Y_{n}))-\mathbb{E}(h(\xi\cdot Y_{n}))=\mathbb{E}\big[(g(\xi\cdot Y_{n})-h(\xi\cdot Y_{n}))\mathds{1}_{\{|\xi\cdot Y_{n}|>R\}}\big].
\]
Therefore, by $|g|\le1$ and $|h|\le1$ and Lemma \ref{lem:-14},
\begin{equation}
|\mathbb{E}(g(\xi\cdot Y_{n}))-\mathbb{E}(h(\xi\cdot Y_{n}))|\le2\mathbb{P}(|\xi\cdot Y_{n}|>R)\le2e^{-\kappa^{-2}\rho^{2}/2}+4c_{Y}n^{-\alpha}.\label{eq:-171}
\end{equation}
Similarly,
\begin{equation}
|\mathbb{E}(g(\xi\cdot Y))-\mathbb{E}(h(\xi\cdot Y))|\le2e^{\kappa^{-2}\rho^{2}/2}.\label{eq:-172}
\end{equation}
Therefore, by (\ref{eq:-171}) and (\ref{eq:-172})
\begin{equation}
|\mathbb{E}(g(\xi\cdot Y_{n}))-\mathbb{E}(g(\xi\cdot Y))|\le|\mathbb{E}(h(\xi\cdot Y_{n}))-\mathbb{E}(h(\xi\cdot Y))|+4(e^{\kappa^{-2}\rho^{2}/2}+c_{Y}n^{-\alpha}).\label{eq:-173}
\end{equation}
Since $h(u)=0$ on $|u|\ge R+1$ and $|h^{\prime}|\le5$, we have
$\Vert h^{\prime}\Vert_{L^{1}}\le5(R+1)=5(\rho|\xi|+1)$. By Lemma
\ref{lem:-13}, 
\[
|\mathbb{E}(h(\xi\cdot Y_{n}))-\mathbb{E}(h(\xi\cdot Y))|\le5c_{Y}(\rho|\xi|+1)n^{-\alpha},
\]
which, together with (\ref{eq:-173}), gives 
\[
|\mathbb{E}(g(\xi\cdot Y_{n}))-\mathbb{E}(g(\xi\cdot Y))|\le(4+9c_{Y})[(\rho|\xi|+1)n^{-\alpha}+e^{\kappa^{-2}\rho^{2}/2}].
\]
When $n\ge e^{\alpha^{-1}/2}$, setting $\rho=\kappa\sqrt{2\alpha\log n}\ge\kappa$
gives 
\begin{equation}
|\mathbb{E}(g(\xi\cdot Y_{n}))-\mathbb{E}(g(\xi\cdot Y))|\le c_{Y,\kappa,\alpha}(1+|\xi|)n^{-\alpha}(\log n)^{1/2},\label{eq:-174}
\end{equation}
where $c_{Y,\kappa,\alpha}$ is a constant depending only on $c_{Y}$,
$\kappa$ and $\alpha$. 

It follows from (\ref{eq:-170}) and (\ref{eq:-174}) that 
\[
|\mathbb{E}(f(Y_{n}))-\mathbb{E}(f(Y))|\le c_{Y,\kappa,\alpha}n^{-\alpha}(\log n)^{1/2}\int_{\mathbb{R}^{d}}(1+|\xi|)|\hat{f}(\xi)|d\xi.
\]
By (\ref{eq:-168}),
\[
\int_{\mathbb{R}^{d}}(1+|\xi|)|\hat{f}(\xi)|d\xi\le c_{d}\Vert f\Vert_{W^{d+2,1}}\int_{\mathbb{R}^{d}}(1+|\xi|)^{-d-1}d\xi\le c_{d}A_{d-1}\Vert f\Vert_{W^{d+2,1}},
\]
where $A_{d-1}=2\pi^{d/2}/\Gamma(d/2)$ is the area of the $(d-1)$-dimensional
unit sphere. This completes the proof.
\end{proof}
The following rate of convergence in distribution for $(\log I_{T},W_{T},\tilde{W}_{T})$
is a direct consequence of Proposition \ref{prop:-4} and Proposition
\ref{prop:-5}.
\begin{cor}
\label{cor:-2}Let $Y_{T}=(\log I_{T},W_{T},\tilde{W}_{T})$ and $\bar{Y}_{T}=(\log X_{T},W_{T},\tilde{W}_{T})$.
Suppose that $\rho(t)=\rho$, $r(t)=r$ and $\sigma_{t}=\sigma$ are
constants. Then, for any $f\in W^{d+2,1}(\mathbb{R}^{d})$, 
\[
|\mathbb{E}(f(Y_{T}))-\mathbb{E}(f(\bar{Y}_{T}))|\le c_{d,\lambda,\sigma_{\ast},\sigma^{\ast},v_{0},T}\Vert f\Vert_{W^{d+2,1}}\Delta t^{1/5}(\log(T\Delta t^{-1}))^{1/2},
\]
where $c_{d,\lambda,v_{0},T}$ is a constant depending only on $d$,
$\lambda$, $\sigma_{\ast}$, $\sigma^{\ast}$, $v_{0}$ and $T$.
\end{cor}

\subsection{\label{subsec:-8}Proof of Theorem \ref{thm:}}

We are now in a position to prove Theorem \ref{thm:} for general
coefficients $\rho(t)$, $r(t)$ and $\{\sigma_{t}\}_{t\ge0}$. We
follow the same strategy as the proof of \cite[Theorem 1.5]{LG26}
with some modifications.
\begin{lem}
\label{lem:-7}Let $\{\mathbf{W}_{t}\}_{t\ge0}$ be a $d$-dimensional
Brownian motion and $\mathbf{W}_{s,t}=\mathbf{W}_{t}-\mathbf{W}_{s}$.
Let $\{Y_{t}\}_{t\ge0}$ be a progressively measurable process such
that $\mathbb{E}\big(\int_{0}^{T}|Y_{t}|^{2}dt\big)<\infty$. Then
there exists a sequence $\{Y_{t}^{n}\}_{t\ge0}$ of simple processes
of the form
\begin{equation}
Y_{t}^{n}=\sum_{k=1}^{n}f_{k-1,n}(\mathbf{W}_{\tau_{0}},\mathbf{W}_{\tau_{0},\tau_{1}}\dots,\mathbf{W}_{\tau_{k-2},\tau_{k-1}})\mathds{1}_{[\tau_{k-1},\tau_{k})}(t),\quad0<\tau_{0}<\cdots<\tau_{n}\le T,\label{eq:-121}
\end{equation}
such that each $f_{k,n}$ is a smooth function with
\[
\mathbb{E}[f(\mathbf{W}_{\tau_{0}},\mathbf{W}_{\tau_{0},\tau_{1}}\dots,\mathbf{W}_{\tau_{k-2},\tau_{k-1}})^{2}]<\infty,
\]
and that 
\begin{equation}
\lim_{n\to\infty}\mathbb{E}\Big(\int_{0}^{T}|Y_{t}^{n}-Y_{t}|^{2}dt\Big)^{1/2}=0.\label{eq:-129}
\end{equation}
If in addition $Y_{\ast}\le Y_{t}\le Y^{\ast}$ for some constants
$Y_{\ast},Y^{\ast}\in\mathbb{R}$, then the smooth functions $f_{k,n}$
can be chosen such that $Y_{\ast}\le f_{k,n}\le Y^{\ast}$ and $\Vert f_{k,n}\Vert_{W^{r,\infty}}<\infty$
for 
\begin{equation}
\Vert f_{k,n}\Vert_{W^{r,\infty}}<\infty,\label{eq:-210}
\end{equation}
 for all $k$, $n$ and $r$.
\end{lem}

\begin{proof}
Since progressively measurable simple processes are dense in $L^{2}(\Omega;L^{2}([0,T]))$,
it suffices to prove the lemma for processes $\{Y_{t}\}_{t\ge0}$
of the form
\begin{equation}
Y_{t}=\sum_{k=1}^{n}\eta_{k-1}\mathds{1}_{[u_{k-1},u_{k})}(t),\quad0\le u_{0}<u_{1}<\cdots<u_{n}\le T,\label{eq:-137}
\end{equation}
where $\eta_{k}\in\mathcal{F}_{u_{k}}$ and $\mathbb{E}(\eta_{k}^{2})<\infty$
for $0\le k<n$. Let $\tau\in(0,T]$ be a constant and $\mathcal{P}_{\tau}$
be the subspace
\[
\mathcal{P}_{\tau}=\{\xi\in L^{2}(\mathbb{P},\mathcal{F}_{\tau}):\xi=f(\mathbf{W}_{\tau_{0}},\mathbf{W}_{\tau_{0},\tau_{1}},\dots,\mathbf{W}_{\tau_{j-1},\tau_{j}})\;\text{for some continuous}\;f\;\text{and}\;0<\tau_{0}<\cdots<\tau_{j}\le\tau\}.
\]
We show that $\mathcal{P}_{\tau}$ is dense in $L^{2}(\Omega,\mathcal{F}_{\tau})$.
To see this, let
\[
h_{n}(t)=\frac{1}{\tau}\sum_{j=1}^{2^{n}}(-1)^{j-1}\mathds{1}_{[(j-1)2^{-n}\tau,j2^{-n}\tau)}(t),\quad t\in[0,\tau].
\]
Then $\{h_{n}\}_{n\ge0}$ is an orthonormal basis of $L^{2}([0,\tau])$.
By the Wiener chaos decomposition (see \cite[Theorem 1.1.1, p. 6]{Nualart06}),
the space of polynomials of $\big\{\int_{0}^{\tau}h_{n}(t)d\mathbf{W}_{t}\big\}_{n\ge0}$
is dense in $L^{2}(\mathbb{P},\mathcal{F}_{\tau})$. Clearly, $\int_{0}^{\tau}h_{n}(t)d\mathbf{W}_{t}\in\mathcal{P}_{\tau}$
for all $n$ and therefore all polynomials of $\big\{\int_{0}^{\tau}h_{n}(t)d\mathbf{W}_{t}\big\}_{n\ge0}$
are in $\mathcal{P}_{\tau}$. This proves the density of $\mathcal{P}_{\tau}$
in $L^{2}(\mathbb{P},\mathcal{F}_{\tau})$. For each $\eta_{k}$ in
(\ref{eq:-137}), by the density of $\mathcal{P}_{u_{k}}$ in $L^{2}(\mathbb{P},\mathcal{F}_{u_{k}})$,
there exists a sequence $\{\xi_{k,n}\}_{n}\subseteq\mathcal{P}_{u_{k}}$
such that $\lim_{n\to\infty}\xi_{k,n}=\eta_{k}$ in $L^{2}(\mathbb{P},\mathcal{F}_{u_{k}})$.
Let $Y_{t}^{n}=\sum_{k=1}^{n}\xi_{k,n}\mathds{1}_{[u_{k-1},u_{k})}(t)$.
It is easily seen that $Y_{t}^{n}$ can be written in the form (\ref{eq:-137})
and satisfies (\ref{eq:-129}). 

Suppose in addition that $Y_{\ast}\le Y_{t}\le Y^{\ast}$. By considering
$Y_{t}^{n}-(A_{\ast}+A^{\ast})/2$ and $Y-(A_{\ast}+A^{\ast})/2$,
we may assume that $A_{\ast}=-M$ and $A^{\ast}=M$ with $M\ge0$.
Let $\{Y_{t}^{n}\}_{t\ge0}$ be a sequence of the form (\ref{eq:-121})
such that $\mathbb{E}\big(\int_{0}^{T}|Y_{t}^{n}-Y_{t}|^{2}dt\big)^{1/2}\to0$.
By considering a subsequence if necessary, we may assume that $Y_{t}^{n}\to Y_{t}$
for all $t\in[0,T]$ a.s. Let $f_{k,n}(x_{1},\dots,x_{k})$, $(x_{1},\dots,x_{k})\in\mathbb{R}^{k\times d}$,
be the smooth function in (\ref{eq:-121}). Let
\[
g_{k,n}(x_{1},\dots,x_{k})=\max(-M,\min(f_{k,n}(x_{1},\dots,x_{k}),M)),
\]
and $\tilde{Y}_{t}^{n}$ be the process of the form with $f_{k,n}$
in (\ref{eq:-121}) replaced by $g_{k,n}$. Clearly, $\tilde{Y}_{t}^{n}$
is of the form (\ref{eq:-121}) and $\tilde{Y}_{t}^{n}\to Y_{t}$
for all $t\in[0,T]$ a.s. By the dominated convergence theorem, $\mathbb{E}\big(\int_{0}^{T}|\tilde{Y}_{t}^{n}-Y_{t}|^{2}dt\big)^{1/2}\to0$.
Therefore, by replacing $Y_{t}^{n}$ by $\tilde{Y}_{t}^{n}$ if necessary,
we may assume that each $Y_{t}^{n}$ is of the form (\ref{eq:-121})
such that $|f_{k,n}|\le M$ with $f_{k,n}$ being continuous only.
Let $\phi^{k}\in C_{c}^{\infty}(\mathbb{R}^{k\times d})$ be a positive
smooth function with compact support such that $\Vert\phi^{k}\Vert_{L^{1}}=1$.
Let $\tilde{f}_{k,n}=f_{k,n}\ast\phi_{\epsilon}^{k}$, where $\phi_{\epsilon}^{k}(x)=\epsilon^{-k}\phi^{k}(x/\epsilon)$.
By the continuity of $f_{k,n}$, $\lim_{\epsilon\to0}(f_{k,n}\ast\phi_{\epsilon}^{k})(x)=f_{k,n}(x)$
for all $x\in\mathbb{R}^{k\times d}$. By Young's inequality, 
\[
\Vert f_{k,n}\ast\phi_{\epsilon}^{k}\Vert_{L^{\infty}}\le\Vert f_{k,n}\Vert_{L^{\infty}}\Vert\phi_{\epsilon}^{k}\Vert_{L^{1}}\le M.
\]
Moreover, by the dominated convergence theorem, 
\[
\lim_{\epsilon\to0}\mathbb{E}\big[\big|(f_{k,n}\ast\phi_{\epsilon}^{k}-f_{k,n})(\mathbf{W}_{\tau_{0}},\mathbf{W}_{\tau_{0},\tau_{1}}\dots,\mathbf{W}_{\tau_{k-1},\tau_{k}})\big|^{2}\big]=0.
\]
Therefore, we may choose $\epsilon_{k,n}$ such that 
\[
\mathbb{E}\big[\big|(f_{k,n}\ast\phi_{\epsilon}^{k}-f_{k,n})(\mathbf{W}_{\tau_{0}},\mathbf{W}_{\tau_{0},\tau_{1}}\dots,\mathbf{W}_{\tau_{k-1},\tau_{k}})\big|^{2}\big]^{1/2}\le n^{-2},
\]
Let $\tilde{f}_{k,n}=f_{k,n}\ast\phi_{\epsilon_{k,n}}^{k}$ and $\tilde{Y}_{t}^{n}$
be the process of the form with $f_{k,n}$ in (\ref{eq:-121}) replaced
by $\tilde{f}_{k,n}$. Then each $\tilde{f}_{k,n}$ is a smooth function
with $|\tilde{f}_{k,n}|\le M$, and $\mathbb{E}\big(\int_{0}^{T}|\tilde{Y}_{t}^{n}-Y_{t}|^{2}dt\big)^{1/2}\to0$.
To see that $\tilde{f}_{k,n}$ satisfies (\ref{eq:-210}), let $\gamma$
be a multi-index $\gamma$ with $|\gamma|=r$. Then $\partial^{\gamma}(f_{k,n}\ast\phi_{\epsilon}^{k})=f_{k,n}\ast\partial^{\gamma}\phi_{\epsilon}^{k}=\epsilon^{-r}f_{k,n}\ast(\partial^{\gamma}\phi^{k})_{\epsilon}$.
By Young's inequality,
\[
\Vert\partial^{\gamma}(f_{k,n}\ast\phi_{\epsilon}^{k})\Vert_{L^{\infty}}\le\epsilon^{-r}\Vert f_{k,n}\Vert_{L^{\infty}}\Vert(\partial^{\gamma}\phi^{k})_{\epsilon}\Vert_{L^{1}}\le M\epsilon^{-r}\Vert\phi^{k}\Vert_{W^{r,1}}<\infty.
\]
This completes the proof.
\end{proof}
\begin{rem}
\label{rem:-8}Let $\{Y_{t}^{n}\}_{t\ge0}$ be a sequence of progressively
measurable processes such that $\{(Y_{t}^{n},W_{t},\tilde{W}_{t})\}_{t\ge0}$
converges in distribution to $\{(Y_{t},W_{t},\tilde{W}_{t})\}_{t\ge0}$.
It follows readily from Lemma \ref{lem:-7} that, for any progressively
measurable process $\{q_{t}\}_{t\ge0}$, the joint process $\{(Y_{t}^{n},q_{t})\}_{t\ge0}$
converges in distribution to $\{(Y_{t},q_{t})\}_{t\ge0}$.
\end{rem}

We are now in a position to prove Theorem \ref{thm:}. 
\begin{proof}[Proof of Theorem \ref{thm:}]
Throughout the proof, we will use $c_{q_{1},\dots,q_{m}}$ to denote
any constant depending only on the quantities $q_{1},\dots,q_{m}$,
and the value of $c_{q_{1},\dots,q_{m}}$ may vary at different appearances.
In view of Corollary \ref{cor:}, it suffices to show, for any $0=\tau_{0}<\tau_{1}<\cdots<\tau_{\ell}\le T$
and smooth functions $F_{k}\in C^{\infty}$ with $\Vert F_{k}\Vert_{W^{1,\infty}}\le1$,
that
\begin{equation}
\lim_{\Delta t\to0}\mathbb{E}\Big(\prod_{k=1}^{\ell}F_{k}(Y_{\tau_{k}})\Big)=\mathbb{E}\Big(\prod_{k=1}^{\ell}F_{k}(\bar{Y}_{\tau_{k}})\Big),\label{eq:-27}
\end{equation}
where $Y_{t}=(\log I_{t},W_{t},\tilde{W}_{t})$ and $\bar{Y}_{t}=(\log X_{t},W_{t},\tilde{W}_{t})$,
and $\Vert\cdot\Vert_{W^{1,\infty}}$ is the Sobolev norm defined
by (\ref{eq:-25}). In view of Proposition (\ref{prop:-6}), Proposition
(\ref{prop:-7}) and Lemma (\ref{lem:-7}), and inserting time points
to the approximating sequence in (\ref{eq:-121}) if necessary, it
suffices to prove (\ref{eq:-27}) for the case when $\rho(t)$ and
$r(t)$ are simple functions of the forms $\rho(t)=\sum_{k=1}^{\ell}\rho_{k}\mathds{1}_{[\tau_{k-1},\tau_{k})}(t)$
and $r(t)=\sum_{k=1}^{\ell}r_{k}\mathds{1}_{[\tau_{k-1},\tau_{k})}(t)$,
and $\{\sigma_{t}\}_{t\ge0}$ is a simple process of the form 
\begin{equation}
\sigma_{t}=\sum_{k=1}^{\ell}\sigma_{k}(\mathbf{W}_{\tau_{0},\tau_{1}},\dots\mathbf{W}_{\tau_{k-2},\tau_{k-1}})\mathds{1}_{[\tau_{k-1},\tau_{k})}(t),\label{eq:-225}
\end{equation}
where $\mathbf{W}_{t}=(W_{t},\tilde{W}_{t})$ and $\mathbf{W}_{s,t}=\mathbf{W}_{t}-\mathbf{W}_{s}$,
and $\sigma_{k}(x_{1},\dots,x_{k-1})$ are smooth functions of $(x_{1},\dots,x_{k-1})$
with $\sigma_{\ast}\le\sigma_{k}\le\sigma^{\ast}$ and 
\begin{equation}
\Vert\sigma_{k}\Vert_{W^{d+2,\infty}}<\infty.\label{eq:-217}
\end{equation}
Let $C_{c}^{\infty}(\mathbb{R}^{d+2})$ be the space of smooth functions
with compact support on $\mathbb{R}^{d+2}$, and $Y_{s,t}=Y_{t}-Y_{s}$
and $\bar{Y}_{s,t}=\bar{Y}_{t}-\bar{Y}_{s}$. Then it suffices to
show that 
\begin{equation}
\lim_{\Delta t\to0}\mathbb{E}\Big(\prod_{k=1}^{\ell}f_{k}(Y_{\tau_{k-1},\tau_{k}})\Big)=\mathbb{E}\Big(\prod_{k=1}^{\ell}f_{k}(\bar{Y}_{\tau_{k-1},\tau_{k}})\Big),\label{eq:-23}
\end{equation}
for any $f_{k}\in C_{c}^{\infty}(\mathbb{R}^{d+2})$ with $|f_{k}|\le1$.
Let $\{\tilde{I}_{t}\}_{t\ge0}$ be the volatility target index obtained
by resetting the realized variance at times $\tau_{1},\dots,\tau_{\ell}$,
and $\tilde{Y}_{s,t}=(\log\tilde{I}_{t}-\log\tilde{I}_{s},W_{t}-W_{s},\tilde{W}_{t}-\tilde{W}_{s})$.
By Lemma \ref{lem:-8}, 
\[
\mathbb{E}\Big(\sup_{t\in[0,T]}|\log I_{t}-\log\tilde{I}_{t}|^{2}\Big)^{1/2}\le c_{\ell,\lambda,v_{0},M_{\rho},M_{r},\sigma_{\ast},\sigma^{\ast},T}\Delta t^{1/2}.
\]
Therefore, by $|f_{k}|\le1$ and the above, it is easily seen that
\begin{equation}
\begin{aligned} & \Big|\mathbb{E}\Big(\prod_{k=1}^{\ell}f_{k}(Y_{\tau_{k-1},\tau_{k}})\Big)-\mathbb{E}\Big(\prod_{k=1}^{\ell}f_{k}(\tilde{Y}_{\tau_{k-1},\tau_{k}})\Big)\Big|\\
 & \le\sum_{k=1}^{\ell}\Vert f_{k}\Vert_{W^{1,\infty}}\mathbb{E}(|Y_{\tau_{k-1},\tau_{k}}-\tilde{Y}_{\tau_{k-1},\tau_{k}}|)\le c_{\ell,\lambda,v_{0},M_{\rho},M_{r},\sigma_{\ast},\sigma^{\ast},T}\Delta t^{1/2}\sum_{k=1}^{\ell}\Vert f_{k}\Vert_{W^{1,\infty}},
\end{aligned}
\label{eq:-26}
\end{equation}
In view of (\ref{eq:-26}), it suffices to prove (\ref{eq:-23}) with
$Y_{\tau_{k-1},\tau_{k}}$ replaced by $\tilde{Y}_{\tau_{k-1},\tau_{k}}$.
By abuse of notation, we denote again by $\{I_{t}\}_{t\ge0}$ the
volatility target index obtained by realized variance reset, and by
$\{Y_{t}\}_{t\ge0}$ the corresponding joint process. Let $(\Omega^{\prime},\{\mathcal{F}_{t}^{\prime}\}_{t},\mathbb{P}^{\prime})$
be the probability space which supports the Brownian motion $\{W_{t}^{\prime}\}_{t\ge0}$.
In the rest of the proof, we will denote the product filtered probability
space $(\Omega\times\Omega^{\prime},\{\mathcal{F}_{t}\times\mathcal{F}_{t}^{\prime}\}_{t},\mathbb{P}\times\mathbb{P}^{\prime})$
by $(\Omega,\{\mathcal{F}_{t}\}_{t},\mathbb{P})$ again. Note that,
when conditional on $\mathcal{F}_{k-1}$, $\{I_{t}/I_{\tau_{k-1}}\}_{\tau_{k-1}\le t<\tau_{k}}$
is a volatility target index with initial realized variance $v_{0}$
and constant coefficients $\rho(t)=\rho_{k}$, $r(t)=r_{k}$ and $\sigma_{t}=\sigma_{k}$.
Define the random variables $g_{k}\in\mathcal{F}_{\tau_{k}}$ by $g_{\ell}=1$
and 
\begin{equation}
g_{k-1}=\mathbb{E}[g_{k}f_{k}(\bar{Y}_{\tau_{k-1},\tau_{k}})|\mathcal{F}_{\tau_{k-1}}],\quad1\le k\le\ell.\label{eq:-211}
\end{equation}
Clearly, $g_{0}=\mathbb{E}\big(\prod_{k=1}^{\ell}f_{k}(\bar{Y}_{\tau_{k-1},\tau_{k}})\big)$,
and $|g_{k}|\le1$ in view of $|f_{k}|\le1$. By (\ref{eq:-211})
and $|f_{k}|\le1$ and $|g_{k}|\le1$, we have
\begin{equation}
\begin{aligned} & \Big|\mathbb{E}\Big(g_{k}\prod_{j=1}^{k}f_{j}(Y_{\tau_{j-1},\tau_{j}})\Big)-\mathbb{E}\Big(g_{k-1}\prod_{j=1}^{k-1}f_{j}(Y_{\tau_{j-1},\tau_{j}})\Big)\Big|\\
 & =\Big|\mathbb{E}\Big(g_{k}f_{k}(Y_{\tau_{k-1},\tau_{k}})\prod_{j=1}^{k-1}f_{j}(Y_{\tau_{j-1},\tau_{j}})\Big)-\mathbb{E}\Big(g_{k}f_{k}(\bar{Y}_{\tau_{k-1},\tau_{k}})\prod_{j=1}^{k-1}f_{j}(Y_{\tau_{j-1},\tau_{j}})\Big)\Big|\\
 & =\Big|\mathbb{E}\Big([g_{k}f_{k}(Y_{\tau_{k-1},\tau_{k}})-g_{k}f_{k}(\bar{Y}_{\tau_{k-1},\tau_{k}})]\prod_{j=1}^{k-1}f_{j}(Y_{\tau_{j-1},\tau_{j}})\Big)\Big|\\
 & \le\mathbb{E}\big[\big|\mathbb{E}\big(g_{k}f_{k}(Y_{\tau_{k-1},\tau_{k}})-g_{k}f_{k}(\bar{Y}_{\tau_{k-1},\tau_{k}})|\mathcal{F}_{\tau_{k-1}}\big)\big|\big].
\end{aligned}
\label{eq:-213}
\end{equation}
Let 
\[
A=\sup_{1\le k\le\ell}\Vert\sigma_{k}\Vert_{W^{d+2,\infty}},
\]
and $B=\prod_{k=1}^{\ell}(\tau_{k}-\tau_{k-1})$. We will show that
\begin{equation}
\begin{aligned}\big|\mathbb{E}\big(g_{k}f_{k}(Y_{\tau_{k-1},\tau_{k}})-g_{k}f_{k}(\bar{Y}_{\tau_{k-1},\tau_{k}})|\mathcal{F}_{\tau_{k-1}}\big)\big|\le c_{d,\lambda,\ell,A,B,\sigma_{\ast},\sigma^{\ast},v_{0},T}\Vert f_{k}\Vert_{W^{d+2,1}}\Delta t^{1/5}(\log(T\Delta t^{-1}))^{1/2}.\end{aligned}
\label{eq:-214}
\end{equation}
Once (\ref{eq:-214}) is obtained, by (\ref{eq:-213}), we deduce
that
\[
\begin{aligned}\Big|\mathbb{E}\Big(\prod_{k=1}^{\ell}f_{k}(Y_{\tau_{k-1},\tau_{k}})\Big)-\mathbb{E}\Big(\prod_{k=1}^{\ell}f_{k}(\bar{Y}_{\tau_{k-1},\tau_{k}})\Big)\Big| & \le\sum_{k=0}^{\ell}\Big|\mathbb{E}\Big(g_{k}\prod_{j=1}^{k}f_{j}(Y_{\tau_{j-1},\tau_{j}})\Big)-\mathbb{E}\Big(g_{k-1}\prod_{j=1}^{k-1}f_{j}(Y_{\tau_{j-1},\tau_{j}})\Big)\Big|\\
 & \le c_{d,\lambda,\ell,A,B,\sigma_{\ast},\sigma^{\ast},v_{0},T}\Big(\sum_{k=0}^{\ell}\Vert f_{k}\Vert_{W^{d+2,1}}\Big)\Delta t^{1/5}(\log(T\Delta t^{-1}))^{1/2},
\end{aligned}
\]
which implies (\ref{eq:-23}) immediately. We now prove (\ref{eq:-214})
in two steps. 

\textbf{Step 1.} \emph{Each $g_{k}$ can be written as 
\begin{equation}
g_{k}=h_{k}(\xi_{k-1},Y_{\tau_{k-1},\tau_{k}})=h_{k}(\xi_{k-1},\bar{Y}_{\tau_{k-1},\tau_{k}}),\label{eq:-212}
\end{equation}
where 
\[
\xi_{k-1}=(\mathbf{W}_{\tau_{0},\tau_{1}},\dots,\mathbf{W}_{\tau_{k-2},\tau_{k-1}}),
\]
and $h_{k}$ is a smooth function with $|h_{k}|\le1$ and
\begin{equation}
\Vert h_{k}\Vert_{W^{d+2,\infty}}<c_{d,\lambda,\ell,A,\sigma_{\ast}}\prod_{j=k+1}^{\ell}(\tau_{j}-\tau_{j-1})^{-d/2-1}<\infty.\label{eq:-222}
\end{equation}
}

We prove this by induction. Suppose that (\ref{eq:-212}) holds for
$k$, and we show that (\ref{eq:-212}) holds for $k-1$. By Proposition
\ref{prop:-4}, $\bar{Y}_{\tau_{k-1},\tau_{k}}$ is a normal distribution
$\mathcal{N}((\tau_{k}-\tau_{k-1})\mu_{k}(\xi_{k-1}),(\tau_{k}-\tau_{k-1})LL^{\text{T}})$
when conditional on $\mathcal{F}_{\tau_{k-1}}$, where 
\begin{equation}
\mu_{k}(x)=\big(r_{k}+(\rho_{k}-r_{k})\bar{\sigma}\sigma_{k}(x)^{-1}U(\lambda)-2^{-1}\bar{\sigma}^{2}V(\lambda),0,\dots,0\big),\label{eq:-215}
\end{equation}
and $L$ is given by (\ref{eq:-126}). Let $\phi$ be the probability
density function of $\mathcal{N}(0,LL^{\text{T}})$ and $\phi_{t}(x)=t^{-(d+2)/2}\phi(t^{-1/2}x)$.
Then 
\begin{equation}
\begin{aligned}g_{k-1} & =\mathbb{E}(g_{k}f_{k}(\bar{Y}_{\tau_{k-1},\tau_{k}})|\mathcal{F}_{\tau_{k-1}})\\
 & =\mathbb{E}(h_{k}(\xi_{k-1},\bar{Y}_{\tau_{k-1},\tau_{k}})f_{k}(\bar{Y}_{\tau_{k-1},\tau_{k}})|\mathcal{F}_{\tau_{k-1}})=\tilde{h}_{k-1}(\xi_{k-1}),
\end{aligned}
\label{eq:-216}
\end{equation}
where
\begin{equation}
\tilde{h}_{k-1}(x)=\int_{\mathbb{R}^{d+2}}h_{k}(x,y)f_{k}(y)\phi_{\tau_{k}-\tau_{k-1}}(\mu_{k}(x)-y)dy,\label{eq:-226}
\end{equation}
with $x=(x_{1},\dots,x_{k-1})\in\mathbb{R}^{(k-1)\times(d+2)}$. Clearly,
$\tilde{h}_{k-1}$ is a smooth function. By the chain rule, we have
\[
\begin{aligned}\sum_{|\gamma|\le d+2}|\partial_{x}^{\gamma}[\phi_{\tau_{k}-\tau_{k-1}}(\mu_{k}(x)-y)]| & \le c_{d}(1+\Vert\mu_{k}\Vert_{W^{d+2,\infty}})^{d+2}\sum_{|\gamma|\le d+2}|(\partial^{\gamma}\phi_{\tau_{k}-\tau_{k-1}})(\mu_{k}(x)-y)|\\
 & \le c_{d,\sigma_{\ast}}(1+\Vert\sigma_{k}\Vert_{W^{d+2,\infty}})^{d+2}\sum_{|\gamma|\le d+2}|(\partial^{\gamma}\phi_{\tau_{k}-\tau_{k-1}})(\mu_{k}(x)-y)|\\
 & \le c_{d,A,\sigma_{\ast}}\sum_{|\gamma|\le d+2}|(\partial^{\gamma}\phi_{\tau_{k}-\tau_{k-1}})(\mu_{k}(x)-y)|.
\end{aligned}
\]
Therefore, by (\ref{eq:-222}), (\ref{eq:-226}) and the fact that
$\Vert\phi\Vert_{W^{d+2,1}}\le c_{d,\lambda}$, we have
\[
\begin{aligned}\sum_{|\gamma|\le d+2}|\partial_{x}^{\gamma}\tilde{h}_{k-1}(x)| & \le c_{d,A,\sigma_{\ast}}\Vert h_{k}\Vert_{W^{d+2,\infty}}\sum_{|\gamma|\le d+2}\int_{\mathbb{R}^{d+2}}|f_{k}(y)(\partial^{\gamma}\phi_{\tau_{k}-\tau_{k-1}})(\mu_{k}(x)-y)|dy\\
 & \le c_{d,A,\sigma_{\ast}}\Vert h_{k}\Vert_{W^{d+2,\infty}}(\tau_{k}-\tau_{k-1})^{-d/2-1}\Vert\phi\Vert_{W^{d+2,1}}\\
 & \le c_{d,\lambda,\ell,A,\sigma_{\ast}}\prod_{j=k}^{\ell}(\tau_{j}-\tau_{j-1})^{-d/2-1}.
\end{aligned}
\]
This implies that
\begin{equation}
\Vert\tilde{h}_{k-1}\Vert_{W^{d+2,\infty}}\le c_{d,\lambda,\ell,A,\sigma_{\ast}}\prod_{j=k}^{\ell}(\tau_{j}-\tau_{j-1})^{-d/2-1}.\label{eq:-218}
\end{equation}
Note that $\mathbf{W}_{\tau_{k-2},\tau_{k-1}}=\pi(Y_{\tau_{k-2},\tau_{k-1}})=\pi(\bar{Y}_{\tau_{k-2},\tau_{k-1}})$,
where $\pi$ is the projection on $\mathbb{R}^{d+2}$ given by $\pi(y_{1},y^{\prime})=y^{\prime}$
for all $y_{1}\in\mathbb{R}$ and $y^{\prime}\in\mathbb{R}^{d+1}$.
We see by (\ref{eq:-216}) that
\[
g_{k-1}=\tilde{h}_{k-1}(\xi_{k-2},\pi(Y_{\tau_{k-2},\tau_{k-1}}))=\tilde{h}_{k-1}(\xi_{k-2},\pi(\bar{Y}_{\tau_{k-2},\tau_{k-1}})).
\]
Let $h_{k-1}(x_{1},\dots,x_{k-2},x_{k-1})=\tilde{h}_{k-1}(x_{1},\dots,x_{k-2},\pi(x_{k-1}))$.
Then (\ref{eq:-212}) is satisfied, and (\ref{eq:-222}) follows readily
from (\ref{eq:-218}). This proves the conclusion of Step 1.

\textbf{Step 2.} \emph{The estimate (\ref{eq:-214}) holds for all
$k$.}

For any fixed $(x_{1},\dots,x_{k-1})\in\mathbb{R}^{(k-1)\times(d+2)}$,
by the Leibniz rule and (\ref{eq:-222}), $h_{k}(x_{1},\dots,x_{k-1},x_{k})f_{k}(x_{k})$
is a $W^{d+2,1}$ function of $x_{k}$, and 
\begin{equation}
\Vert h_{k}(x_{1},\dots,x_{k-1},x_{k})f_{k}(x_{k})\Vert_{W^{d+2,1}(x_{k})}\le c_{d,\lambda,\ell,A,\sigma_{\ast}}\Big(\prod_{j=k}^{\ell}(\tau_{j}-\tau_{j-1})\Big)^{-d/2-1}\Vert f_{k}\Vert_{W^{d+2,1}},\label{eq:-220}
\end{equation}
where $\Vert\cdot\Vert_{W^{d+2,1}(x_{k})}$ is the Sobolev norm with
respect to the variable $x_{k}$. By (\ref{eq:-212}), (\ref{eq:-220})
and Corollary \ref{cor:-2}, 
\begin{equation}
\begin{aligned} & \big|\mathbb{E}\big(g_{k}f_{k}(Y_{\tau_{k-1},\tau_{k}})-g_{k}f_{k}(\bar{Y}_{\tau_{k-1},\tau_{k}})\big|\mathcal{F}_{\tau_{k-1}}\big)\big|\\
 & =\big|\mathbb{E}\big(h_{k}(\xi_{k-1},Y_{\tau_{k-1},\tau_{k}})f_{k}(Y_{\tau_{k-1},\tau_{k}})\big|\mathcal{F}_{\tau_{k-1}}\big)-\mathbb{E}\big(h_{k}(\xi_{k-1},\bar{Y}_{\tau_{k-1},\tau_{k}})f_{k}(\bar{Y}_{\tau_{k-1},\tau_{k}})\big|\mathcal{F}_{\tau_{k-1}}\big)\big|\\
 & \le c_{d,\lambda,\ell,A,\sigma_{\ast},\sigma^{\ast},v_{0},T}\Big(\prod_{j=k}^{\ell}(\tau_{j}-\tau_{j-1})\Big)^{-d/2-1}\Vert f_{k}\Vert_{W^{d+2,1}}\Delta t^{1/5}(\log(T\Delta t^{-1}))^{1/2}.
\end{aligned}
\label{eq:-221}
\end{equation}
This proves (\ref{eq:-214}), and therefore completes the proof. 
\end{proof}
\begin{rem}
\label{rem:-2}The proof of (\ref{eq:-23}) can be significantly simplified
when $\sigma_{t}=\sigma$ is a constant\footnote{The proof for deterministic $\sigma_{t}=\sigma(t)$ follows from a
similar argument and a continuity argument. See \cite[Section 2.5]{LG26}
for details.}. This is due to the fact that, when conditional on $\mathcal{F}_{\tau_{k-1}}$
and with realized variance reset, $\log(I_{\tau_{k}}/I_{\tau_{k-1}})$
is a volatility target index independent of $\mathcal{F}_{\tau_{k-1}}$.
Consequently, by Proposition \ref{prop:-4}, for any $\epsilon>0$,
there exists a $\delta_{k}>0$ such that when $\Delta t<\delta_{k}$,
\[
|\mathbb{E}[f_{k}(Y_{\tau_{k-1},\tau_{k}})|\mathcal{F}_{\tau_{k-1}}]-\mathbb{E}[f_{k}(Y_{\tau_{k-1},\tau_{k}})|\mathcal{F}_{\tau_{k-1}}]|=|\mathbb{E}[f_{k}(Y_{\tau_{k-1},\tau_{k}})]-\mathbb{E}[f_{k}(Y_{\tau_{k-1},\tau_{k}})]|<\epsilon
\]
Setting $\Delta t<\min_{k}\delta_{k}$ completes the proof of (\ref{eq:-23})
for constant $\sigma_{t}=\sigma$. In contrast, when $\{\sigma_{t}\}_{t\ge0}$
is stochastic, the convergence rate in Corollary \ref{cor:-2} gives
a uniform control independent of $\mathcal{F}_{\tau_{k-1}}$.
\end{rem}

\subsection{\label{subsec:-7}Proof of Proposition \ref{prop:-1}}

In this subsection, we prove Proposition \ref{prop:-1}. 
\begin{proof}[Proof of Proposition \ref{prop:-1}]
We prove the proposition by contradiction. To simplify notation,
we assume that $\rho(t)=r(t)=0$ and $d=0$; the proof for general
$\rho(t),r(t)\in L^{\infty}([0,T])$ and $d\ge0$ is similar. Suppose
that $\{I_{t}^{(n)}\}_{t\ge0}$ is a subsequence of $\{I_{t}^{(N)}\}_{t\ge0}$
and $\{\tilde{X}_{t}\}_{t\ge0}$ is a process such that $\lim_{n\to\infty}I_{t}^{(n)}=\tilde{X}_{t}$
a.s. for all $t\in[0,T]$. By Corollary \ref{cor:} and Fatou's lemma,
\[
\mathbb{E}\big(|\log\tilde{X}_{t}-\log\tilde{X}_{s}|^{2}\big)^{1/2}\le c_{p,\lambda,v_{0},\sigma_{\ast},\sigma^{\ast},M_{r},M_{\rho},T}|t-s|^{1/2},
\]
which implies that $\tilde{X}_{t}$ has a continuous modification.
By abuse of notation, we denote the continuous modification of $\{\tilde{X}_{t}\}_{t\ge0}$
by $\{\tilde{X}_{t}\}_{t\ge0}$ itself. Let $\{\tilde{\mathcal{F}}_{t}\}_{t\ge0}$
be the completion of $\{\mathcal{F}_{t}\}_{t\ge0}$. Clearly, $\tilde{X}_{t}\in\tilde{\mathcal{F}}_{t}$
for each $t\in[0,T]$. In view of the continuity of $\tilde{X}_{t}$,
we see that $\{\tilde{X}_{t}\}_{t\ge0}$ is $\{\tilde{\mathcal{F}}_{t}\}_{t\ge0}$-adapted.
Note that 
\begin{equation}
\sup_{n\ge1}\mathbb{E}\big(\sup_{t\in[0,T]}|\log I_{t}^{(n)}|^{p}\big)<\infty,\label{eq:-232}
\end{equation}
for all $p>0$, which implies that $\log I_{t}^{(n)}$ is uniformly
integrable for each $t$, and therefore, 
\[
\mathbb{E}(\log\tilde{X}_{t}-\log\tilde{X}_{s})=\lim_{n\to\infty}\mathbb{E}(\log I_{t}^{(n)}-\log I_{s}^{(n)})=-2^{-1}\bar{\sigma}^{2}V(\lambda)(t-s),
\]
where the last equality follows from Theorem \ref{thm:}. Therefore,
$\log\tilde{X}_{t}+2^{-1}\bar{\sigma}^{2}V(\lambda)t$ is an $\{\tilde{\mathcal{F}}_{t}\}_{t\ge0}$-adapted
martingale. By the martingale representation theorem, there exists
a progressively measurable process $\{\beta_{t}\}_{t\ge0}$ such that
\begin{equation}
\log\tilde{X}_{t}=-2^{-1}\bar{\sigma}^{2}V(\lambda)t+\int_{0}^{t}\beta_{s}dW_{s}.\label{eq:-11}
\end{equation}

Let $\mathcal{H}$ be the space of all continuous $L^{2}$-integrable
processes $\{Y_{t}\}_{t\ge0}$ of the form 
\begin{equation}
Y_{t}=A_{t}+M_{t},\quad a.s.,\label{eq:-8}
\end{equation}
where $A_{t}=\int_{0}^{t}a_{s}ds$ is a bounded variation process
satisfying $\mathbb{E}\big(\int_{0}^{T}|a_{s}|^{2}ds\big)<\infty$,
and $\{M_{t}\}_{t\ge0}$ is an $L^{2}$-integrable continuous martingale.
By the Doob--Meyer decomposition theorem, the decomposition (\ref{eq:-8})
is unique for each $\{Y_{t}\}_{t\ge0}\in\mathcal{H}$. For any $Y,Y^{\prime}\in\mathcal{H}$
with Doob--Meyer decomposition $Y=\int a_{s}ds+M$ and $Y^{\prime}=\int a_{s}^{\prime}ds+M^{\prime}$,
define 
\begin{equation}
(Y,Y^{\prime})_{\mathcal{H}}=\mathbb{E}\Big(\int_{0}^{T}a_{t}a_{t}^{\prime}dt\Big)+\mathbb{E}\big(\langle M,M^{\prime}\rangle_{T}\big).\label{eq:-9}
\end{equation}
Clearly, $\mathcal{H}$ is a Hilbert space with the inner product
$(\cdot,\cdot)_{\mathcal{H}}$, and therefore satisfies the Banach--Saks
property (see for example \cite{Kakutani39}); that is, every bounded
sequence $\{Y^{(n)}\}\subseteq\mathcal{H}$ contains a subsequence
$\{Y^{(n_{k})}\}$ such that the Ces\`{a}ro mean $\frac{1}{k}\sum_{j=1}^{k}Y^{(n_{j})}$
converges strongly in $\mathcal{H}$. It is easily seen by (\ref{eq:-86}),
It\^{o}'s formula and H\"{o}lder's inequality that
\[
\Vert\log I^{(n)}\Vert_{\mathcal{H}}^{2}\le c_{M_{\rho},M_{r},\sigma^{\ast}}\frac{1}{n}\sum_{k=0}^{n-1}\big(\mathbb{E}(v_{k}^{-1/2})+\mathbb{E}(v_{k}^{-2})\big)
\]
for some constant $c_{M_{\rho},M_{r},\sigma^{\ast}}$ depending only
on $M_{\rho}$, $M_{r}$ and $\sigma^{\ast}$. It follows from the
above and Lemma \ref{lem:-4} that $\sup_{n}\Vert\log I^{(n)}\Vert_{\mathcal{H}}<\infty$.
Therefore, there exists a subsequence $\{\log I^{(n_{k})}\}_{k}$
such that the strong limit $\lim_{k\to\infty}\frac{1}{k}\sum_{j=1}^{k}\log I^{(n_{j})}$
exists in $\mathcal{H}$. Note that, by H\"{o}lder's inequality and
Doob's maximal inequality, 
\[
\mathbb{E}\Big(\sup_{0\le t\le T}|Y_{t}|^{2}\Big)\le2\mathbb{E}\Big[\Big(\int_{0}^{T}|a_{t}|dt\Big)^{2}+\sup_{0\le t\le T}M_{t}^{2}\Big]\le c\Vert Y\Vert_{\mathcal{H}}^{2},\quad Y\in\mathcal{H}.
\]
Therefore, convergence of $\frac{1}{k}\sum_{j=1}^{k}\log I^{(n_{j})}$
in $\mathcal{H}$ implies convergence of $\frac{1}{k}\sum_{j=1}^{k}\log I_{t}^{(n_{j})}$
in probability for each $t$. Consequently, $\{\log\tilde{X}_{t}\}_{t\ge0}$
coincides with the strong limit $\lim_{k\to\infty}\frac{1}{k}\sum_{j=1}^{k}\log I^{(n_{j})}$
in $\mathcal{H}$; that is, 
\begin{equation}
\lim_{k\to\infty}\Big\Vert\frac{1}{k}\sum_{j=1}^{k}\log I^{(n_{j})}-\log\tilde{X}\Big\Vert_{\mathcal{H}}=0.\label{eq:-12}
\end{equation}

For any $\tau\in[0,T]$, let $W_{t}^{\tau}=W_{t\wedge\tau}$. By It\^{o}'s
formula, 
\begin{equation}
\begin{aligned}\lim_{n\to\infty}(\log I^{(n)},W^{\tau})_{\mathcal{H}} & =\lim_{n\to\infty}\sum_{k=1}^{n}\mathbb{E}\Big(\bar{\sigma}v_{k-1}^{-1/2}\int_{t_{k-1}}^{t_{k}}\sigma d\langle W,W^{\tau}\rangle_{u}\Big)\\
 & =\lim_{n\to\infty}\sum_{k=1}^{n}\mathbb{E}\Big(\bar{\sigma}v_{k-1}^{-1/2}\Big)\int_{t_{k-1}\wedge\tau}^{t_{k}\wedge\tau}\sigma du=\bar{\sigma}U(\lambda)\tau.
\end{aligned}
\label{eq:-13}
\end{equation}
By (\ref{eq:-11}), (\ref{eq:-12}) and (\ref{eq:-13}), we have 
\[
\int_{0}^{\tau}\mathbb{E}(\beta_{t})dt=(\log\tilde{X},W^{\tau})_{\mathcal{H}}=\lim_{k\to\infty}\frac{1}{k}\sum_{j=1}^{k}(\log I^{(n_{j})},W^{\tau})_{\mathcal{H}}=\bar{\sigma}U(\lambda)\tau.
\]
Therefore, 
\begin{equation}
\mathbb{E}(\beta_{t})=\bar{\sigma}U(\lambda),\quad t\in[0,T].\label{eq:-14}
\end{equation}
On the other hand, by (\ref{eq:-232}) and Theorem \ref{thm:}, we
have 
\[
\mathbb{E}[(\log\tilde{X}_{t}+2^{-1}\bar{\sigma}^{2}V(\lambda)t)^{2}]=\lim_{n\to\infty}\mathbb{E}[(\log I_{t}^{(n)}+2^{-1}\bar{\sigma}^{2}V(\lambda)t)^{2}]=\bar{\sigma}^{2}V(\lambda)t,
\]
which implies that 
\begin{equation}
\beta_{t}^{2}=\bar{\sigma}^{2}V(\lambda),\quad t\in[0,T]\;.a.s.\label{eq:-15}
\end{equation}
For any positive progressively measurable process $\{\xi_{t}\}_{t\ge0}$,
we have $\big(\log I^{(n)},\int\xi_{t}dW_{t}\big)_{\mathcal{H}}\ge0$.
By (\ref{eq:-12}), we deduce that 
\[
\int_{0}^{T}\mathbb{E}(\beta_{t}\xi_{t})dt=(\log\tilde{X},W^{\tau})_{\mathcal{H}}=\lim_{k\to\infty}\frac{1}{k}\sum_{j=1}^{k}(\log I^{(n_{j})},W^{\tau})_{\mathcal{H}}\ge0,
\]
for any positive progressively measurable process $\{\xi_{t}\}_{t\ge0}$.
This implies that $\beta_{t}\ge0$ for all $t\in[0,T]$ a.s. Therefore,
by (\ref{eq:-15}), we see that 
\begin{equation}
\beta_{t}=\bar{\sigma}\sqrt{V(\lambda)},\quad t\in[0,T]\;.a.s.\label{eq:-34}
\end{equation}
Now (\ref{eq:-14}) contradicts with (\ref{eq:-34}) in view of the
fact that $U(\lambda)^{2}<V(\lambda)$. This completes the proof.
\end{proof}

\section{\label{sec:-2}The Discrete Correction}

In this section, we derive the discrete correction to the diffusion
(\ref{eq:}) when $\sigma_{t}=\sigma(t)$ is a deterministic $C^{1,1}([0,T])$
function. As we will see later in this section, the discrete correction
formula can be regarded as a quantification of the competing effect
of $\lambda\to1$ and $\Delta t\to0$, or more precisely, the size
of $\delta(\Delta t,\lambda)=(1-\lambda)^{-1}\Delta t$ (see \cite[Section 4]{LG26}
for more details). 

Throughout this section, we will denote 
\begin{equation}
P_{k}(t)=\prod_{j=0}^{k-1}(1+t\lambda^{j})^{-1/2},\quad P_{\infty}(t)=\prod_{j=0}^{\infty}(1+t\lambda^{j})^{-1/2},\label{eq:-63}
\end{equation}
and denote by $V_{k}(\lambda,\theta)$ the function defined in (\ref{eq:-5});
that is, 
\[
V_{k}(\lambda,\theta)=\frac{1}{2(1-\lambda)}\int_{0}^{\infty}e^{-t\theta(1-\lambda)^{-1}\lambda^{k}/2}P_{k}(t)dt.
\]

\subsection{\label{subsec:-3}Estimates for $V_{k}(\lambda,\theta)$}

In this subsection, we establish estimates for $V_{k}(\lambda,\theta)$
and its derivatives, which will be needed for the determination of
order of terms in the discrete correction formula. 
\begin{lem}
\label{lem:-21}For any integer $n\ge1$ and any $\gamma\in[0,1)$,
we have 
\begin{equation}
\int_{0}^{\infty}t^{n-1}P_{\infty}(t)\frac{dt}{t^{\gamma}}\le c_{n,\gamma}\lambda^{-n^{2}}(1-\lambda)^{n-\gamma},\label{eq:-229}
\end{equation}
where $c_{n,\gamma}$ is a constant depending only on $n$ and $\gamma$.
\end{lem}

\begin{proof}
We first prove (\ref{eq:-229}) for $\gamma=0$. Let $\tilde{P}_{\infty}(t)=\prod_{j=0}^{\infty}(1+t\lambda^{2j})^{-1/2}$.
By 
\[
\begin{aligned}\prod_{k=0}^{\infty}(1+t\lambda^{k})^{-1/2} & =\prod_{k=0}^{\infty}[(1+t\lambda^{k})(1+t\lambda^{2k+1})]^{-1/2}\le\tilde{P}_{\infty}(t\lambda),\end{aligned}
\]
we have
\begin{equation}
\begin{aligned}\int_{0}^{\infty}t^{n-\gamma-1}\prod_{k=0}^{\infty}(1+t\lambda^{k})^{-1/2}dt & \le\int_{0}^{\infty}t^{n-1}\tilde{P}_{\infty}(t\lambda)dt=\lambda^{-n}\int_{0}^{\infty}t^{n-1}\tilde{P}_{\infty}(t)dt.\end{aligned}
\label{eq:-228}
\end{equation}
Let $I_{n}=\int_{0}^{\infty}t^{n}\tilde{P}_{\infty}(t)dt$. Then $(1+t)\tilde{P}_{\infty}(t)=\tilde{P}(t\lambda^{2})$
and therefore, 
\[
\begin{aligned}I_{n} & =\int_{0}^{\infty}t^{n-1}(1+t-1)\tilde{P}_{\infty}(t)dt\\
 & =\int_{0}^{\infty}t^{n-1}\tilde{P}_{\infty}(t\lambda^{2})dt-\int_{0}^{\infty}t^{n-1}\tilde{P}_{\infty}(t)dt\\
 & =(\lambda^{-2n}-1)I_{n-1}.
\end{aligned}
\]
By induction, we obtain that 
\[
I_{n}=I_{0}\prod_{k=1}^{n}(\lambda^{-2k}-1)=I_{0}\lambda^{-n(n+1)}\prod_{k=1}^{n}(1-\lambda^{2k})\le(2n)!!\lambda^{-n(n+1)}(1-\lambda)^{n}I_{0}.
\]
Since $I_{0}=2\log\lambda^{-1}\le2(\lambda^{-1}-1)$ (see \cite[Lemma 3.7]{LG26}),
we deduce that  
\[
I_{n}\le2(2n)!!\lambda^{-n^{2}-n-1}(1-\lambda)^{n+1},
\]
which, together with (\ref{eq:-228}), completes the proof for $\gamma=0$.

For $\gamma\in(0,1)$, we have 
\[
\begin{aligned}\int_{0}^{\infty}t^{n-1}P_{\infty}(t)\frac{dt}{t^{\gamma}} & \le\int_{0}^{1-\lambda}t^{n-1}\frac{dt}{t^{\gamma}}+\int_{1-\lambda}^{\infty}t^{n-1}P_{\infty}(t)\frac{dt}{t^{\gamma}}\\
 & \le\frac{(1-\lambda)^{n-\gamma}}{n-\gamma}+(1-\lambda)^{-\gamma}\int_{1-\lambda}^{\infty}t^{n-1}P_{\infty}(t)dt\\
 & \le\frac{(1-\lambda)^{n-\gamma}}{n-\gamma}+c_{n}(1-\lambda)^{n-\gamma},
\end{aligned}
\]
where the last inequality follows from (\ref{eq:-229}) for $\gamma=0$.
This completes the proof.
\end{proof}
\begin{rem}
\label{rem:-11}We would like to make a comment on the role of $\gamma$
in (\ref{eq:-229}); the same consideration applies to several estimates
in this subsection. The case for $\gamma=0$ corresponds to estimates
related to the functions $V(\lambda)$ and $V_{k}(\lambda,\theta)$,
while $\gamma=1/2$ corresponds to the functions $U(\lambda)$ and
\begin{equation}
U_{k}(\lambda,\theta)=\sqrt{\frac{2}{\pi(1-\lambda)}}\int_{0}^{\infty}e^{-t^{2}\theta(1-\lambda)^{-1}\lambda^{k}/2}P_{k}(t^{2})dt.\label{eq:-20}
\end{equation}
This is because, by a change of variable, 
\[
U(\lambda)=\frac{1}{\sqrt{2\pi(1-\lambda)}}\int_{0}^{\infty}P_{\infty}(t)\frac{dt}{t^{1/2}},
\]
and 
\[
U_{k}(\lambda,\theta)=\frac{1}{\sqrt{2\pi(1-\lambda)}}\int_{0}^{\infty}e^{-t\theta(1-\lambda)^{-1}\lambda^{k}/2}P_{k}(t)\frac{dt}{t^{1/2}}.
\]
\end{rem}

\begin{cor}
\label{cor:-4}Let $V(\lambda)$ be the function given in (\ref{eq:-3}).
Then we have 
\begin{equation}
V(\lambda)\le c_{0}\lambda^{-1},\label{eq:-280}
\end{equation}
\begin{equation}
V^{\prime}(\lambda)\le c_{0}\lambda^{-5}(1-\lambda)^{-1},\label{eq:-281}
\end{equation}
and 
\begin{equation}
V^{\prime\prime}(\lambda)\le c_{0}\lambda^{-11}(1-\lambda)^{-2},\label{eq:-282}
\end{equation}
where $c_{0}$ is a universal constant.
\end{cor}

\begin{proof}
By Lemma \ref{lem:-21}, 
\[
V(\lambda)=2^{-1}(1-\lambda)^{-1}\int_{0}^{\infty}P_{\infty}(t)dt\le c_{0}\lambda^{-1}.
\]

For $V^{\prime}(\lambda)$, by the dominated convergence theorem,
\begin{equation}
\begin{aligned}V^{\prime}(\lambda) & =\frac{\partial}{\partial\lambda}\Big(\int_{0}^{\infty}P_{\infty}(2t(1-\lambda))dt\Big)\\
 & =\int_{0}^{\infty}\frac{\partial}{\partial\lambda}[P_{\infty}(2t(1-\lambda))]dt\\
 & =-\int_{0}^{\infty}\sum_{k=0}^{\infty}\frac{t[k\lambda^{k-1}(1-\lambda)-\lambda^{k}]}{1+2t(1-\lambda)\lambda^{k}}P_{\infty}(2t(1-\lambda))dt.
\end{aligned}
\label{eq:-283}
\end{equation}
Since 
\begin{equation}
\sum_{k=0}^{\infty}\frac{t|k\lambda^{k-1}(1-\lambda)-\lambda^{k}|}{1+2t(1-\lambda)\lambda^{k}}\le2t\Big((1-\lambda)\sum_{k=1}^{\infty}k\lambda^{k-1}+\sum_{k=0}^{\infty}\lambda^{k}\Big)\le c_{0}t\lambda^{-1}(1-\lambda)^{-1},\label{eq:-276}
\end{equation}
we deduce that 
\[
\begin{aligned}|V^{\prime}(\lambda)|\le\lambda^{-1}(1-\lambda)^{-1}\int_{0}^{\infty}tP_{\infty}(2t(1-\lambda))dt\le c_{0}\lambda^{-5}(1-\lambda)^{-1},\end{aligned}
\]
where the last inequality follows from Lemma \ref{lem:-21} and a
change of variable.

For $V^{\prime\prime}(\lambda)$, by the dominated convergence theorem,
\[
\begin{aligned}V^{\prime\prime}(\lambda) & =-\int_{0}^{\infty}\frac{\partial}{\partial\lambda}\Big(\sum_{k=0}^{\infty}\frac{t[k\lambda^{k-1}(1-\lambda)-\lambda^{k}]}{1+2t(1-\lambda)\lambda^{k}}P_{\infty}(2t(1-\lambda))\Big)dt\\
 & =E_{1}+E_{2}+E_{3},
\end{aligned}
\]
where 
\begin{align*}
E_{1} & =\int_{0}^{\infty}\Big(\sum_{j=0}^{\infty}\frac{t[j\lambda^{j-1}(1-\lambda)-\lambda^{j}]}{1+2t(1-\lambda)\lambda^{j}}\Big)^{2}P_{\infty}(2t(1-\lambda))dt,\\
E_{2} & =-\int_{0}^{\infty}\sum_{j=0}^{\infty}\frac{t[j^{2}\lambda^{j-2}(1-\lambda)-2j\lambda^{j-1}]}{1+2t(1-\lambda)\lambda^{j}}P_{\infty}(2t(1-\lambda))dt,\\
E_{3} & =\int_{0}^{\infty}\sum_{j=0}^{\infty}\frac{t^{2}[j\lambda^{j-1}(1-\lambda)-\lambda^{j}]^{2}}{[1+2t(1-\lambda)\lambda^{j}]^{2}}P_{\infty}(2t(1-\lambda))dt.
\end{align*}
For $E_{1}$, by (\ref{eq:-276}) and Lemma \ref{lem:-21},
\begin{equation}
|E_{1}|\le c_{0}\lambda^{-2}(1-\lambda)^{-2}\int_{0}^{\infty}t^{2}P_{\infty}(2t(1-\lambda))dt\le c_{0}\lambda^{-11}(1-\lambda)^{-2}.\label{eq:-277}
\end{equation}
For $E_{2}$, note that
\[
\sum_{j=0}^{\infty}\frac{t|j^{2}\lambda^{j-2}(1-\lambda)-2j\lambda^{j-1}|}{1+2t(1-\lambda)\lambda^{j}}\le t\Big(\sum_{j=0}^{\infty}j^{2}\lambda^{j-2}(1-\lambda)+2\sum_{j=0}^{\infty}j\lambda^{j-1}\Big)\le c_{0}t\lambda^{-2}(1-\lambda)^{-2}.
\]
Therefore, by Lemma \ref{lem:-21},
\begin{equation}
|E_{2}|\le c_{0}\lambda^{-2}(1-\lambda)^{-2}\int_{0}^{\infty}tP_{\infty}(2t(1-\lambda))dt\le c_{0}\lambda^{-6}(1-\lambda)^{-2}.\label{eq:-278}
\end{equation}
For $E_{3}$, we have
\[
\sum_{j=0}^{\infty}\frac{t^{2}[j\lambda^{j-1}(1-\lambda)-\lambda^{j}]^{2}}{[1+2t(1-\lambda)\lambda^{j}]^{2}}\le c_{0}t^{2}\sum_{j=0}^{\infty}[j^{2}\lambda^{2j-2}(1-\lambda)^{2}+\lambda^{2j}]\le c_{0}t^{2}\lambda^{-2}(1-\lambda)^{-1}.
\]
Therefore, by Lemma \ref{lem:-21},
\begin{equation}
|E_{3}|\le c_{0}\lambda^{-2}(1-\lambda)^{-1}\int_{0}^{\infty}t^{2}P_{\infty}(2t(1-\lambda))dt\le c_{0}\lambda^{-11}(1-\lambda)^{-1}.\label{eq:-279}
\end{equation}
Combining (\ref{eq:-277}), (\ref{eq:-278}) and (\ref{eq:-279})
completes the proof.
\end{proof}
\begin{rem}
\label{rem:-9}We would like to point out that, though the estimates
(\ref{eq:-281}) and (\ref{eq:-282}) will be sufficient for the purpose
of establishing the discrete correction, they can in fact be improved
using the cancellation property of the integrand. We briefly demonstrate
this for $V^{\prime}(\lambda)$. Note that
\[
\begin{aligned}\sum_{k=0}^{\infty}\frac{t[k\lambda^{k-1}-(k+1)\lambda^{k}]}{1+2t(1-\lambda)\lambda^{k}} & =\sum_{k=0}^{\infty}\frac{tk\lambda^{k-1}}{1+2t(1-\lambda)\lambda^{k}}-\sum_{k=1}^{\infty}\frac{tk\lambda^{k-1}}{1+2t(1-\lambda)\lambda^{k-1}}\\
 & =\sum_{k=1}^{\infty}\frac{2t^{2}k\lambda^{2k-2}(1-\lambda)^{2}}{[1+2t(1-\lambda)\lambda^{k}][1+2t(1-\lambda)\lambda^{k-1}]}.
\end{aligned}
\]
Therefore, the integrand for (\ref{eq:-283}) is dominated by $(1-\lambda)^{2}\big(\sum_{k=1}^{\infty}k\lambda^{2k-2}\big)t^{2}P_{\infty}(2t(1-\lambda))\le c_{0}\lambda^{-2}t^{2}P_{\infty}(2t(1-\lambda))$.
Now it follows from Lemma \ref{lem:-21} that 
\[
|V^{\prime}(\lambda)|\le c_{0}\lambda^{-11}.
\]
\end{rem}

\begin{lem}
\label{lem}For any $\theta>0$ and any $\gamma\in[0,1)$, we have
\begin{equation}
\sum_{k=0}^{\infty}\int_{0}^{\infty}e^{-t\theta(1-\lambda)^{-1}\lambda^{k}}t^{n}[P_{k}(t)-P_{\infty}(t)]\frac{dt}{t^{\gamma}}\le c_{n,\gamma}(1+\theta^{-n})\lambda^{-(n+1)^{2}}(1-\lambda)^{n-\gamma},\label{eq:-64}
\end{equation}
where $c_{n,\gamma}$ is a constant depending only on $n$ and $\gamma$. 
\end{lem}

\begin{proof}
We denote by $c_{n,\gamma}$ a constant depending only on $n$ and
$\gamma$, of which the value may vary at different appearances. Let
\[
I_{k}=\int_{0}^{\infty}e^{-t\theta(1-\lambda)^{-1}\lambda^{k}}t^{n}[P_{k}(t)-P_{\infty}(t)]\frac{dt}{t^{\gamma}},
\]
and $\tilde{P}_{k}(t)=\prod_{j=1}^{k}(1+t\lambda^{-j})^{-1/2}$ and
$\tilde{P}_{0}(t)=1$. Clearly, $P_{k}(t\lambda^{-k})=\tilde{P}_{k}(t)$
and $P_{\infty}(t\lambda^{-k})=\tilde{P}_{k}(t)P_{\infty}(t)$. Therefore,
by a change of variable, 
\begin{equation}
\begin{aligned}I_{k} & =\lambda^{-(n+1)k}\int_{0}^{\infty}e^{-t\theta(1-\lambda)^{-1}}t^{n}[P_{k}(t\lambda^{-k})-P_{\infty}(t\lambda^{-k})]\frac{dt}{t^{\gamma}}\\
 & =\lambda^{-(n+1)k}\int_{0}^{\infty}e^{-t\theta(1-\lambda)^{-1}}t^{n}\tilde{P}_{k}(t)[1-P_{\infty}(t)]\frac{dt}{t^{\gamma}}.
\end{aligned}
\label{eq:-48}
\end{equation}
By $\tilde{P}_{2k+1}(t)\le\tilde{P}_{2k}(t)$, we have $I_{2k+1}\le\lambda^{-1}I_{2k}$
and therefore, 
\[
\sum_{k=0}^{\infty}I_{k}\le2\lambda^{-1}\int_{0}^{\infty}e^{-t\theta(1-\lambda)^{-1}}t^{n}\sum_{k=0}^{\infty}\lambda^{-2(n+1)k}\tilde{P}_{2k}(t)[1-P_{\infty}(t)]\frac{dt}{t^{\gamma}}.
\]
Let $\tilde{Q}_{k}(t)=\prod_{j=1}^{k}(1+t\lambda^{-2j})^{-1}$ and
$\tilde{Q}_{0}(t)=1$. It is easily seen that $\tilde{P}_{2k}(t)\le\tilde{Q}_{k}(t\lambda)$.
Therefore, 
\begin{equation}
\begin{aligned}\sum_{k=0}^{\infty}I_{k} & \le2\lambda^{-1}\int_{0}^{\infty}e^{-t\theta(1-\lambda)^{-1}}t^{n}\sum_{k=0}^{\infty}\lambda^{-2(n+1)k}\tilde{Q}_{k}(t\lambda)[1-P_{\infty}(t)]\frac{dt}{t^{\gamma}}\\
 & =2\lambda^{-n-1}\int_{0}^{\infty}e^{-t\theta(1-\lambda)^{-1}}S_{n+1}(t\lambda)[1-P_{\infty}(t)]\frac{dt}{t^{1+\gamma}},
\end{aligned}
\label{eq:-235}
\end{equation}
where 
\begin{equation}
S_{n}(t)=t^{n}\sum_{k=0}^{\infty}\lambda^{-2nk}\tilde{Q}_{k}(t),\quad n\ge1.\label{eq:-236}
\end{equation}

We show that 
\begin{equation}
S_{n+1}(t)\le c_{n}\lambda^{-n(n+1)}(1-\lambda+t)^{n},\label{eq:-290}
\end{equation}
For $S_{1}(t)$, we have
\[
\begin{aligned}S_{1}(t) & =t+\sum_{k=1}^{\infty}(1+t\lambda^{-2k}-1)\tilde{Q}_{k}(t\lambda)=t+\sum_{k=1}^{\infty}[\tilde{Q}_{k-1}-\tilde{Q}_{k}(t)]=t+1.\end{aligned}
\]
A recursion for $S_{n}(t)$ can be obtained similarly. We have 
\begin{equation}
\begin{aligned}S_{n+1}(t) & =t^{n+1}+t^{n}\sum_{k=1}^{\infty}\lambda^{-2nk}(1+t\lambda^{-2k}-1)\tilde{Q}_{k}(t)\\
 & =t^{n+1}+t^{n}\sum_{k=1}^{\infty}\lambda^{-2nk}[\tilde{Q}_{k-1}(t)-\tilde{Q}_{k}(t)]\\
 & =t^{n+1}+t^{n}\lambda^{-2n}\sum_{k=1}^{\infty}\lambda^{-2n(k-1)}\tilde{Q}_{k-1}(t)-S_{n}(t)\\
 & =t^{n+1}+(\lambda^{-2n}-1)S_{n}(t).
\end{aligned}
\label{eq:-238}
\end{equation}
By (\ref{eq:-238}) and $S_{1}(t)=t+1$, we obtain that 
\begin{equation}
S_{n+1}(t)=\sum_{k=1}^{n}\Big(\prod_{j=k+1}^{n}(\lambda^{-2j}-1)\Big)t^{k+1}+\Big(\prod_{j=1}^{n}(\lambda^{-2j}-1)\Big)(t+1).\label{eq:-239}
\end{equation}
Let $[n]_{q}!=\prod_{j=1}^{n}\frac{1-q^{j}}{1-q}$ be the $q$-factorial.
The above can be written as 
\[
S_{n+1}(t)=\sum_{k=1}^{n}\lambda^{-(n+k+1)(n-k)}\frac{[n]_{\lambda^{2}}!}{[k]_{\lambda^{2}}!}(1-\lambda^{2})^{n-k}t^{k+1}+\lambda^{-n(n+1)}[n]_{\lambda^{2}}!(1-\lambda^{2})^{n}(t+1).
\]
Since $[n]_{q}!/[k]_{q}!$ is increasing in $q\in(0,1)$ and $[n]_{q}!/[k]_{q}!\to n!/k!$
as $q\to1$, we see that $[n]_{q}!/[k]_{q}!\le n!/k!$ for $0<q<1$.
Therefore, 
\[
\begin{aligned}S_{n+1}(t) & \le\lambda^{-n(n+1)}\sum_{k=1}^{n}\frac{n!}{k!}(1-\lambda^{2})^{n-k}t^{k+1}+\lambda^{-n(n+1)}n!(1-\lambda^{2})^{n}(t+1)\\
 & \le c_{n}\lambda^{-n(n+1)}\Big(\sum_{k=0}^{n}\binom{n}{k}(1-\lambda)^{n-k}t^{k+1}+(1-\lambda)^{n}\Big)\\
 & \le c_{n}\lambda^{-n(n+1)}(1-\lambda+t)^{n}.
\end{aligned}
\]

By (\ref{eq:-235}) and (\ref{eq:-290}), we see that
\begin{equation}
\begin{aligned}\sum_{k=0}^{\infty}I_{k} & \le c_{n}\lambda^{-(n+1)^{2}}\int_{0}^{\infty}e^{-t\theta(1-\lambda)^{-1}}(1-\lambda+t\lambda)^{n}[1-P_{\infty}(t)]\frac{dt}{t^{1+\gamma}}\\
 & \le c_{n}\lambda^{-(n+1)^{2}}\Big(\int_{0}^{1-\lambda}(1-\lambda+t\lambda)^{n}[1-P_{\infty}(t)]\frac{dt}{t^{1+\gamma}}+\int_{1-\lambda}^{\infty}e^{-t\theta(1-\lambda)^{-1}}(1-\lambda+t\lambda)^{n}\frac{dt}{t^{1+\gamma}}\Big).
\end{aligned}
\label{eq:-62}
\end{equation}
For the second integral, by a change of variable, 
\begin{equation}
\begin{aligned}\int_{1-\lambda}^{\infty}e^{-t\theta(1-\lambda)^{-1}}(1-\lambda+t\lambda)^{n}\frac{dt}{t^{1+\gamma}} & =(1-\lambda)^{n-\gamma}\int_{1}^{\infty}e^{-t\theta}(1+t\lambda)^{n}\frac{dt}{t^{1+\gamma}}\\
 & \le(1-\lambda)^{n-\gamma}\int_{1}^{\infty}e^{-t\theta}t^{n-1-\gamma}dt\le c_{n,\gamma}\theta^{-n}(1-\lambda)^{n}.
\end{aligned}
\label{eq:-219}
\end{equation}
We now estimate the first integral in (\ref{eq:-62}). Note that $\lim_{k\to\infty}P_{\infty}(t\lambda^{k})=1$.
By the monotone convergence theorem and that
\[
P_{\infty}(t\lambda^{k})=(1+t\lambda^{k})^{-1/2}P_{\infty}(t\lambda^{k+1}),
\]
we have 
\[
\begin{aligned}\int_{0}^{1-\lambda}(1-\lambda+t\lambda)^{n}[1-P_{\infty}(t)]\frac{dt}{t^{1+\gamma}} & =\sum_{k=1}^{\infty}\int_{0}^{1-\lambda}(1-\lambda+t\lambda)^{n}[P_{\infty}(t\lambda^{k+1})-P_{\infty}(t\lambda^{k})]\frac{dt}{t^{1+\gamma}}\\
 & \le c_{n}(1-\lambda)^{n}\sum_{k=1}^{\infty}\int_{0}^{1-\lambda}[1-(1+t\lambda^{k})^{-1/2}]P_{\infty}(t\lambda^{k+1})\frac{dt}{t^{1+\gamma}}\\
 & \le c_{n}(1-\lambda)^{n}\sum_{k=1}^{\infty}\lambda^{k}\int_{0}^{1-\lambda}P_{\infty}(t\lambda^{k+1})\frac{dt}{t^{\gamma}},
\end{aligned}
\]
where the last inequality is due to Bernoulli's inequality. Since
$P_{\infty}(t)\le1$, we deduce that 
\begin{equation}
\int_{0}^{1-\lambda}(1-\lambda+t\lambda)^{n}[1-P_{\infty}(t)]\frac{dt}{t^{1+\gamma}}\le c_{n,\gamma}(1-\lambda)^{n}\sum_{k=1}^{\infty}(1-\lambda)^{1-\gamma}\lambda^{k}\le c_{n,\gamma}(1-\lambda)^{n-\gamma}.\label{eq:-227}
\end{equation}
Combining (\ref{eq:-62}), (\ref{eq:-219}) and (\ref{eq:-227}) completes
the proof.
\end{proof}
\begin{lem}
\label{lem:-23}For any $m\ge0$, $n\ge0$, $r\ge0$ and any $\gamma\in[0,1)$,
we have
\begin{equation}
\sum_{k=r}^{\infty}k^{m}\lambda^{k}\int_{0}^{\infty}e^{-t\theta(1-\lambda)^{-1}\lambda^{k}}t^{n}P_{k}(t)\frac{dt}{t^{\gamma}}\le c_{m,n,\gamma}(1+\theta^{-n})(1+r)^{m}\lambda^{r-(n+1)^{2}}(1-\lambda)^{n-m-\gamma},\label{eq:-246}
\end{equation}
where $c_{m,n}$ is a constant depending only on $m$, $n$ and $\gamma$. 
\end{lem}

\begin{proof}
We denote by $c_{n}$ a constant depending only on $n$ of which the
value may vary at different appearances. Let $S_{0}=0$,
\[
S_{k}=\sum_{j=1}^{k}\int_{0}^{\infty}e^{-t\theta(1-\lambda)^{-1}\lambda^{j}}t^{n}[P_{j}(t)-P_{\infty}(t)]\frac{dt}{t^{\gamma}},
\]
and $S_{\infty}=\lim_{k\to\infty}S_{k}$. Let $R_{r}=\sum_{k=r}^{\infty}k^{m}\lambda^{k}$.
Then
\[
\begin{aligned} & \sum_{k=r}^{\infty}k^{m}\lambda^{k}\int_{0}^{\infty}e^{-t\theta(1-\lambda)^{-1}\lambda^{k}}t^{n}P_{k}(t)\frac{dt}{t^{\gamma}}\\
 & \le\sum_{k=r}^{\infty}k^{m}\lambda^{k}\int_{0}^{\infty}e^{-t\theta(1-\lambda)^{-1}\lambda^{k}}t^{n}[P_{k}(t)-P_{\infty}(t)]\frac{dt}{t^{\gamma}}+\sum_{k=r}^{\infty}k^{m}\lambda^{k}\int_{0}^{\infty}t^{n}P_{\infty}(t)\frac{dt}{t^{\gamma}}\\
 & \le\sum_{k=r}^{\infty}k^{m}\lambda^{k}(S_{k}-S_{k-1})+R_{r}\int_{0}^{\infty}t^{n}P_{\infty}(t)\frac{dt}{t^{\gamma}}\\
 & \le\sum_{k=r}^{\infty}(k^{m}\lambda^{k}-k^{m+1}\lambda^{k+1})S_{k}-r^{m}\lambda^{r}S_{r-1}+R_{r}\int_{0}^{\infty}t^{n}P_{\infty}(t)\frac{dt}{t^{\gamma}}\\
 & \le S_{\infty}\sum_{k=r}^{\infty}(k^{m}\lambda^{k}-k^{m}\lambda^{k+1})+R_{r}\int_{0}^{\infty}t^{n}P_{\infty}(t)dt\\
 & =R_{r}\Big((1-\lambda)S_{\infty}+\int_{0}^{\infty}t^{n}P_{\infty}(t)\frac{dt}{t^{\gamma}}\Big).
\end{aligned}
\]
By Lemma \ref{lem:-21},
\[
\int_{0}^{\infty}t^{n}P_{\infty}(t)dt\le c_{n,\gamma}\lambda^{-(n+1)^{2}}(1-\lambda)^{n+1-\gamma},
\]
By Lemma \ref{lem}, 
\begin{equation}
S_{\infty}\le c_{n,\gamma}(1+\theta^{-n})\lambda^{-(n+1)^{2}}(1-\lambda)^{n-\gamma}.\label{eq:-292}
\end{equation}
Therefore, 
\begin{equation}
\begin{aligned}\sum_{k=r}^{\infty}k^{m}\lambda^{k}\int_{0}^{\infty}e^{-t\theta(1-\lambda)^{-1}\lambda^{k}}t^{n}P_{k}(t)\frac{dt}{t^{\gamma}} & \le c_{n,\gamma}(1+\theta^{-n})\lambda^{-(n+1)^{2}}(1-\lambda)^{n+1-\gamma}R_{r}.\end{aligned}
\label{eq:-291}
\end{equation}
Note that 
\[
\begin{aligned}R_{r} & =\lambda^{r}\sum_{k=0}^{\infty}(k+r)^{m}\lambda^{k}\le c_{m}\lambda^{r}\sum_{k=0}^{\infty}(k^{m}+r^{m})\lambda^{k}\\
 & \le c_{m}\lambda^{r}[(1-\lambda)^{-m-1}+r^{m}(1-\lambda)^{-1}]\le c_{m}(1+r)^{m}\lambda^{r}(1-\lambda)^{-m-1}.
\end{aligned}
\]
Combining (\ref{eq:-291}) and the above gives (\ref{eq:-246}). 
\end{proof}
\begin{prop}
\label{prop:-9}Let $V(\lambda)$ be the function defined by (\ref{eq:-3}).
For any $n\ge0$, we have 
\begin{equation}
\sum_{k=n}^{\infty}|V_{k}(\lambda,\theta)-V(\lambda)|\le c_{0}(\theta+\theta^{-1})\lambda^{n-4}(1-\lambda)^{-1},\label{eq:-293}
\end{equation}
where $c_{0}$ is a universal constant. 
\end{prop}

\begin{proof}
Let $P_{k}(t)=\prod_{j=0}^{k-1}(1+t\lambda^{j})^{-1/2}$. Then
\[
\sum_{k=n}^{\infty}|V_{k}(\lambda,\theta)-V(\lambda)|\le\sum_{k=n}^{\infty}I_{k}^{\prime}+\sum_{k=n}^{\infty}I_{k}^{\prime\prime},
\]
where 
\[
I_{k}^{\prime}=\frac{1}{2(1-\lambda)}\int_{0}^{\infty}(1-e^{-t\theta(1-\lambda)^{-1}\lambda^{k}/2})P_{\infty}(t)dt,
\]
and
\[
I_{k}^{\prime\prime}=\frac{1}{2(1-\lambda)}\int_{0}^{\infty}e^{-t\theta(1-\lambda)^{-1}\lambda^{k}/2}[P_{k}(t)-P_{\infty}(t)]dt.
\]

By $1-e^{-t\theta(1-\lambda)^{-1}\lambda^{k}/2}\le t\theta(1-\lambda)^{-1}\lambda^{k}/2$,
we have
\begin{equation}
\begin{aligned}I_{k}^{\prime} & \le\frac{(1-\lambda)^{-2}\theta\lambda^{k}}{4}\int_{0}^{\infty}tP_{\infty}(t)dt.\end{aligned}
\label{eq:-294}
\end{equation}
Therefore, by Lemma \ref{lem:-21}, 
\begin{equation}
\sum_{k=n}^{\infty}I_{k}^{\prime}\le c_{0}(1-\lambda)^{-3}\theta\lambda^{n}\int_{0}^{\infty}tP_{\infty}(t)dt\le c_{0}\theta\lambda^{n-4}(1-\lambda)^{-1}.\label{eq:-296}
\end{equation}

For $\sum_{k=n}^{\infty}I_{k}^{\prime\prime}$, note that $1-(1+t\lambda^{j})^{-1/2}\le2^{-1}t\lambda^{j}$
and therefore,
\[
\begin{aligned}P_{k}(t)-P_{\infty}(t) & =\sum_{j=k}^{\infty}[P_{j}(t)-P_{j+1}(t)]=\sum_{j=k}^{\infty}[1-(1+t\lambda^{j})^{-1/2}]P_{j}(t)\\
 & \le2^{-1}\sum_{j=k}^{\infty}t\lambda^{j}P_{j}(t)\le tP_{k}(t)\sum_{j=k}^{\infty}\lambda^{j}=(1-\lambda)^{-1}\lambda^{k}tP_{k}(t).
\end{aligned}
\]
By the above and Lemma \ref{lem:-23}, 
\begin{equation}
\begin{aligned}\sum_{k=n}^{\infty}I_{k}^{\prime\prime} & \le2^{-1}(1-\lambda)^{-2}\sum_{k=n}^{\infty}\lambda^{k}\int_{0}^{\infty}e^{-t\theta(1-\lambda)^{-1}\lambda^{k}/2}tP_{k}(t)dt\\
 & \le c_{0}(1+\theta^{-1})\lambda^{n-4}(1-\lambda)^{-1}.
\end{aligned}
\label{eq:-295}
\end{equation}
The inequality (\ref{eq:-293}) follows readily from (\ref{eq:-295})
and (\ref{eq:-296}).
\end{proof}
\begin{prop}
\label{prop:-8}Let $V(\lambda)$ be the function given in (\ref{eq:-3}).
For any $\theta>0$,
\begin{equation}
\sum_{k=0}^{\infty}|\partial_{\lambda}V_{k}(\lambda,\theta)-V^{\prime}(\lambda)|\le c_{0}(\theta+\theta^{-1})\lambda^{-5}(1-\lambda)^{-2},\label{eq:-240}
\end{equation}
where $V_{k}(\lambda,\theta)$ is the function defined by (\ref{eq:-246}),
and $c_{0}$ is a universal constant.
\end{prop}

\begin{proof}
We have
\begin{equation}
\begin{aligned}\partial_{\lambda}V_{k}(\lambda,\theta) & =\frac{\partial}{\partial\lambda}\Big(\int_{0}^{\infty}e^{-t\theta\lambda^{k}}tP_{k}(2t(1-\lambda))dt\Big)\\
 & =-\theta k\lambda^{k-1}\int_{0}^{\infty}e^{-t\theta\lambda^{k}}tP_{k}(2t(1-\lambda))dt\\
 & \quad-\int_{0}^{\infty}e^{-t\theta\lambda^{k}}\sum_{j=0}^{k-1}\frac{t[j\lambda^{j-1}(1-\lambda)-\lambda^{j}]}{1+2t(1-\lambda)\lambda^{j}}P_{k}(2t(1-\lambda))dt.
\end{aligned}
\label{eq:-247}
\end{equation}
By Lemma \ref{lem:-23} and a change of variable, 
\begin{equation}
\sum_{k=0}^{\infty}k\lambda^{k-1}\int_{0}^{\infty}e^{-t\theta\lambda^{k}}tP_{k}(2t(1-\lambda))dt\le c_{0}(1+\theta^{-1})\lambda^{-4}(1-\lambda)^{-2},\label{eq:-244}
\end{equation}
where $c_{0}$ is a universal constant. 

We replace the integrand of the second integral on the right-hand
side of (\ref{eq:-247}) one by one. We first show that
\begin{equation}
\sum_{k=0}^{\infty}\Big|\int_{0}^{\infty}e^{-t\theta\lambda^{k}}\sum_{j=0}^{k-1}\frac{t[j\lambda^{j-1}(1-\lambda)-\lambda^{j}]}{1+2t(1-\lambda)\lambda^{j}}[P_{k}(2t(1-\lambda))-P_{\infty}(2t(1-\lambda))]dt\Big|\le c_{0}(1-\lambda)^{-2}.\label{eq:-237}
\end{equation}
Clearly, 
\[
\sum_{j=0}^{k-1}\frac{t|j\lambda^{j-1}(1-\lambda)-\lambda^{j}|}{1+2t(1-\lambda)\lambda^{j}}\le t\sum_{j=0}^{k-1}[j\lambda^{j-1}(1-\lambda)+\lambda^{j}]\le c_{0}t(1-\lambda)^{-1}.
\]
Therefore, the left-hand side of (\ref{eq:-237}) is dominated by
\[
c_{0}(1-\lambda)^{-1}\sum_{k=0}^{\infty}\int_{0}^{\infty}e^{-t\theta\lambda^{k}}t[P_{k}(2t(1-\lambda))-P_{\infty}(2t(1-\lambda))]dt,
\]
which, by Lemma \ref{lem} and a change of variable, is dominated
by $c_{0}(1-\lambda)^{-2}$.

We next show that
\begin{equation}
\sum_{k=0}^{\infty}\Big|\int_{0}^{\infty}e^{-t\theta\lambda^{k}}\sum_{j=k}^{\infty}\frac{t[j\lambda^{j-1}(1-\lambda)-\lambda^{j}]}{1+2t(1-\lambda)\lambda^{j}}P_{\infty}(2t(1-\lambda))dt\Big|\le c_{0}\lambda^{-5}(1-\lambda)^{-2}.\label{eq:-241}
\end{equation}
Note that
\begin{equation}
\begin{aligned}\sum_{j=k}^{\infty}\frac{t|j\lambda^{j-1}(1-\lambda)-\lambda^{j}|}{1+2t(1-\lambda)\lambda^{j}} & \le t\sum_{j=k}^{\infty}[j\lambda^{j-1}(1-\lambda)+\lambda^{j}]\le c_{0}t[k\lambda^{k-1}+(1-\lambda)^{-1}\lambda^{k}].\end{aligned}
\label{eq:-242}
\end{equation}
Therefore, the left-hand side of (\ref{eq:-241}) is dominated by
\[
\begin{aligned} & c_{0}\sum_{k=0}^{\infty}\int_{0}^{\infty}e^{-t\theta\lambda^{k}}t[k\lambda^{k-1}+(1-\lambda)^{-1}\lambda^{k}]P_{\infty}(2t(1-\lambda))dt\\
 & \le c_{0}\sum_{k=0}^{\infty}[k\lambda^{k-1}+(1-\lambda)^{-1}\lambda^{k}]\int_{0}^{\infty}tP_{\infty}(2t(1-\lambda))dt\\
 & \le c_{0}\lambda^{-1}(1-\lambda)^{-2}\int_{0}^{\infty}tP_{\infty}(2t(1-\lambda))dt\\
 & \le c_{0}\lambda^{-5}(1-\lambda)^{-2},
\end{aligned}
\]
where the last inequality follows from Lemma \ref{lem:-21} and a
change of variable.

We now show that
\begin{equation}
\sum_{k=0}^{\infty}\Big|\int_{0}^{\infty}(e^{-t\theta\lambda^{k}}-1)\sum_{j=0}^{\infty}\frac{t[j\lambda^{j-1}(1-\lambda)-\lambda^{j}]}{1+2t(1-\lambda)\lambda^{j}}P_{\infty}(2t(1-\lambda))dt\Big|\le c_{0}\theta\lambda^{-1}(1-\lambda)^{-2}.\label{eq:-243}
\end{equation}
By (\ref{eq:-242}), the left-hand side of (\ref{eq:-243}) is dominated
by 
\[
\begin{aligned} & c_{0}\lambda^{-1}(1-\lambda)^{-1}\sum_{k=0}^{\infty}\int_{0}^{\infty}(1-e^{-t\theta\lambda^{k}})tP_{\infty}(2t(1-\lambda))dt\\
 & \le c_{0}\theta\lambda^{-1}(1-\lambda)^{-1}\sum_{k=0}^{\infty}\lambda^{k}\int_{0}^{\infty}t^{2}P_{\infty}(2t(1-\lambda))dt\\
 & \le c_{0}\theta\lambda^{-1}(1-\lambda)^{-2}\int_{0}^{\infty}t^{2}P_{\infty}(2t(1-\lambda))dt\\
 & \le c_{0}\theta\lambda^{-1}(1-\lambda)^{-2},
\end{aligned}
\]
where the last inequality follows from Lemma \ref{lem:-21} and a
change of variable.

In view of (\ref{eq:-244}), (\ref{eq:-237}), (\ref{eq:-241}) and
(\ref{eq:-243}), we see that
\begin{equation}
\sum_{k=0}^{\infty}\Big|\partial_{\lambda}V_{k}(\lambda,\theta)-\int_{0}^{\infty}\sum_{j=0}^{\infty}\frac{t[j\lambda^{j-1}(1-\lambda)-\lambda^{j}]}{1+2t(1-\lambda)\lambda^{j}}P_{\infty}(2t(1-\lambda))dt\Big|\le c_{0}(\theta+\theta^{-1})\lambda^{-5}(1-\lambda)^{-2}.\label{eq:-245}
\end{equation}
Clearly, 
\[
\begin{aligned}V^{\prime}(\lambda) & =\frac{\partial}{\partial\lambda}\Big(\int_{0}^{\infty}P_{\infty}(2t(1-\lambda))dt\Big)\\
 & =\int_{0}^{\infty}\sum_{j=0}^{\infty}\frac{t[j\lambda^{j-1}(1-\lambda)-\lambda^{j}]}{1+2t(1-\lambda)\lambda^{j}}P_{\infty}(2t(1-\lambda))dt,
\end{aligned}
\]
where differentiation under the integral sign is due to the dominated
convergence theorem. The inequality (\ref{eq:-240}) follows readily
from (\ref{eq:-245}) and the above.
\end{proof}
The following proposition can be obtained similarly using Lemma \ref{lem:-21},
Lemma \ref{lem} and Lemma \ref{lem:-23}.
\begin{prop}
\label{prop:-10}Let $V(\lambda)$ be the function given in (\ref{eq:-3}).
For any $\theta>0$, 
\begin{equation}
\sum_{k=0}^{\infty}|\partial_{\lambda}^{2}V_{k}(\lambda,\theta)-V^{\prime\prime}(\lambda)|\le c_{0}(\theta+\theta^{-1})^{2}\lambda^{-9}(1-\lambda)^{-3},\label{eq:-19}
\end{equation}
where $V_{k}(\lambda,\theta)$ is the function defined by (\ref{eq:-246}),
and $c_{0}$ is a universal constant.
\end{prop}

\begin{proof}
We have
\[
\begin{aligned}\partial_{\lambda}^{2}V_{k}(\lambda,\theta) & =\frac{\partial^{2}}{\partial\lambda^{2}}\Big(\int_{0}^{\infty}e^{-t\theta\lambda^{k}}tP_{k}(2t(1-\lambda))dt\Big)\\
 & =-\theta E_{1,k}+\theta^{2}E_{2,k}+\theta E_{3,k}+E_{4,k}+E_{5,k}+2E_{6,k},
\end{aligned}
\]
where
\begin{align*}
E_{1,k} & =-k(k-1)\lambda^{k-2}\int_{0}^{\infty}e^{-t\theta\lambda^{k}}tP_{k}(2t(1-\lambda))dt,\\
E_{2,k} & =k^{2}\lambda^{2k-2}\int_{0}^{\infty}e^{-t\theta\lambda^{k}}t^{2}P_{k}(2t(1-\lambda))dt,
\end{align*}
and
\begin{align*}
E_{3,k} & =k\lambda^{k-1}\int_{0}^{\infty}e^{-t\theta\lambda^{k}}\sum_{j=0}^{k-1}\frac{t^{2}[j\lambda^{j-1}(1-\lambda)-\lambda^{j}]}{1+2t(1-\lambda)\lambda^{j}}P_{k}(2t(1-\lambda))dt,\\
E_{4,k} & =\int_{0}^{\infty}e^{-t\theta\lambda^{k}}\Big(\sum_{j=0}^{k-1}\frac{t[j\lambda^{j-1}(1-\lambda)-\lambda^{j}]}{1+2t(1-\lambda)\lambda^{j}}\Big)^{2}P_{k}(2t(1-\lambda))dt,\\
E_{5,k} & =-\int_{0}^{\infty}e^{-t\theta\lambda^{k}}\sum_{j=0}^{k-1}\frac{t[j^{2}\lambda^{j-2}(1-\lambda)-2j\lambda^{j-1}]}{1+2t(1-\lambda)\lambda^{j}}P_{k}(2t(1-\lambda))dt,\\
E_{6,k} & =\int_{0}^{\infty}e^{-t\theta\lambda^{k}}\sum_{j=0}^{k-1}\frac{t^{2}[j\lambda^{j-1}(1-\lambda)-\lambda^{j}]^{2}}{[1+2t(1-\lambda)\lambda^{j}]^{2}}P_{k}(2t(1-\lambda))dt.
\end{align*}

For $E_{1,k}$ and $E_{2,k}$, by Lemma \ref{lem:-23} and a change
of variable,
\begin{equation}
\sum_{k=0}^{\infty}|E_{1,k}|+\sum_{k=0}^{\infty}|E_{2,k}|\le c_{0}(1+\theta^{-1})^{2}\lambda^{-9}(1-\lambda)^{-3}.\label{eq:-255}
\end{equation}

For $E_{3,k}$, it is easily seen that 
\begin{equation}
\sum_{j=0}^{k-1}\frac{|j\lambda^{j-1}(1-\lambda)-\lambda^{j}|}{1+2t(1-\lambda)\lambda^{j}}\le c_{0}(1-\lambda)^{-1},\label{eq:-65}
\end{equation}
and therefore, 
\begin{equation}
\begin{aligned}\sum_{k=0}^{\infty}|E_{3,k}| & \le c_{0}(1-\lambda)^{-1}\sum_{k=1}^{\infty}k\lambda^{k-1}\int_{0}^{\infty}e^{-t\theta\lambda^{k}}t^{2}P_{k}(2t(1-\lambda))dt\\
 & \le c_{0}(1+\theta^{-1})^{2}\lambda^{-9}(1-\lambda)^{-3},
\end{aligned}
\label{eq:-256}
\end{equation}
where the last inequality follows from Lemma \ref{lem:-23} and a
change of variable.

For $E_{4,k}$, by (\ref{eq:-65}), we have
\begin{equation}
\begin{aligned} & \sum_{k=0}^{\infty}\Big|\int_{0}^{\infty}e^{-t\theta\lambda^{k}}\Big(\sum_{j=0}^{k-1}\frac{t[j\lambda^{j-1}(1-\lambda)-\lambda^{j}]}{1+2t(1-\lambda)\lambda^{j}}\Big)^{2}[P_{k}(2t(1-\lambda))-P_{\infty}(2t(1-\lambda))]dt\Big|\\
 & \le c_{0}(1-\lambda)^{-2}\sum_{k=1}^{\infty}\int_{0}^{\infty}e^{-t\theta\lambda^{k}}t^{2}[P_{k}(2t(1-\lambda))-P_{\infty}(2t(1-\lambda))]dt\\
 & \le c_{0}(1+\theta^{-1})^{2}\lambda^{-9}(1-\lambda)^{-3},
\end{aligned}
\label{eq:-66}
\end{equation}
where the last inequality follows from Lemma \ref{lem} and a change
of variable. Let 
\[
S_{k}(t)=\sum_{j=0}^{k-1}\frac{t[j\lambda^{j-1}(1-\lambda)-\lambda^{j}]}{1+2t(1-\lambda)\lambda^{j}}.
\]
Note that
\[
\begin{aligned}|S_{k}(t)^{2}-S_{\infty}(t)^{2}| & \le2\Big(\sum_{j=0}^{\infty}\frac{t[j\lambda^{j-1}(1-\lambda)-\lambda^{j}]}{1+2t(1-\lambda)\lambda^{j}}\Big)\Big(\sum_{j=k}^{\infty}\frac{t|j\lambda^{j-1}(1-\lambda)-\lambda^{j}|}{1+2t(1-\lambda)\lambda^{j}}\Big)\\
 & \le c_{0}t^{2}(1-\lambda)^{-1}[k\lambda^{k-1}+(1-\lambda)\lambda^{k}].
\end{aligned}
\]
Therefore,
\begin{equation}
\begin{aligned} & \sum_{k=0}^{\infty}\Big|\int_{0}^{\infty}e^{-t\theta\lambda^{k}}(S_{k}(t)^{2}-S_{\infty}(t)^{2})P_{\infty}(2t(1-\lambda))dt\Big|\\
 & \le c_{0}(1-\lambda)^{-1}\sum_{k=0}^{\infty}[k\lambda^{k-1}+(1-\lambda)\lambda^{k}]\int_{0}^{\infty}e^{-t\theta\lambda^{k}}t^{2}P_{\infty}(2t(1-\lambda))dt\\
 & \le c_{0}(1+\theta^{-1})^{2}\lambda^{-9}(1-\lambda)^{-3},
\end{aligned}
\label{eq:-248}
\end{equation}
where the last inequality follows from Lemma \ref{lem:-23} and a
change of variable. By (\ref{eq:-65}) and $1-e^{-t\theta\lambda^{k}}\le t\theta\lambda^{k}$,
it is easily seen that
\begin{equation}
\begin{aligned} & \sum_{k=0}^{\infty}\Big|\int_{0}^{\infty}(1-e^{-t\theta\lambda^{k}})S_{\infty}(t)^{2}P_{\infty}(2t(1-\lambda))dt\Big|\\
 & \le c_{0}\theta(1-\lambda)^{-2}\sum_{k=1}^{\infty}\lambda^{k}\int_{0}^{\infty}t^{3}P_{\infty}(2t(1-\lambda))dt\\
 & \le c_{0}\theta(1-\lambda)^{-3}\int_{0}^{\infty}t^{3}P_{\infty}(2t(1-\lambda))dt\\
 & \le c_{0}\theta(1-\lambda)^{-3},
\end{aligned}
\label{eq:-67}
\end{equation}
where the last inequality follows from Lemma \ref{lem:-21}. Combining
(\ref{eq:-66}), (\ref{eq:-248}) and (\ref{eq:-67}) gives 
\begin{equation}
\sum_{k=0}^{\infty}|E_{4,k}-E_{4}|\le c_{0}(1+\theta^{-1})^{2}\lambda^{-9}(1-\lambda)^{-3}.\label{eq:-249}
\end{equation}
where 
\begin{equation}
E_{4}=\int_{0}^{\infty}\Big(\sum_{j=0}^{\infty}\frac{t[j\lambda^{j-1}(1-\lambda)-\lambda^{j}]}{1+2t(1-\lambda)\lambda^{j}}\Big)^{2}P_{\infty}(2t(1-\lambda))dt.\label{eq:-252}
\end{equation}

For $E_{5,k}$ and $E_{6,k}$, similar argument yields that
\begin{equation}
\sum_{k=0}^{\infty}|E_{5,k}-E_{5}|\le c_{0}(1+\theta^{-1})\lambda^{-4}(1-\lambda)^{-3},\label{eq:-250}
\end{equation}
where 
\begin{equation}
E_{5}=-\int_{0}^{\infty}\sum_{j=0}^{\infty}\frac{t[j^{2}\lambda^{j-2}(1-\lambda)-2j\lambda^{j-1}]}{1+2t(1-\lambda)\lambda^{j}}P_{\infty}(2t(1-\lambda))dt,\label{eq:-253}
\end{equation}
and
\begin{equation}
\sum_{k=0}^{\infty}|E_{6,k}-E_{6}|\le c_{0}(1+\theta^{-1})^{2}\lambda^{-9}(1-\lambda)^{-3},\label{eq:-251}
\end{equation}
where 
\begin{equation}
E_{6}=\int_{0}^{\infty}\sum_{j=0}^{\infty}\frac{t^{2}[j\lambda^{j-1}(1-\lambda)-\lambda^{j}]^{2}}{[1+2t(1-\lambda)\lambda^{j}]^{2}}P_{\infty}(2t(1-\lambda))dt.\label{eq:-254}
\end{equation}

Note that, by the dominated convergence theorem, $V^{\prime\prime}(\lambda)=E_{4}+E_{5}+2E_{6}$.
Combining (\ref{eq:-255}), (\ref{eq:-256}), (\ref{eq:-249}), (\ref{eq:-250})
and (\ref{eq:-251}) completes the proof.
\end{proof}
Estimates for partial derivatives $\partial_{\theta}V_{k}(\lambda,\theta)$,
$\partial_{\theta}^{2}V_{k}(\lambda,\theta)$ and $\partial_{\theta}\partial_{\lambda}V_{k}(\lambda,\theta)$
can be obtained by an argument similar to Proposition \ref{prop:-8}
or Proposition \ref{prop:-10}. However, the derivation can be simplified
by using properties of analytic functions.
\begin{prop}
\label{prop:-11}For any $\theta>0$,
\begin{equation}
\sum_{k=0}^{\infty}|\partial_{\theta}^{r}V_{k}(\lambda,\theta)|\le c_{r}(\theta+\theta^{-1})^{r+1}\lambda^{-4}(1-\lambda)^{-1},\label{eq:-258}
\end{equation}
and 
\begin{equation}
\sum_{k=0}^{\infty}|\partial_{\theta}^{r}\partial_{\lambda}V_{k}(\lambda,\theta)|\le c_{r}(\theta+\theta^{-1})^{r+1}\lambda^{-5}(1-\lambda)^{-1},\label{eq:-259}
\end{equation}
where $c_{r}$ is a constant depending only on $r$.
\end{prop}

\begin{proof}
Clearly, the functions $V_{k}(\lambda,z)$ are analytic in $z$ on
$\{z\in\mathbb{C}:\text{Re}(z)>0\}$, and (\ref{eq:-232}) easily
extends to 
\[
\sum_{k=0}^{\infty}|V_{k}(\lambda,z)-V(\lambda)|\le c_{0}(\text{Re}(z)+\text{Re}(z)^{-1})\lambda^{-4}(1-\lambda)^{-1},\quad\text{Re}(z)>0.
\]
By Cauchy's integral formula, 
\[
\partial_{\theta}^{r}V_{k}(\lambda,\theta)=\frac{r!}{\text{i}2\pi}\int_{|z-\theta|=\theta/2}\frac{V_{k}(\lambda,z)-V(\lambda)}{(z-\theta)^{r+1}}dz.
\]
Therefore,
\[
\begin{aligned}\sum_{k=0}^{\infty}|\partial_{\theta}^{r}V_{k}(\lambda,\theta)| & \le\frac{r!}{2\pi}\int_{|z-\theta|=\theta/2}\sum_{k=1}^{\infty}\frac{|V_{k}(\lambda,z)-V(\lambda)|}{|z-\theta|^{r+1}}|dz|\\
 & \le c_{r}\int_{|z-\theta|=\theta/2}\frac{(\text{Re}(z)+\text{Re}(z)^{-1})\lambda^{-4}(1-\lambda)^{-1}}{\theta^{r+1}}|dz|\\
 & \le c_{r}\theta^{-r}(\theta+\theta^{-1})\lambda^{-4}(1-\lambda)^{-1}\\
 & \le c_{r}(\theta+\theta^{-1})^{r+1}\lambda^{-4}(1-\lambda)^{-1}.
\end{aligned}
\]

The estimate (\ref{eq:-259}) can be obtained similarly by considering
the functions $\partial_{\lambda}V_{k}(\lambda,z)$ on $\{z\in\mathbb{C}:\text{Re}(z)>0\}$
and using the estimate (\ref{eq:-240}).
\end{proof}
\begin{rem}
\label{rem:-10}Using similar arguments and by Lemma \ref{lem:-21},
Lemma \ref{lem} and Lemma \ref{lem:-23} with $\gamma=1/2$, it can
be shown that the estimates (\ref{eq:-293}), (\ref{eq:-240}), (\ref{eq:-19}),
(\ref{eq:-258}) and (\ref{eq:-259}) also hold if $V_{k}(\lambda,\theta)$
and $V(\lambda)$ are replaced by $U_{k}(\lambda,\theta)$ and $U(\lambda)$
respectively, where $U_{k}(\lambda,\theta)$ is the function defined
by (\ref{eq:-20}).
\end{rem}

\subsection{\label{subsec:}Proof of Proposition \ref{prop:}}

In this subsection, we assume that $\sigma_{t}=\sigma(t)$ is a $C^{1,1}([0,T])$.
Let $t_{n}=n\Delta t$. It is easily seen by It\^{o}'s formula that
\begin{equation}
\begin{aligned}\mathbb{E}[\langle\log I\rangle_{t_{n}}] & =\mathbb{E}\Big[\Big(\sum_{k=1}^{n}\bar{\sigma}v_{k-1}^{-1/2}\int_{t_{k-1}}^{t_{k}}\sigma(u)dW_{u}\Big)^{2}\Big]=\bar{\sigma}^{2}\sum_{k=0}^{n-1}\mathbb{E}(v_{k}^{-1})A_{k+1},\end{aligned}
\label{eq:-18}
\end{equation}
where 
\[
v_{k}=\lambda^{k}v_{0}+(1-\lambda)\Delta t^{-1}\sum_{j=1}^{k}\lambda^{k-j}\Big(\int_{t_{j-1}}^{t_{j}}\sigma(u)dW_{u}\Big)^{2}.
\]
and
\begin{equation}
A_{k}=\Delta t^{-1}\int_{t_{k-1}}^{t_{k}}\sigma(u)^{2}du,\quad1\le k\le N.\label{eq:-56}
\end{equation}
By $x^{-1}=\int_{0}^{\infty}e^{-tx}dt$ and Fubini's theorem,
\[
\begin{aligned}\mathbb{E}(v_{k}^{-1}) & =\mathbb{E}\Big[\int_{0}^{\infty}\exp\Big(-t\lambda^{k}v_{0}-t(1-\lambda)\Delta t^{-1}\sum_{j=1}^{k}\lambda^{k-j}\Big(\int_{t_{j-1}}^{t_{j}}\sigma(u)dW_{u}\Big)^{2}\Big)dt\Big]\\
 & =\int_{0}^{\infty}e^{-t\lambda^{k}v_{0}}\mathbb{E}\Big[\prod_{j=1}^{k}\exp\Big(-t(1-\lambda)\Delta t^{-1}\lambda^{k-j}\Big(\int_{t_{j-1}}^{t_{j}}\sigma(u)dW_{u}\Big)^{2}\Big)\Big]dt\\
 & =\int_{0}^{\infty}e^{-t\lambda^{k}v_{0}}\prod_{j=1}^{k}\big(1+2t(1-\lambda)\lambda^{k-j}A_{j}\big)^{-1/2}dt.
\end{aligned}
\]
Therefore, by the above and (\ref{eq:-18}) and change of variables,
we obtain that
\begin{equation}
\begin{aligned}\mathbb{E}[\langle\log I\rangle_{t_{n}}] & =\frac{\bar{\sigma}^{2}\Delta t}{2(1-\lambda)}\sum_{k=0}^{n-1}\int_{0}^{\infty}e^{-t(1-\lambda)^{-1}\lambda^{k}A_{k+1}^{-1}v_{0}/2}\prod_{j=0}^{k-1}\Big(1+t\lambda^{j}\frac{A_{k-j}}{A_{k+1}}\Big)^{-1/2}dt.\end{aligned}
\label{eq:-57}
\end{equation}

Let $\delta(\Delta t,\lambda)$ be the quantity defined by (\ref{eq:-307}).
We now derive the first order approximation, in terms of $\delta(\Delta t,\lambda)$,
of the right-hand side of (\ref{eq:-57}). Note that the right-hand
side of (\ref{eq:-57}) is not in a form for which Taylor's theorem
can be readily applied to give an expansion in $\Delta t$, or more
precisely, in $\delta(\Delta t,\lambda)$. We will first rewrite the
right-hand side of (\ref{eq:-57}) such that its expansion in $\Delta t$
is more transparent. The following lemma will be crucial in rewriting
(\ref{eq:-57}) in a form for which Taylor's theorem can be easily
applied. 
\begin{lem}
\label{lem:-22}Let $\mathbb{R}_{+}^{\omega}$ be the space of infinite
sequences $x=(x_{0},x_{1},\dots)$ with $x_{j}>0$. For any $x\in\mathbb{R}_{+}^{\omega}$,
let
\begin{equation}
P_{k}(t,x)=\prod_{j=0}^{k-1}(1+t\lambda^{j}x_{j})^{-1/2}.\label{eq:-29}
\end{equation}
Suppose that $a_{k}=(a_{k,0},a_{k,1},\dots)$ and $b_{k}=(b_{k,0},b_{k,1},\dots)$
are sequences in $\mathbb{R}_{+}^{\omega}$ satisfying 
\begin{equation}
a_{k,j}\ge\kappa\lambda^{\gamma j},\;b_{k,j}\ge\kappa\lambda^{\gamma j},\quad0\le j\le k,\label{eq:-51}
\end{equation}
for some constants $\kappa>0$ and $\gamma\ge0$. Then, for any $1\le n\le N$
and any $(\theta_{0},\theta_{1},\dots)\in\mathbb{R}_{+}^{\omega}$
satisfying $\theta_{k}\ge\theta$ for some $\theta>0$, 
\begin{equation}
\begin{aligned} & \frac{\Delta t}{1-\lambda}\sum_{k=0}^{n-1}\int_{0}^{\infty}e^{-t(1-\lambda)^{-1}\lambda^{k}\theta_{k}}|P_{k}(t,a_{k})-P_{k}(t,b_{k})|dt\\
 & \le c_{\kappa,\gamma,T}(1+\theta^{-1})\lambda^{-4(1+\gamma)}(1-\lambda)^{-2}[1+\delta(\Delta t,\lambda)]\sup_{k\ge0}\rho(a_{k},b_{k})
\end{aligned}
\label{eq:-49}
\end{equation}
where 
\begin{equation}
\rho(x,y)=\sup_{j\ge0}(j+1)^{-2}|x_{j}-y_{j}|,\quad x,y\in\mathbb{R}_{+}^{\omega},\label{eq:-50}
\end{equation}
and $c_{\kappa,\gamma,T}$ is a constant depending only on $\kappa$,
$\gamma$ and $T$.
\end{lem}

\begin{proof}
Clearly, 
\[
|\partial_{x_{j}}P_{k}(t,x)|=\frac{t\lambda^{j}}{2(1+t\lambda^{j}x_{j})}P_{k}(t,x)\le t\lambda^{j}P_{k}(t,x).
\]
If $x$ satisfies that $x_{j}\ge\kappa\lambda^{\gamma j}$ for all
$j\le k$, by the fact that $P_{k}(t,x)$ is decreasing in each $x_{j}$,
we have
\[
|\partial_{x_{j}}P_{k}(t,x)|\le t\lambda^{j}P_{k}(t,\kappa,\kappa\lambda^{\gamma},\dots,\kappa\lambda^{\gamma(k-1)})=t\lambda^{j}\prod_{j=0}^{k-1}(1+t\kappa\lambda_{0}^{j})^{-1/2}.
\]
where $\lambda_{0}=\lambda^{1+\gamma}\le\lambda$. Therefore, by Taylor's
theorem,
\[
\begin{aligned}|P_{k}(t,a_{k})-P_{k}(t,b_{k})| & \le\sum_{j=0}^{k-1}|a_{k,j}-b_{k,j}|\lambda^{j}t\prod_{j=0}^{k-1}(1+t\kappa\lambda_{0}^{j})^{-1/2}\\
 & \le\rho(a_{k},b_{k})\sum_{j=0}^{k-1}(j+1)^{2}\lambda^{j}t\prod_{j=0}^{k-1}(1+t\kappa\lambda_{0}^{j})^{-1/2}\\
 & \le4\rho(a_{k},b_{k})(1-\lambda)^{-3}t\prod_{j=0}^{k-1}(1+t\kappa\lambda_{0}^{j})^{-1/2}.
\end{aligned}
\]
Let $P_{k,\lambda}(t)=\prod_{j=0}^{k-1}(1+t\lambda^{j})^{-1/2}$.
Then, by $\theta_{k}\ge\theta$,
\begin{equation}
\begin{aligned} & \frac{\Delta t}{1-\lambda}\sum_{k=0}^{n-1}\int_{0}^{\infty}e^{-t(1-\lambda)^{-1}\lambda^{k}\theta_{k}}|P_{k}(t,a_{k})-P_{k}(t,b_{k})|dt\\
 & \le\frac{\Delta t}{1-\lambda}\sum_{k=0}^{n-1}\int_{0}^{\infty}e^{-t(1-\lambda)^{-1}\lambda^{k}\theta}|P_{k}(t,a_{k})-P_{k}(t,b_{k})|dt\\
 & \le4\Delta t\Big(\sup_{k\ge0}\rho(a_{k},b_{k})\Big)(1-\lambda)^{-4}\sum_{k=0}^{n-1}\int_{0}^{\infty}e^{-t(1-\lambda)^{-1}\lambda^{k}\theta}t\prod_{j=0}^{k-1}(1+t\kappa\lambda_{0}^{j})^{-1/2}dt\\
 & \le c_{\kappa}\Delta t\Big(\sup_{k\ge0}\rho(a_{k},b_{k})\Big)(1-\lambda)^{-4}\sum_{k=0}^{n-1}\int_{0}^{\infty}e^{-t(1-\lambda_{0})^{-1}\lambda_{0}^{k}\kappa^{-1}\theta}tP_{k,\lambda_{0}}(t)^{-1/2}dt
\end{aligned}
\label{eq:-30}
\end{equation}
where the last inequality is due to the fact that $(1-\lambda)^{-1}\lambda^{k}$
is increasing in $\lambda$ and a change of variable. By Lemma \ref{lem:-21}
and Lemma \ref{lem},
\begin{equation}
\begin{aligned} & \sum_{k=0}^{n-1}\int_{0}^{\infty}e^{-t(1-\lambda_{0})^{-1}\lambda_{0}^{k}\kappa^{-1}\theta}tP_{k,\lambda_{0}}(t)dt\\
 & \le\sum_{k=0}^{n-1}\int_{0}^{\infty}e^{-t(1-\lambda_{0})^{-1}\lambda_{0}^{k}\kappa^{-1}\theta}t[P_{k,\lambda_{0}}(t)-P_{\infty,\lambda_{0}}(t)]dt+n\int_{0}^{\infty}tP_{\infty,\lambda_{0}}(t)dt\\
 & \le c_{0}(1+\kappa\theta^{-1})\lambda_{0}^{-4}(1-\lambda_{0})+c_{n}n\lambda_{0}^{-4}(1-\lambda_{0})^{2}\\
 & \le c_{\kappa,\gamma}(1+\theta^{-1})\lambda^{-4(1+\gamma)}(1-\lambda)[1+n(1-\lambda)],
\end{aligned}
\label{eq:-31}
\end{equation}
where the last inequality is due to $1-\lambda^{1+\gamma}\le(1+\gamma)(1-\lambda)$.
The estimate (\ref{eq:-49}) now follows readily from (\ref{eq:-30})
and (\ref{eq:-31}) and that $n\Delta t\le T$.
\end{proof}
\begin{lem}
\label{lem:-24}Let $A_{0}=\sigma(0)^{2}$ and $A_{k}$ be given by
(\ref{eq:-56}) for $1\le k\le N$. Define $a_{k},b_{k}\in\mathbb{R}_{+}^{\omega}$
by
\[
a_{k,j}=A_{k-j}/A_{k+1},\quad b_{k,j}=\exp\Big(-2(j+1)\Delta t\frac{\sigma^{\prime}(t_{k})}{\sigma(t_{k})}\Big),
\]
for $0\le j\le k$, and $a_{k,j}=b_{k,j}=0$ for $j>k$. Let
\begin{equation}
L_{\sigma,T}=\Vert\sigma^{\prime}\Vert_{C^{0,1}}\sigma^{\ast}+\Vert\sigma^{\prime}\Vert_{L^{\infty}}^{2}+4e^{3(1+T)\Vert\sigma^{\prime}\Vert_{L^{\infty}}/\sigma_{\ast}}.\label{eq:-58}
\end{equation}
Then, when $\Delta t<2^{-1/2}L_{\sigma,T}^{-1}$, we have
\begin{equation}
\rho(a_{k},b_{k})\le c_{T,\sigma_{\ast},\sigma^{\ast},\Vert\sigma^{\prime}\Vert_{L^{\infty}},\Vert\sigma^{\prime}\Vert_{C^{0,1}}}\Delta t^{2},\label{eq:-268}
\end{equation}
where $\rho(\cdot,\cdot)$ is the metric given by (\ref{eq:-50}),
and $c_{T,\sigma_{\ast},\sigma^{\ast},\Vert\sigma^{\prime}\Vert_{L^{\infty}},\Vert\sigma^{\prime}\Vert_{C^{0,1}}}$
is a constant depending only on $T$, $\sigma_{\ast}$, $\sigma^{\ast}$,
$\Vert\sigma^{\prime}\Vert_{L^{\infty}}$ and $\Vert\sigma^{\prime}\Vert_{C^{0,1}}$.
\end{lem}

\begin{proof}
Note that $\Vert\sigma^{\prime}\sigma\Vert_{C^{0,1}}\le\Vert\sigma^{\prime}\Vert_{C^{0,1}}\Vert\sigma\Vert_{L^{\infty}}+\Vert\sigma^{\prime}\Vert_{L^{\infty}}^{2}$.
By Taylor's theorem, 
\[
|\sigma(t)^{2}-\sigma(t_{k})^{2}-2\sigma^{\prime}(t_{k})\sigma(t_{k})(t-t_{k})|\le(\Vert\sigma^{\prime}\Vert_{C^{0,1}}\sigma^{\ast}+\Vert\sigma^{\prime}\Vert_{L^{\infty}}^{2})(t-t_{k})^{2},
\]
Integrating the above over $[t_{k-j-1},t_{k-j}]$ gives that
\[
|A_{k-j}\Delta t-\sigma(t_{k})^{2}\Delta t+(2j+1)\sigma^{\prime}(t_{k})\sigma(t_{k})\Delta t^{2}|\le(\Vert\sigma^{\prime}\Vert_{C^{0,1}}\sigma^{\ast}+\Vert\sigma^{\prime}\Vert_{L^{\infty}}^{2})\frac{3j^{2}+3j+1}{3}\Delta t^{3},
\]
Therefore, by $\sigma(t_{k})\ge\sigma_{\ast}$,
\begin{equation}
\Big|\frac{A_{k-j}}{\sigma(t_{k})^{2}}-1+(2j+1)\frac{\sigma^{\prime}(t_{k})}{\sigma(t_{k})}\Delta t\Big|\le(\Vert\sigma^{\prime}\Vert_{C^{0,1}}\sigma^{\ast}+\Vert\sigma^{\prime}\Vert_{L^{\infty}}^{2})(j+2)^{2}\Delta t^{2}.\label{eq:-52}
\end{equation}
Note that
\begin{equation}
\Big|(2j+1)\frac{\sigma^{\prime}(t_{k})}{\sigma(t_{k})}\Delta t\Big|\le3T\Big|\frac{\sigma^{\prime}(t_{k})}{\sigma(t_{k})}\Big|\le3T\Vert\sigma^{\prime}\Vert_{L^{\infty}}/\sigma_{\ast}.\label{eq:-61}
\end{equation}
Applying $|e^{x}-1-x|\le x^{2}\sum_{k=2}^{\infty}|x|^{k-2}/k!\le x^{2}e^{|x|}$
to $x=-(2j+1)\Delta t\sigma^{\prime}(t_{k})/\sigma(t_{k})$ gives
that
\begin{equation}
\begin{aligned} & \Big|\exp\Big(-(2j+1)\frac{\sigma^{\prime}(t_{k})}{\sigma(t_{k})}\Delta t\Big)-1+(2j+1)\frac{\sigma^{\prime}(t_{k})}{\sigma(t_{k})}\Delta t\Big|\\
 & \le(2j+1)^{2}\Big|\frac{\sigma^{\prime}(t_{k})}{\sigma(t_{k})}\Big|^{2}\Delta t^{2}\exp\Big((2j+1)\Big|\frac{\sigma^{\prime}(t_{k})}{\sigma(t_{k})}\Big|\Delta t\Big)\\
 & \le4(j+2)^{2}\frac{\Vert\sigma^{\prime}\Vert_{L^{\infty}}^{2}}{\sigma_{\ast}^{2}}e^{3T\Vert\sigma^{\prime}\Vert_{L^{\infty}}/\sigma_{\ast}}\Delta t^{2}\\
 & \le4e^{3(1+T)\Vert\sigma^{\prime}\Vert_{L^{\infty}}/\sigma_{\ast}}(j+1)^{2}\Delta t^{2}.
\end{aligned}
\label{eq:-53}
\end{equation}
Combining (\ref{eq:-52}) and (\ref{eq:-53}) gives 
\begin{equation}
\Big|\frac{A_{k-j}}{\sigma(t_{k})^{2}}-\exp\Big(-(2j+1)\frac{\sigma^{\prime}(t_{k})}{\sigma(t_{k})}\Delta t\Big)\Big|\le L_{\sigma,T}(j+2)^{2}\Delta t^{2},\label{eq:-59}
\end{equation}
where $L_{\sigma,T}$ is given by (\ref{eq:-58}). Or equivalently,
\begin{equation}
\frac{A_{k-j}}{\sigma(t_{k})^{2}}=\exp\Big(-(2j+1)\frac{\sigma^{\prime}(t_{k})}{\sigma(t_{k})}\Delta t\Big)+\theta_{k,j}(j+2)^{2}\Delta t^{2}\label{eq:-261}
\end{equation}
with
\begin{equation}
|\theta_{k,j}|\le L_{\sigma,T}.\label{eq:-260}
\end{equation}
Applying (\ref{eq:-59}) to $j=-1$ gives
\begin{equation}
\Big|\frac{A_{k+1}}{\sigma(t_{k})^{2}}-\exp\Big(\frac{\sigma^{\prime}(t_{k})}{\sigma(t_{k})}\Delta t\Big)\Big|\le L_{\sigma,T}\Delta t^{2},\label{eq:-60}
\end{equation}
or equivalently,
\[
\frac{A_{k+1}}{\sigma(t_{k})^{2}}=\exp\Big(\frac{\sigma^{\prime}(t_{k})}{\sigma(t_{k})}\Delta t\Big)(1-\alpha_{k}\Delta t^{2}),
\]
with
\begin{equation}
|\alpha_{k}|\le\exp\Big(-\frac{\sigma^{\prime}(t_{k})}{\sigma(t_{k})}\Delta t\Big)L_{\sigma,T}\le e^{3T\Vert\sigma^{\prime}\Vert_{L^{\infty}}/\sigma_{\ast}}L_{\sigma,T}\le L_{\sigma,T}^{2},\label{eq:-234}
\end{equation}
where the second inequality is due to (\ref{eq:-61}). When $\Delta t<2^{-1/2}L_{\sigma,T}^{-1}$,
we have $|\alpha_{k}|\Delta t^{2}<1/2$. Consequently, by (\ref{eq:-261}),
\[
\begin{aligned} & \Big|\frac{A_{k-j}}{A_{k+1}}-\exp\Big(-2(j+1)\frac{\sigma^{\prime}(t_{k})}{\sigma(t_{k})}\Delta t\Big)\Big|\\
 & \le\exp\Big(-2(j+1)\frac{\sigma^{\prime}(t_{k})}{\sigma(t_{k})}\Delta t\Big)|(1-\alpha_{k}\Delta t^{2})^{-1}-1|+|\theta_{k,j}|(j+2)^{2}\Delta t^{2}\exp\Big(-\frac{\sigma^{\prime}(t_{k})}{\sigma(t_{k})}\Delta t\Big)(1-\alpha_{k}\Delta t^{2})^{-1}\\
 & \le L_{\sigma,T}\frac{|\alpha_{k}|\Delta t^{2}}{1-\alpha_{k}\Delta t^{2}}+L_{\sigma,T}^{2}(j+2)^{2}\Delta t^{2}(1-\alpha_{k}\Delta t^{2})^{-1}\\
 & \le10(1+T)^{2}L_{\sigma,T}^{3}(j+1)^{2}\Delta t^{2},
\end{aligned}
\]
where the second inequality is due to (\ref{eq:-61}) and (\ref{eq:-260}),
and the last inequality is due to (\ref{eq:-234}) and $|\alpha_{k}|\Delta t^{2}<1/2$.
Therefore, we see that 
\[
|a_{k,j}-b_{k,j}|\le10(1+T)^{2}L_{\sigma,T}^{3}(j+1)^{2}\Delta t^{2},
\]
which implies (\ref{eq:-268}) readily.
\end{proof}
\begin{cor}
\label{cor:-3}Let
\begin{equation}
\beta_{k}=\exp[-2\sigma^{\prime}(t_{k})/\sigma(t_{k})],\label{eq:-233}
\end{equation}
and $L_{\sigma,T}$ be the constant given by (\ref{eq:-58}). For
any $1\le n\le N$, when $\Delta t<2^{-1/2}L_{\sigma,T}^{-1}$, we
have
\begin{equation}
\begin{aligned}\big|\mathbb{E}[\langle\log I\rangle_{t_{n}}]-s_{n}^{2}(\lambda,\sigma)\big| & \le c_{T,\sigma_{\ast},\sigma^{\ast},\Vert\sigma^{\prime}\Vert_{L^{\infty}},\Vert\sigma^{\prime}\Vert_{C^{0,1}}}\lambda^{-8}\delta(\Delta t,\lambda)^{2}[1+\delta(\Delta t,\lambda)].\end{aligned}
\label{eq:-262}
\end{equation}
where 
\begin{equation}
s_{n}^{2}(\lambda,\sigma)=\frac{\bar{\sigma}^{2}\Delta t}{2(1-\lambda)}\sum_{k=0}^{n-1}\int_{0}^{\infty}e^{-t(1-\lambda)^{-1}\beta_{k}^{(k+1)\Delta t}\lambda^{k}v_{0}\sigma(0)^{-2}/2}\prod_{j=0}^{k-1}(1+t\beta_{k}^{(j+1)\Delta t}\lambda^{j})^{-1/2}dt,\label{eq:-263}
\end{equation}
and $c_{T,\sigma_{\ast},\sigma^{\ast},\Vert\sigma^{\prime}\Vert_{L^{\infty}},\Vert\sigma^{\prime}\Vert_{C^{0,1}}}$
is a constant depending only on $T$, $\sigma_{\ast}$, $\sigma^{\ast}$,
$\Vert\sigma^{\prime}\Vert_{L^{\infty}}$ and $\Vert\sigma^{\prime}\Vert_{C^{0,1}}$.
\end{cor}

\begin{proof}
Let $a_{k,j}=A_{k-j}/A_{k+1}$ and $b_{k,j}=\beta_{k}^{(j+1)\Delta t}$
for $0\le j\le k$, and $a_{k,j}=b_{k,j}=0$ for $j>k$. Clearly,
$a_{k,j}\ge(\sigma_{\ast}/\sigma^{\ast})^{2}$ and $b_{k,j}\ge e^{-2(j+1)T\Vert\sigma^{\prime}\Vert_{L^{\infty}}/\sigma^{\ast}}$
for $j\le k$. Let 
\[
\kappa=\min[(\sigma_{\ast}/\sigma^{\ast})^{2},e^{-2T\Vert\sigma^{\prime}\Vert_{L^{\infty}}/\sigma^{\ast}}]
\]
and 
\[
\gamma=e^{-2T\Vert\sigma^{\prime}\Vert_{L^{\infty}}/\sigma^{\ast}}.
\]
Then $\{a_{k}\}_{k}$ and $\{b_{k}\}_{k}$ satisfy the condition (\ref{eq:-51}).
By Lemma \ref{lem:-24}, 
\begin{equation}
\rho(a_{k},b_{k})\le c_{T,\sigma}\Delta t^{2}.\label{eq:-267}
\end{equation}
where $c_{T,\sigma}$ is a constant depending only on $T$, $\sigma_{\ast}$,
$\sigma^{\ast}$, $\Vert\sigma^{\prime}\Vert_{L^{\infty}}$ and $\Vert\sigma^{\prime}\Vert_{C^{0,1}}$.In
particular,
\begin{equation}
\begin{aligned}|A_{k+1}^{-1}-\beta_{k}^{(k+1)\Delta t}\sigma(0)^{-2}| & =\sigma(0)^{-2}|a_{k}-b_{k}|\le c_{T,\sigma}(k+1)^{2}\Delta t^{2}.\end{aligned}
\label{eq:-269}
\end{equation}
Let $\theta_{\sigma}=(\sigma_{\ast}/\sigma^{\ast})^{2}(1+\sigma^{\ast})^{-2}e^{-2T\Vert\sigma^{\prime}\Vert_{L^{\infty}}/\sigma^{\ast}}<1$.
Then $A_{k+1}^{-1}\ge(\sigma^{\ast})^{-2}\ge\theta_{\sigma}$, and
\[
\beta_{k}^{(k+1)\Delta t}\ge e^{-(k+1)\Delta t\Vert\sigma^{\prime}\Vert_{L^{\infty}}/\sigma^{\ast}}\ge e^{-2T\Vert\sigma^{\prime}\Vert_{L^{\infty}}/\sigma^{\ast}}\ge\sigma(0)^{2}\theta_{\sigma}.
\]
Denote 
\[
I_{k}=\int_{0}^{\infty}\big|e^{-t(1-\lambda)^{-1}\lambda^{k}A_{k+1}^{-1}v_{0}/2}-e^{-t(1-\lambda)^{-1}\beta_{k}^{(k+1)\Delta t}\lambda^{k}v_{0}\sigma(0)^{-2}/2}\big|\prod_{j=0}^{k-1}(1+t\lambda^{j}a_{k,j})^{-1/2}dt.
\]
By $|e^{-x}-e^{-y}|\le e^{-\min(x,y)}|x-y|$ and $\beta_{k}^{(k+1)\Delta t}\sigma(0)^{-2}\ge\theta_{\sigma}$
and $A_{k+1}^{-1}\ge\theta_{\sigma}$, we have
\[
\begin{aligned}I_{k} & \le v_{0}(1-\lambda)^{-1}\lambda^{k}|A_{k+1}^{-1}-\beta_{k}^{(k+1)\Delta t}\sigma(0)^{-2}|\int_{0}^{\infty}e^{-t(1-\lambda)^{-1}\lambda^{k}\theta_{\sigma}v_{0}/2}t\prod_{j=0}^{k-1}(1+t\lambda^{j}a_{k,j})^{-1/2}dt\\
 & \le c_{\sigma,v_{0}}\theta_{\sigma}^{-2}(1-\lambda)^{-1}\Delta t^{2}(k+1)^{2}\lambda^{k}\int_{0}^{\infty}e^{-t(1-\lambda)^{-1}\lambda^{k}v_{0}/2}t\prod_{j=0}^{k-1}(1+t\lambda^{j})^{-1/2}dt,
\end{aligned}
\]
where the last inequality is due to $a_{k,j}\le(\sigma_{\ast}/\sigma^{\ast})^{2}\le\theta_{\sigma}$
and a simple change of variable. Therefore, by Lemma \ref{lem:-23},
\begin{equation}
\frac{\Delta t}{1-\lambda}\sum_{k=0}^{n-1}I_{k}\le c_{\sigma,v_{0}}\theta_{\sigma}^{-2}\lambda^{-4}(1-\lambda)^{-3}\Delta t^{3}\le c_{T,\sigma,v_{0}}\lambda^{-4}\delta(\Delta t,\lambda)^{3}.\label{eq:-271}
\end{equation}
Moreover, by (\ref{eq:-267}) and Lemma \ref{lem:-22} with $\theta_{k}=\beta_{k}^{(k+1)\Delta t}v_{0}\sigma(0)^{-2}/2\ge\theta_{\sigma}v_{0}/2$,
we see that 
\begin{equation}
\frac{\Delta t}{1-\lambda}\sum_{k=0}^{n-1}I_{k}^{\prime}\le c_{T,\sigma,v_{0}}\lambda^{-8}\delta(\Delta t,\lambda)^{2}[1+\delta(\Delta t,\lambda)],\label{eq:-270}
\end{equation}
where
\[
I_{k}^{\prime}=\int_{0}^{\infty}e^{-t(1-\lambda)^{-1}\lambda^{k}\beta_{k}^{(k+1)\Delta t}v_{0}\sigma(0)^{-2}/2}\Big|\prod_{j=0}^{k-1}(1+t\lambda^{j}a_{k,j})^{-1/2}-\prod_{j=0}^{k-1}(1+t\lambda^{j}b_{k,j})^{-1/2}\Big|dt.
\]
Combining (\ref{eq:-271}) and (\ref{eq:-270}) completes the proof.
\end{proof}
In view of Corollary \ref{cor:-3}, the first order approximation
of $\mathbb{E}[\langle\log I\rangle_{t_{n}}]$ in terms of $\delta(\Delta t,\lambda)$
is the same as that of $s_{n}^{2}(\lambda,\sigma)$. As will be seen
in the proof of Proposition \ref{prop:}, by writing 
\[
\prod_{j=0}^{k-1}(1+t\beta_{k}^{(j+1)\Delta t}\lambda^{j})^{-1/2}=\prod_{j=0}^{k-1}[1+t\beta_{k}^{\Delta t}(\beta_{k}^{\Delta t}\lambda)^{j}]^{-1/2},
\]
the factor $\beta_{k}^{(j+1)\Delta t}$ can be moved out from the
integrand by applying Taylor's theorem to $\int_{0}^{\infty}e^{-t\theta(1-\lambda)^{-1}\lambda^{k}}\prod_{j=0}^{k-1}(1+t\lambda^{j})^{-1/2}dt$
as a function of $\lambda$. 

We now turn to the proof of Proposition \ref{prop:} for the discrete
correction.
\begin{proof}[Proof of Proposition \ref{prop:}]
It suffices to prove (\ref{eq:-7}) for $\tau=t_{n}=n\Delta t$ satisfying
$t_{n}\ge\tau_{\ast}$. Throughout the proof, we denote by $c_{q_{1},\dots,q_{m},\sigma}$
a constant depending only on $q_{1},\dots,q_{m}$ and $\sigma_{\ast}$,
$\sigma^{\ast}$, $\Vert\sigma^{\prime}\Vert_{L^{\infty}}$ and $\Vert\sigma^{\prime}\Vert_{C^{0,1}}$,
and the value of $c_{q_{1},\dots,q_{m},\sigma}$ may vary at different
appearances. Let $s_{n}^{2}(\lambda,\sigma)$ be given by (\ref{eq:-263}).
When $\delta(\Delta t,\lambda)\le2^{-1}\sigma_{\ast}/(1+\Vert\sigma^{\prime}\Vert_{L^{\infty}})$,
by $\log\lambda^{-1}\ge1-\lambda$, we have
\[
\beta_{k}^{\Delta t}\le e^{\Delta t\Vert\sigma^{\prime}\Vert_{L^{\infty}}/\sigma_{\ast}}=e^{\delta(\Delta t,\lambda)(1-\lambda)\Vert\sigma^{\prime}\Vert_{L^{\infty}}/\sigma_{\ast}}\le\lambda^{-1/2}.
\]
Similarly, we have $\beta_{k}^{\Delta t}\ge\lambda^{1/2}$. Therefore,
\begin{equation}
\lambda^{1/2}\le\beta_{k}^{\Delta t}\le\lambda^{-1/2}.\label{eq:-264}
\end{equation}
Let $V_{k}(\lambda,\theta)$ be the function defined by (\ref{eq:-5});
that is,
\[
V_{k}(\lambda,\theta)=\frac{1}{2(1-\lambda)}\int_{0}^{\infty}e^{-t\theta(1-\lambda)^{-1}\lambda^{k}/2}\prod_{j=0}^{k-1}(1+t\lambda^{j})^{-1/2}dt.
\]
It is easily seen by a change of variable that 
\[
s_{n}^{2}(\lambda,\sigma)=\bar{\sigma}^{2}\Delta t\sum_{k=0}^{n-1}\frac{\beta_{k}^{-\Delta t}-\lambda}{1-\lambda}V_{k}\Big(\beta_{k}^{\Delta t}\lambda,\frac{1-\beta_{k}^{\Delta t}\lambda}{1-\lambda}\theta_{0}\Big),
\]
where $\theta_{0}=v_{0}\sigma(0)^{-2}$. Let 
\begin{equation}
\phi(x,\lambda)=\Big(x\lambda,\frac{1-x\lambda}{1-\lambda}\theta_{0}\Big).\label{eq:-265}
\end{equation}
Applying Taylor's theorem to the function 
\begin{equation}
f_{k}(x)=\frac{x^{-1}-\lambda}{1-\lambda}V_{k}(\phi(x,\lambda))\label{eq:-266}
\end{equation}
at $x=1$ gives 
\begin{equation}
\bar{\sigma}^{-2}s_{n}^{2}(\lambda,\sigma)=\Delta t\sum_{k=0}^{n-1}f_{k}(1)+\Delta t\sum_{k=0}^{n-1}f_{k}^{\prime}(1)(\beta_{k}^{\Delta t}-1)+\Delta t\sum_{k=0}^{n-1}\int_{1}^{\beta_{k}^{\Delta t}}f_{k}^{\prime\prime}(x)(\beta_{k}^{\Delta t}-x)dx.\label{eq:-289}
\end{equation}

For the first sum in the right-hand side of (\ref{eq:-289}), by (\ref{eq:-4}),
we have
\begin{equation}
\begin{aligned}\Delta t\sum_{k=0}^{n-1}f_{k}(1) & =\Delta t\sum_{k=0}^{n-1}V_{k}(\lambda,\theta_{0})=V(\lambda)n\Delta t+\Delta t\sum_{k=0}^{n-1}[V_{k}(\lambda,\theta_{0})-V(\lambda)]\\
 & =V(\lambda)t_{n}+Q_{V}(\lambda,\theta_{0})\Delta t+\Delta t\sum_{k=n}^{\infty}[V_{k}(\lambda,\theta_{0})-V(\lambda)].
\end{aligned}
\label{eq:-28}
\end{equation}
By Proposition \ref{prop:-9}, we have
\[
\begin{aligned}\Delta t\sum_{k=n}^{\infty}|V_{k}(\lambda,\theta_{0})-V(\lambda)| & \le c_{\lambda_{\ast},\theta_{0}}\lambda^{n}\delta(\Delta t,\lambda).\end{aligned}
\]
Note that, when $t_{n}\ge\tau_{\ast}$, we have
\[
t_{n}\Delta t^{-1}\log\lambda^{-1}=\frac{t_{n}}{\delta(\Delta t,\lambda)}\frac{\log\lambda^{-1}}{1-\lambda}\ge\frac{t_{n}\lambda}{\delta(\Delta t,\lambda)}\ge\log\delta(\Delta t,\lambda)^{-1},
\]
and therefore, $\lambda^{n}=e^{-t_{n}\Delta t^{-1}\log\lambda^{-1}}\le\delta(\Delta t,\lambda)$,
which implies that
\begin{equation}
\begin{aligned}\Delta t\sum_{k=n}^{\infty}|V_{k}(\lambda,\theta_{0})-V(\lambda)| & \le c_{\lambda_{\ast},\theta_{0}}\delta(\Delta t,\lambda)^{2}.\end{aligned}
\label{eq:-257}
\end{equation}
It follows from (\ref{eq:-28}) and (\ref{eq:-257}) that
\begin{equation}
\Delta t\sum_{k=0}^{n-1}f_{k}(1)=V(\lambda)t_{n}+Q_{V}(\lambda,\theta_{0})\Delta t+\text{O}(\delta(\Delta t,\lambda)^{2}).\label{eq:-230}
\end{equation}

For the second sum on the right-hand side of (\ref{eq:-289}), by
Taylor's theorem, 
\begin{equation}
\beta_{k}^{\Delta t}-1=-\frac{2\sigma^{\prime}(t_{k})}{\sigma(t_{k})}\Delta t+r_{k}\Delta t^{2},\label{eq:-297}
\end{equation}
with $|r_{k}|\le c_{\sigma}$. It is easily seen that
\[
\begin{aligned}f_{k}^{\prime}(1) & =-(1-\lambda)^{-1}V_{k}(\lambda,\theta_{0})+\lambda\partial_{\lambda}V_{k}(\lambda,\theta_{0})-\lambda\theta_{0}(1-\lambda)^{-1}\partial_{\theta}V_{k}(\lambda,\theta_{0}).\end{aligned}
\]
Therefore, 
\[
\begin{aligned}\Delta t\sum_{k=0}^{n-1}f_{k}^{\prime}(1)(\beta_{k}^{\Delta t}-1) & =E_{1}+E_{2}+E_{3}+E_{4},\end{aligned}
\]
where
\begin{align*}
E_{1} & =[-(1-\lambda)^{-1}V(\lambda)+\lambda V^{\prime}(\lambda)]\Delta t\sum_{k=0}^{n-1}(\beta_{k}^{\Delta t}-1),\\
E_{2} & =-(1-\lambda)^{-1}\Delta t\sum_{k=0}^{n-1}[V_{k}(\lambda,\theta_{0})-V(\lambda)](\beta_{k}^{\Delta t}-1),\\
E_{3} & =\lambda\Delta t\sum_{k=0}^{n-1}[\partial_{\lambda}V_{k}(\lambda,\theta_{0})-V^{\prime}(\lambda)](\beta_{k}^{\Delta t}-1),\\
E_{4} & =-\lambda\theta_{0}(1-\lambda)^{-1}\Delta t\sum_{k=0}^{n-1}\partial_{\theta}V_{k}(\lambda,\theta_{0})(\beta_{k}^{\Delta t}-1).
\end{align*}
For $E_{1}$, by (\ref{eq:-297}), we have
\begin{equation}
\begin{aligned}E_{1} & =[(1-\lambda)^{-1}V(\lambda)-\lambda V^{\prime}(\lambda)]\int_{0}^{t_{n}}\frac{2\sigma^{\prime}(t)}{\sigma(t)}dt+\text{O}(\Delta t^{2})\\
 & =2[(1-\lambda)^{-1}V(\lambda)-\lambda V^{\prime}(\lambda)]\log\frac{\sigma(t_{n})}{\sigma(0)}+\text{O}(\Delta t^{2}).
\end{aligned}
\label{eq:-298}
\end{equation}
For $E_{2}$, since $|\beta_{k}^{\Delta t}-1|\le c_{\sigma}\Delta t$,
by Proposition \ref{prop:-9}, we have
\begin{equation}
\begin{aligned}|E_{2}| & \le c_{\sigma}(1-\lambda)^{-1}\Delta t^{2}\sum_{k=0}^{\infty}|V_{k}(\lambda,\theta_{0})-V(\lambda,\theta_{0})|\\
 & \le c_{\lambda_{\ast},\sigma}\delta(\Delta t,\lambda)^{2}.
\end{aligned}
\label{eq:-299}
\end{equation}
Similarly, by Proposition \ref{prop:-8} and Proposition \ref{prop:-11},
we have
\begin{equation}
|E_{3}|\le c_{\lambda_{\ast},\sigma}\delta(\Delta t,\lambda)^{2},\label{eq:-300}
\end{equation}
\begin{equation}
|E_{4}|\le c_{\lambda_{\ast},\sigma}\delta(\Delta t,\lambda)^{2}.\label{eq:-301}
\end{equation}
Combining (\ref{eq:-298}), (\ref{eq:-299}), (\ref{eq:-300}) and
(\ref{eq:-301}) gives
\begin{equation}
\Delta t\sum_{k=0}^{n-1}f_{k}^{\prime}(1)(\beta_{k}^{\Delta t}-1)=2[(1-\lambda)^{-1}V(\lambda)-\lambda V^{\prime}(\lambda)]\log\frac{\sigma(t_{n})}{\sigma(0)}+\text{O}(\delta(\Delta t,\lambda)^{2}).\label{eq:-302}
\end{equation}

It remains to show that 
\begin{equation}
\Delta t\sum_{k=0}^{n-1}\Big|\int_{1}^{\beta_{k}^{\Delta t}}f_{k}^{\prime\prime}(x)(\beta_{k}^{\Delta t}-x)dx\Big|\le c_{\lambda_{\ast},\theta_{0},\sigma,T}\delta(\Delta t,\lambda)^{2}.\label{eq:-287}
\end{equation}
It is easily seen that
\begin{equation}
\begin{aligned}f^{\prime\prime}(x) & =\frac{2x^{-3}}{1-\lambda}V_{k}(\phi)-\frac{2x^{-2}\lambda}{1-\lambda}\Big(\partial_{\lambda}V_{k}(\phi)-\frac{\theta_{0}}{1-\lambda}\partial_{\theta}V_{k}(\phi)\Big)\\
 & \quad+\frac{x^{-1}-\lambda}{1-\lambda}\Big(\lambda^{2}\partial_{\lambda}^{2}V_{k}(\phi)-\frac{2\lambda^{2}\theta_{0}}{1-\lambda}\partial_{\lambda}\partial_{\theta}V_{k}(\phi)+\frac{\lambda^{2}\theta_{0}^{2}}{(1-\lambda)^{2}}\partial_{\theta}^{2}V_{k}(\phi)\Big),
\end{aligned}
\label{eq:-288}
\end{equation}
where we denote $V_{k}(\phi)=V_{k}(\phi(x,\lambda))$ to simplify
notation. Note that, by (\ref{eq:-264}), we have
\[
\frac{1}{2}\le\frac{1-\beta_{k}^{\Delta t}\lambda}{1-\lambda}\le1.
\]
Therefore, when $x$ is between $1$ and $\beta_{k}^{\Delta t}$,
we have
\begin{equation}
\lambda^{1/2}\le x\le\lambda^{-1/2},\label{eq:-303}
\end{equation}
and therefore,
\begin{equation}
\theta_{0}/2\le\frac{1-x\lambda}{1-\lambda}\theta_{0}\le\theta_{0},\label{eq:-272}
\end{equation}
\begin{equation}
(1-\lambda)/2\le1-x\lambda\le1-\lambda,\label{eq:-273}
\end{equation}
and 
\begin{equation}
\Big|\frac{x^{-1}-\lambda}{1-\lambda}\Big|\le\frac{\lambda^{-1/2}-\lambda}{1-\lambda}\le2\lambda_{\ast}^{-1/2}.\label{eq:-231}
\end{equation}
Moreover, by 
\[
e^{-\delta(\Delta t,\lambda)(1-\lambda)\Vert\sigma^{\prime}\Vert_{L^{\infty}}/\sigma^{\ast}}\le\beta_{k}^{\Delta t}\le e^{\delta(\Delta t,\lambda)(1-\lambda)\Vert\sigma^{\prime}\Vert_{L^{\infty}}/\sigma_{\ast}},
\]
and $\delta(\Delta t,\lambda)\le2^{-1}\sigma_{\ast}/(1+\Vert\sigma^{\prime}\Vert_{L^{\infty}})$,
we have
\begin{equation}
|\beta^{\Delta t}-1|\le2(1-\lambda)\delta_{1}(\Delta t,\lambda)\Vert\sigma^{\prime}\Vert_{L^{\infty}}/\sigma_{\ast}.\label{eq:-275}
\end{equation}

By Proposition \ref{prop:-9}, 
\[
\sum_{k=1}^{\infty}|V_{k}(\phi(x,\lambda))-V(x\lambda)|\le c_{0}(\theta_{0}+\theta_{0}^{-1})(x\lambda)^{-4}(1-x\lambda)^{-1}\le c_{\theta_{0},\lambda_{\ast}}(1-\lambda)^{-1},
\]
where $c_{\theta_{0}}$ is a constant depending only on $\theta_{0}$.
Consequently, 
\begin{equation}
\begin{aligned} & \Delta t\sum_{k=0}^{n-1}\Big|\int_{1}^{\beta_{k}^{\Delta t}}\frac{x^{-3}}{1-\lambda}V_{k}(\phi(x,\lambda))(\beta_{k}^{\Delta t}-x)dx\Big|\\
 & \le\Delta t\sum_{k=0}^{n-1}\Big|\int_{1}^{\beta_{k}^{\Delta t}}\frac{x^{-3}}{1-\lambda}[V_{k}(\phi(x,\lambda))-V(x\lambda)](\beta_{k}^{\Delta t}-x)dx\Big|+\Delta t\sum_{k=0}^{n-1}\Big|\int_{1}^{\beta_{k}^{\Delta t}}\frac{x^{-3}}{1-\lambda}V(x\lambda)(\beta_{k}^{\Delta t}-x)dx\Big|\\
 & \le c_{\sigma}\Delta t\lambda^{-3/2}\delta(\Delta t,\lambda)\Big(\int_{\lambda^{1/2}}^{\lambda^{-1/2}}\sum_{k=0}^{\infty}|V_{k}(\phi(x,\lambda))-V(x\lambda)|dx+\sum_{k=0}^{n-1}\Big|\int_{1}^{\beta_{k}^{\Delta t}}V(x\lambda)dx\Big|\Big)\\
 & \le c_{\lambda_{\ast},\theta_{0},\sigma}\Delta t\delta(\Delta t,\lambda)[(1-\lambda)^{-1}(\lambda^{-1/2}-\lambda^{1/2})+n(1-\lambda)\delta(\Delta t,\lambda)]\\
 & \le c_{\lambda_{\ast},\theta_{0},\sigma,T}\delta(\Delta t,\lambda)^{2},
\end{aligned}
\label{eq:-274}
\end{equation}
where the second inequality is due to $x\ge\lambda^{1/2}$ and $|\beta_{k}^{\Delta t}-x|\le|\beta_{k}^{\Delta t}-1|\le c_{\sigma}(1-\lambda)\delta(\Delta t,\lambda)$
for $x$ between $1$ and $\beta_{k}^{\Delta t}$, the third inequality
is due to (\ref{eq:-280}), and the last inequality follows from $n\Delta t\le T$.

Similarly, by Proposition \ref{prop:-8}, we have
\[
\sum_{k=1}^{\infty}|\partial V_{k}(\phi(x,\lambda))-V^{\prime}(x\lambda)|\le c_{\lambda_{\ast},\theta_{0}}(1-\lambda)^{-2}.
\]
Similar to (\ref{eq:-274}), we have
\begin{equation}
\begin{aligned} & \Delta t\sum_{k=0}^{n-1}\Big|\int_{1}^{\beta_{k}^{\Delta t}}\frac{x^{-2}}{1-\lambda}\partial_{\lambda}V_{k}(\phi(x,\lambda))(\beta_{k}^{\Delta t}-x)dx\Big|\\
 & \le c_{\sigma}\Delta t\lambda^{-1}\delta(\Delta t,\lambda)\Big(\int_{\lambda^{1/2}}^{\lambda^{-1/2}}\sum_{k=0}^{\infty}|\partial_{\lambda}V_{k}(\phi(x,\lambda))-V^{\prime}(\lambda)|dx+\sum_{k=0}^{n-1}\Big|\int_{1}^{\beta_{k}^{\Delta t}}V^{\prime}(x\lambda)dx\Big|\Big)\\
 & \le c_{\lambda_{\ast},\theta_{0},\sigma,T}\Delta t\delta(\Delta t,\lambda)[(1-\lambda)^{-2}(\lambda^{-1/2}-\lambda^{1/2})+n(1-\lambda)^{-1}(1-\lambda)\delta(\Delta t,\lambda)]\\
 & \le c_{\lambda_{\ast},\theta_{0},\sigma,T}\delta(\Delta t,\lambda)^{2},
\end{aligned}
\label{eq:-285}
\end{equation}
where the second inequality is due to (\ref{eq:-281}). 

By Proposition \ref{prop:-10}, 
\[
\sum_{k=1}^{\infty}|\partial_{\lambda}^{2}V_{k}(\phi(x,\lambda))-V^{\prime\prime}(x\lambda)|\le c_{\lambda_{\ast},\theta_{0}}(1-\lambda)^{-3}.
\]
Therefore, similar to (\ref{eq:-274}), we have
\begin{equation}
\begin{aligned} & \Delta t\sum_{k=0}^{n-1}\Big|\int_{1}^{\beta_{k}^{\Delta t}}\frac{x^{-1}-\lambda}{1-\lambda}\partial_{\lambda}^{2}V_{k}(\phi(x,\lambda))(\beta_{k}^{\Delta t}-x)dx\Big|\\
 & \le c_{\sigma}\Delta t(1-\lambda)\delta(\Delta t,\lambda)\Big(\int_{\lambda^{1/2}}^{\lambda^{-1/2}}\sum_{k=0}^{\infty}|\partial_{\lambda}^{2}V_{k}(\phi(x,\lambda))-V^{\prime\prime}(x\lambda)|dx+\sum_{k=0}^{n-1}\Big|\int_{1}^{\beta_{k}^{\Delta t}}V^{\prime\prime}(x\lambda)dx\Big|\Big)\\
 & \le c_{\lambda_{\ast},\theta_{0},\sigma}\Delta t(1-\lambda)\delta(\Delta t,\lambda)[(1-\lambda)^{-3}(\lambda^{-1/2}-\lambda^{1/2})+n(1-\lambda)\delta_{1}(\Delta t,\lambda)(1-\lambda)^{-2}]\\
 & \le c_{\lambda_{\ast},\theta_{0},\sigma,T}\delta(\Delta t,\lambda)^{2},
\end{aligned}
\label{eq:-284}
\end{equation}
where the first inequality is due to (\ref{eq:-231}) and (\ref{eq:-275}),
and the second inequality is due to (\ref{eq:-282}).

Similar to the above, using Proposition \ref{prop:-11}, it can be
easily shown that 
\begin{equation}
\Delta t\sum_{k=0}^{n-1}\Big|\int_{1}^{\beta_{k}^{\Delta t}}\frac{x^{-2}}{(1-\lambda)^{2}}\partial_{\theta}V_{k}(\phi(x,\lambda))(\beta_{k}^{\Delta t}-x)dx\Big|\le c_{\lambda_{\ast},\theta_{0},\sigma,T}\delta(\Delta t,\lambda)^{2},\label{eq:-286}
\end{equation}
\begin{equation}
\Delta t\sum_{k=0}^{n-1}\Big|\int_{1}^{\beta_{k}^{\Delta t}}\frac{x^{-1}-\lambda}{(1-\lambda)^{2}}\partial_{\lambda}\partial_{\theta}V_{k}(\phi(x,\lambda))(\beta_{k}^{\Delta t}-x)dx\Big|\le c_{\lambda_{\ast},\theta_{0},\sigma,T}\delta(\Delta t,\lambda)^{2},\label{eq:-304}
\end{equation}
and 
\begin{equation}
\Delta t\sum_{k=0}^{n-1}\Big|\int_{1}^{\beta_{k}^{\Delta t}}\frac{x^{-1}-\lambda}{(1-\lambda)^{3}}\partial_{\theta}^{2}V_{k}(\phi(x,\lambda))(\beta_{k}^{\Delta t}-x)dx\Big|\le c_{\lambda_{\ast},\theta_{0},\sigma,T}\delta(\Delta t,\lambda)^{2}.\label{eq:-17}
\end{equation}
Combining (\ref{eq:-288}), (\ref{eq:-274}), (\ref{eq:-285}), (\ref{eq:-284}),
(\ref{eq:-286}), (\ref{eq:-304}) and (\ref{eq:-17}) proves (\ref{eq:-287}).
\end{proof}

\section{\label{sec:}Numerical Results}

In this section, we present numerical results which support the conclusions
of the paper. 

Figure \ref{fig:} shows the probability density comparison of $\log I_{T}$
and $\log X_{T}$ computed from $100000$ simulated samples with $\Delta t=0.0002$
under the local volatility model; that is $dS_{t}=(r-q)S_{t}dt+\sigma(t,S_{t})S_{t}dW_{t}$,
where $\sigma(t,x)$ is a deterministic function. For this test, we
use the following surface SVI model (see \cite[Section 4]{GJ14})
\begin{equation}
w(T,K)=\frac{\theta_{T}}{2}\Big(1+\theta_{T}\log\frac{K}{F_{T}}+\sqrt{\Big(\varphi(\theta_{T})\log\frac{K}{F_{T}}+\rho\Big)^{2}+(1-\rho)^{2}}\Big)\label{eq:-32}
\end{equation}
for the implied volatility surface construction, where $w(T,K)$ is
the implied total variance, $\theta_{T}$ is the at-the-money variance,
$F_{T}=e^{(r-q)T}$ is the forward value, and $\varphi(\theta)=\theta^{-1/2}(1+\theta)^{-1/2}$.
The local volatility function $\sigma(t,x)$ is computed using the
above implied volatility and Dupire's formula
\[
\sigma(T,K)^{2}=\frac{\partial_{T}C+(r-q)K\partial_{K}C+qC}{2^{-1}K^{2}\partial_{K}^{2}C},
\]
where $C=C(T,K)$ is the price function of the European call option
with maturity $T$ and strike $K$. The parameters we use for (\ref{eq:-32})
are $r=0.05$, $q=0.03$, $\rho=-0.9$ and $\theta_{T}=\int_{0}^{T}[\sigma_{1}(1-e^{-\kappa t})/(\kappa t)+\sigma_{2}]^{2}dt$.
And the parameters for volatility target are $T=1.0$, $\lambda=0.95$,
$v_{0}=0.01$, $\bar{\sigma}=0.1$. As reference, the density of a
normal distribution with sample mean and sample variance computed
from the samples of $\log I_{T}$ is also shown in Figure \ref{fig:}.
\begin{figure}[H]

\begin{centering}
\includegraphics[scale=0.5]{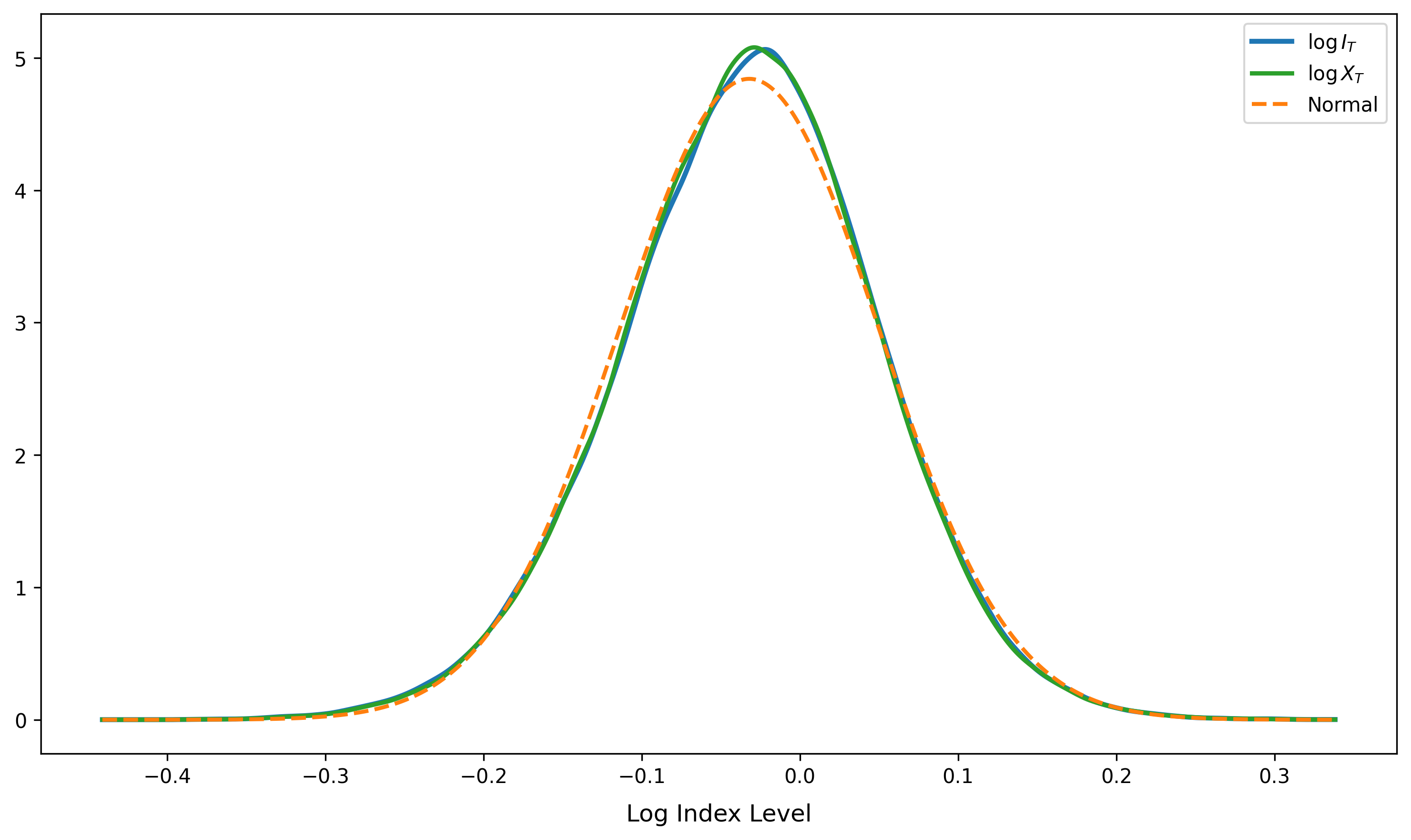}\caption{Probability Density of $\log I_{T}$ and $\log X_{T}$ under the Local
Volatility Model}
\label{fig:}
\par\end{centering}
\end{figure}

Figure \ref{fig:-1} shows the probability density comparison of $\log I_{T}$
and $\log X_{T}$ computed from $100000$ simulated samples with $\Delta t=0.0002$
under the Heston stochastic volatility model; that is 
\begin{equation}
\begin{aligned}dS_{t} & =(r-q)S_{t}dt+\sqrt{v_{t}}S_{t}dW_{t}^{1},\\
dv_{t} & =\kappa(\theta-v_{t})dt+\sigma\sqrt{v_{t}}dW_{t}^{2},
\end{aligned}
\label{eq:-6}
\end{equation}
with $d\langle W^{1},W^{2}\rangle_{t}=\rho dt$. We should point out
that, when Feller's condition $2\kappa\theta\ge\sigma^{2}$ is violated,
the volatility process $v_{t}$ can touch zero (see \cite[Lemma 4]{Feller51}).
In this case, a floor $\sigma_{\ast}=0.001$ is applied to the dynamics
of the underlying risky asset; that is, the dynamics (\ref{eq:-6})
is replaced by $dS_{t}=(r-q)S_{t}dt+\max(\sigma_{\ast},\sqrt{v_{t}})S_{t}dW_{t}^{1}$.
The parameters we use for (\ref{eq:-6}) are $r=0.05$, $q=0.03$,
$\theta=0.04$, $\kappa=0.3$, $\sigma=0.5$, $\rho=-0.8$, and the
parameters for volatility target are $T=1.0$, $\lambda=0.95$, $v_{0}=0.01$,
$\bar{\sigma}=0.1$. Similar to the test for the local volatility
model, the density of a normal distribution with sample mean and sample
variance computed from the samples of $\log I_{T}$ is also shown
in Figure \ref{fig:-1} as reference. 

\begin{figure}[H]
\centering{}\includegraphics[scale=0.5]{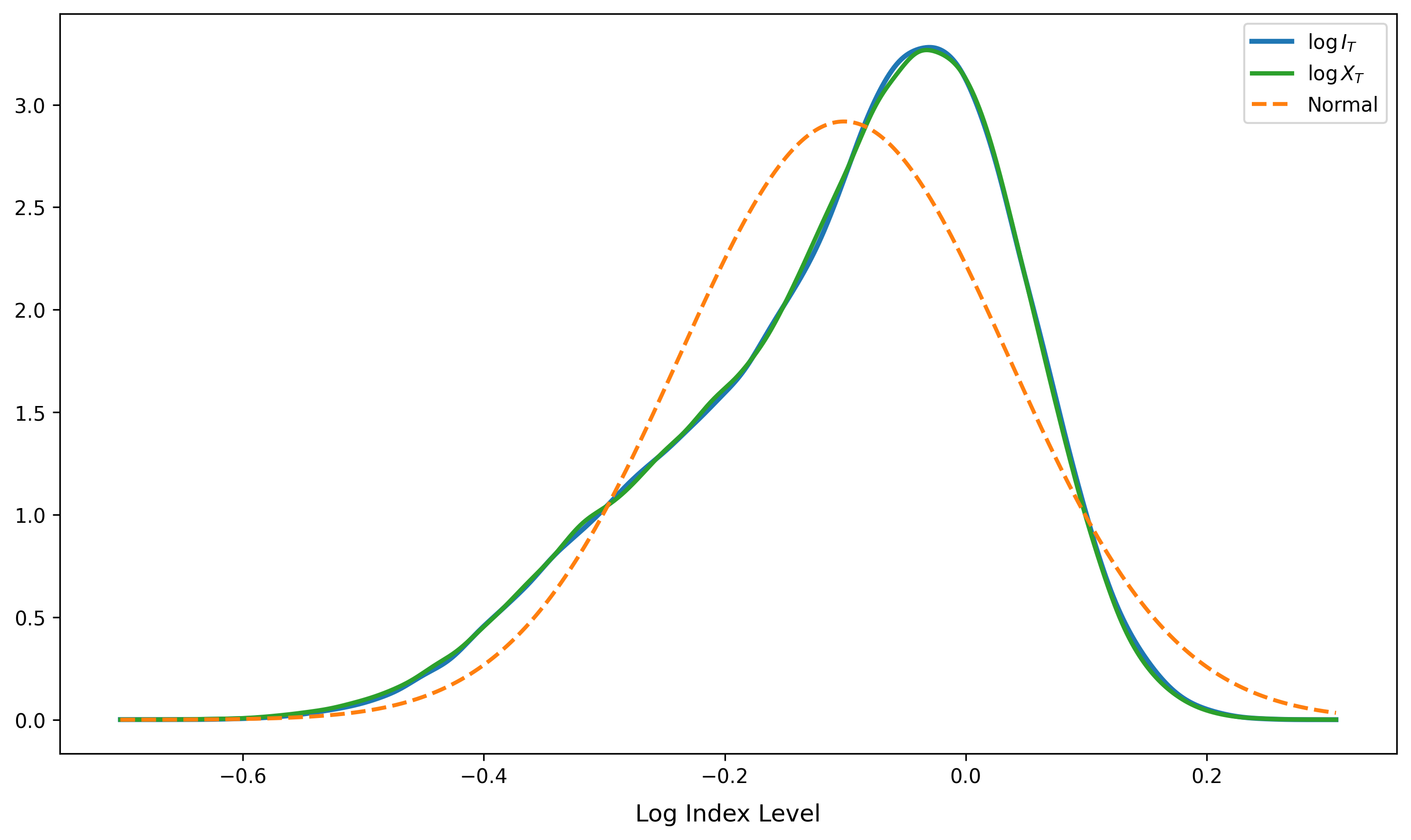}\caption{Probability Density of $\log I_{T}$ and $\log X_{T}$ under the Heston
Stochastic Volatility Model}
\label{fig:-1}
\end{figure}

Figure \ref{fig:-2} shows the approximation quality of the discrete
correction formula (\ref{eq:-7}). The parameters we use are $T=1.0$,
$\lambda=0.9$, $v_{0}=0.01$, $\bar{\sigma}=0.1$, $r=0.05$, $\rho=0.02$
and $\sigma(t)=\sigma_{1}(1-e^{-\kappa t})/(\kappa t)+\sigma_{2}$
with $\sigma_{1}=0.3$, $\sigma_{2}=0.2$ and $\kappa=3$. 
\begin{figure}[H]
\centering{}\includegraphics[scale=0.5]{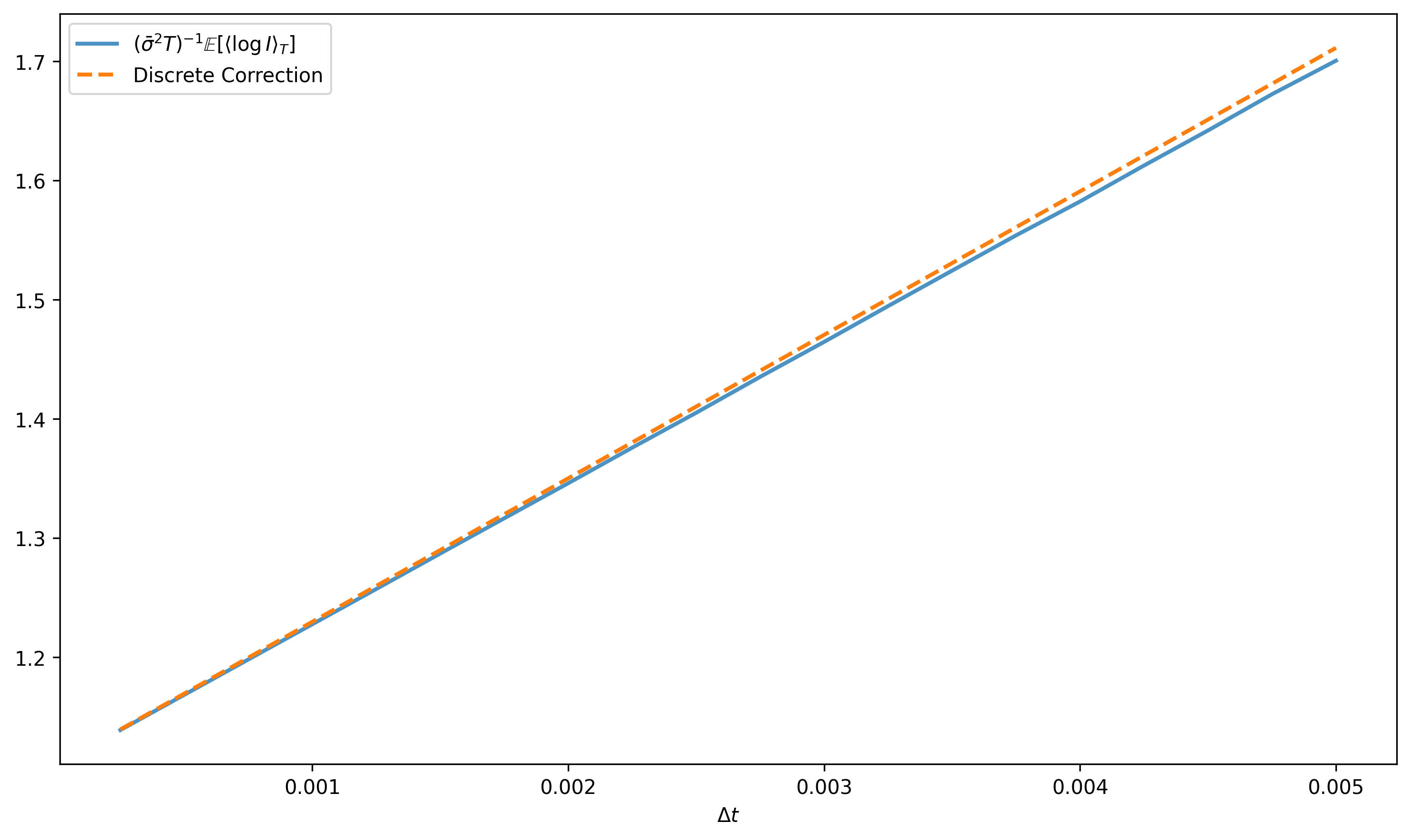}\caption{Approximation Quality of the Discrete Correction Formula}
\label{fig:-2}
\end{figure}

Figure \ref{fig:-3} compares the pricing of an European call option
with maturity $T=1.0$ and strike $K=1.0$ on volatility target indices
with different values of $\Delta t$ and $\lambda$. The rest of the
parameters we use are $v_{0}=0.01$, $\bar{\sigma}=0.1$, $r=0.05$,
$\rho=0.02$ and $\sigma(t)=\sigma_{1}(1-e^{-\kappa t})/(\kappa t)+\sigma_{2}$
with $\sigma_{1}=0.3$, $\sigma_{2}=0.2$ and $\kappa=3$. The solid
curves in Figure \ref{fig:-3} are computed using Monte Carlo simulations
with $150000$ samples of $I_{T}$, while the dotted curves are computed
using the Black--Scholes formula together with the discrete correction
formulae (\ref{eq:-22}) and (\ref{eq:-21}), and the dashed curve
is computed using the Black--Scholes formula and the limiting distribution
(\ref{eq:}). 

\begin{figure}[H]
\centering{}\includegraphics[scale=0.5]{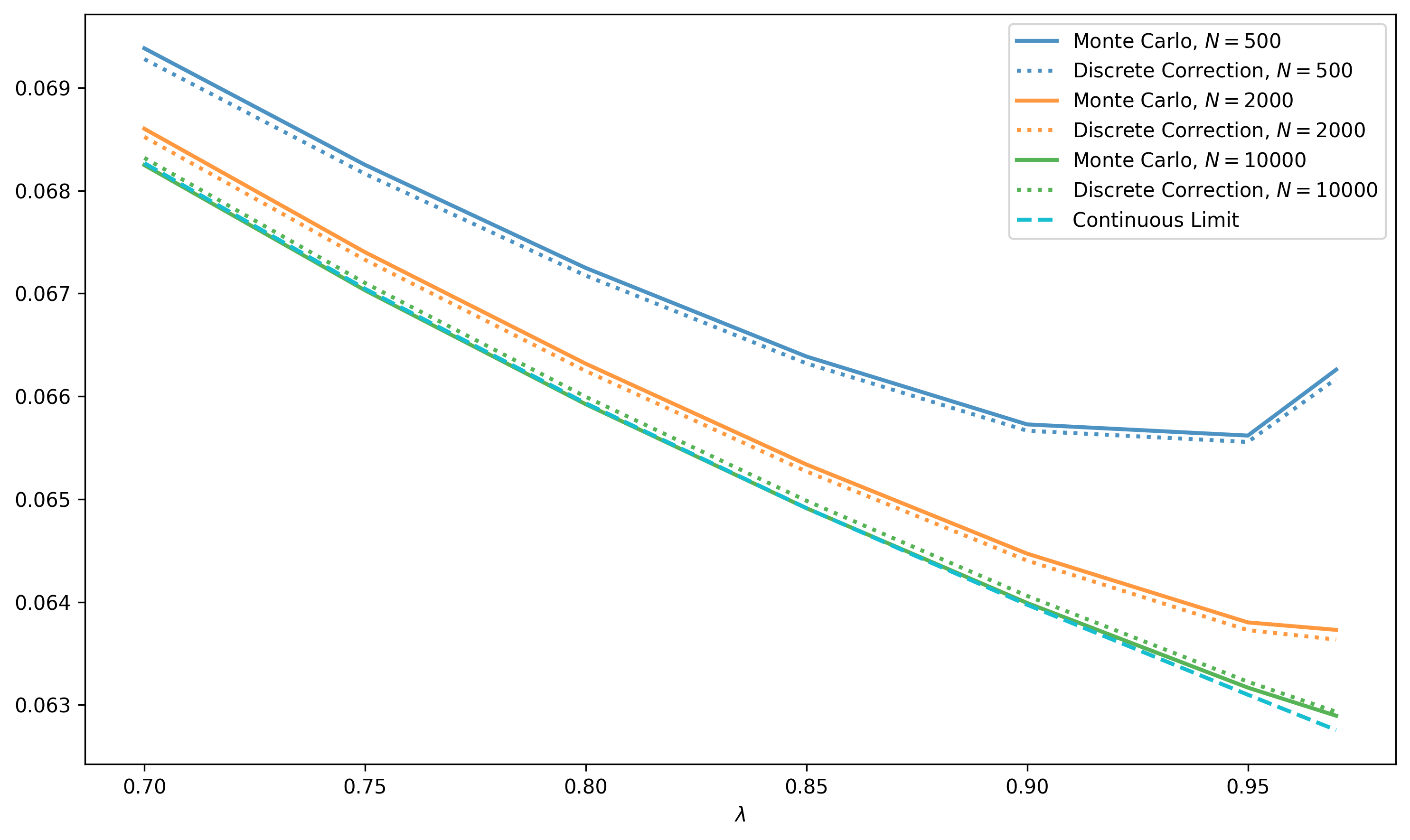}\caption{Pricing of European Call Options on $\{I_{t}\}_{t\ge0}$}
\label{fig:-3}
\end{figure}

\section{\label{sec:-6}Conclusion}

In this paper, under the mild boundedness assumption (\ref{eq:-10})
on the stochastic volatility process $\{\sigma_{t}\}_{t\ge0}$, we
have established the structural limit of volatility target indices
as the rebalancing time step $\Delta t\to0$. Compared to the limit
theorem in \cite{LG26}, Theorem \ref{thm:} shows that, when the
distributional limit of the joint process of the volatility target
and the driver Brownian motions is considered, an extra independent
Brownian motion must be involved; though the volatility, given by
$\sqrt{1-U(\lambda)^{2}/V(\lambda)}$ in (\ref{eq:-35}), in front
of this extra Brownian motion is small for commonly used values of
$\lambda$ in practice. It is also shown that the convergence in distribution
in Theorem \ref{thm:} cannot be improved to pathwise convergence
or convergence in probability. Besides the structural limit theorem,
a discrete correction for the expected quadratic variation of $\{\log I_{t}\}_{t\ge0}$
is derived, together with a similar correction formula for the expected
drift, which can be regarded as moment matching formulae. Moreover,
as a corollary of the discrete correction, we give a complete characterization
of the competing effect between $\Delta t\to0$ and $\lambda\to1$,
which is partially studied in \cite[Section 4]{LG26}. We should point
out that the structural limit theorem in this paper relies heavily
on the assumption that the risky asset is driven by Brownian motions.
This assumption is implicitly used when the martingale central limit
theorem is applied to derive the marginal limiting distribution (Proposition
\ref{prop:-4}); it is also used in the proof of Lemma \ref{lem:-7},
where the Wiener chaos decomposition is applied to give the approximating
volatility processes with simple explicit forms. Though the derivations
in this paper do not immediately extend to dynamics such as jump diffusion
which is not driven by Brownian motions only, it is reasonable to
expect that a similar structural limit remains valid for a large class
of dynamics for the underlying risky asset. 

\section{\label{sec:-4}Appendix}
\begin{lem}
\label{lem:-16}Suppose that $\mu$ and $\nu$ are positive Borel
measures on $[0,\infty$) such that $\int_{0}^{\infty}(1+t)^{-m}\mu(dt)<\infty$
and $\int_{0}^{\infty}(1+t)^{-n}\nu(dt)<\infty$ for some $m\ge0$
and $n\ge0$. Then 
\[
\int_{0}^{\infty}(1+t)^{-m-n}(\mu\ast\nu)(dt)<\Big(\int_{0}^{\infty}(1+t)^{-m}\mu(dt)\Big)\Big(\int_{0}^{\infty}(1+t)^{-n}\nu(dt)\Big).
\]
\end{lem}

\begin{proof}
By the definition of $\mu\ast\nu$,
\[
\begin{aligned}\int_{0}^{\infty}(1+t)^{-m-n}(\mu\ast\nu)(dt) & =\int_{0}^{\infty}\int_{0}^{\infty}(1+t+s)^{-m-n}\mu(dt)\nu(ds)\\
 & \le\int_{0}^{\infty}\int_{0}^{\infty}(1+t)^{-m}(1+s)^{-n}\mu(dt)\nu(ds)\\
 & =\Big(\int_{0}^{\infty}(1+t)^{-m}\mu(dt)\Big)\Big(\int_{0}^{\infty}(1+t)^{-n}\nu(dt)\Big).
\end{aligned}
\]
This completes the proof.
\end{proof}
\begin{lem}
\label{lem:-20}Suppose that $\mu$ and $\nu$ are positive Borel
measures on $[0,\infty$) and that $\int_{0}^{\infty}(1+t)^{-m}(\mu\ast\nu)(dt)<\infty$
for some $m\ge0$. Then $\int_{0}^{\infty}(1+t)^{-m}\mu(dt)<\infty$
and $\int_{0}^{\infty}(1+t)^{-m}\nu(dt)<\infty$.
\end{lem}

\begin{proof}
By the definition of $\mu\ast\nu$,
\[
\begin{aligned}\int_{0}^{\infty}(1+t)^{-m}(\mu\ast\nu)(dt) & =\int_{0}^{\infty}\int_{0}^{\infty}(1+t+s)^{-m}\mu(dt)\nu(ds)\\
 & \ge\int_{0}^{\infty}\int_{0}^{\infty}(1+t)^{-m}(1+s)^{-m}\mu(dt)\nu(ds)\\
 & =\Big(\int_{0}^{\infty}(1+t)^{-m}\mu(dt)\Big)\Big(\int_{0}^{\infty}(1+t)^{-m}\nu(dt)\Big).
\end{aligned}
\]
This completes the proof.
\end{proof}
\begin{lem}
\label{lem:-11}Let $\{\xi_{n}\}_{n\ge1}$ be a sequence of random
variables and $\xi$ be a random variable such that
\begin{equation}
\sup_{u\in\mathbb{R}}|\mathbb{P}(\xi_{n}\ge u)-\mathbb{P}(\xi\ge u)|\le c_{\xi}n^{-\alpha},\label{eq:-154}
\end{equation}
and
\begin{equation}
\sup_{u\in\mathbb{R}}\mathbb{P}(u\le\xi<u+\delta)\le c_{\xi}\delta,\quad\delta>0,\label{eq:-157}
\end{equation}
for some constants $c_{\xi}>0$ and $\alpha>0$. Suppose that $\{\eta_{n}\}_{n\ge1}$
is a sequence of random variables and $\eta\in\mathbb{R}$ such that
\begin{equation}
\sup_{u>0}u^{p}\mathbb{P}(|\eta_{n}-\eta|>u)\le c_{\eta}n^{-\beta(p+1)},\label{eq:-127}
\end{equation}
for some constants $c_{\eta}>0$, $\beta>0$, and $p>0$. Then
\[
\sup_{u\in\mathbb{R}}|\mathbb{P}(\xi_{n}+\eta_{n}\ge u)-\mathbb{P}(\xi+\eta\ge u)|\le c_{\xi,\eta}n^{-\min(\alpha,\beta)},
\]
where $c_{\xi,\eta}$ is a constant depending only on the constants
$c_{\xi}$ and $c_{\eta}$.
\end{lem}

\begin{proof}
For any $\delta>0$, by (\ref{eq:-127}),
\begin{equation}
\begin{aligned} & \mathbb{P}(\xi_{n}+\eta_{n}\ge u)-\mathbb{P}(\xi_{n}+\eta\ge u-\delta)\\
 & \le\mathbb{P}(\xi_{n}+\eta_{n}\ge u,\xi_{n}+\eta<u-\delta)\le\mathbb{P}(\eta_{n}-\eta>\delta)\le c_{\eta}\delta^{-p}n^{-\beta(p+1)}.
\end{aligned}
\label{eq:-138}
\end{equation}
By (\ref{eq:-154}), 
\begin{equation}
\mathbb{P}(\xi_{n}+\eta\ge u-\delta)-\mathbb{P}(\xi+\eta\ge u-\delta)\le c_{\xi}n^{-\alpha}.\label{eq:-156}
\end{equation}
By (\ref{eq:-157}), 
\begin{equation}
\mathbb{P}(\xi+\eta\ge u-\delta)-\mathbb{P}(\xi+\eta\ge u)\le c_{\xi}\delta.\label{eq:-158}
\end{equation}
Combining (\ref{eq:-138}), (\ref{eq:-156}) and (\ref{eq:-158})
gives 
\[
\mathbb{P}(\xi_{n}+\eta_{n}\ge u)-\mathbb{P}(\xi+\eta\ge u)\le c_{\xi}\delta+c_{\eta}\delta^{-p}n^{-\beta(p+1)}+c_{\xi}n^{-\alpha}.
\]
Setting $\delta=n^{-\beta}$ gives
\begin{equation}
\mathbb{P}(\xi_{n}+\eta_{n}\ge u)-\mathbb{P}(\xi+\eta\ge u)\le(c_{\xi}+c_{\eta})n^{-\beta}+c_{\xi}n^{-\alpha}\le c_{\xi,\eta}n^{-\min(\alpha,\beta)}.\label{eq:-159}
\end{equation}

Note that $\{-\xi_{n}\}_{n}$ also satisfies (\ref{eq:-154}) and
(\ref{eq:-157}). Applying (\ref{eq:-159}) to $\{-\xi_{n}\}_{n}$,
$\{-\eta_{n}\}_{n}$ gives 
\[
\mathbb{P}(-\xi_{n}-\eta_{n}\ge-u)-\mathbb{P}(-\xi-\eta\ge-u)\le c_{\xi,\eta}n^{-\min(\alpha,\beta)},
\]
or equivalently,
\[
\mathbb{P}(\xi_{n}+\eta_{n}\le u)-\mathbb{P}(\xi+\eta\le u)\le c_{\xi,\eta}n^{-\min(\alpha,\beta)}.
\]
Since $\mathbb{P}(\xi+\eta\le y)=\mathbb{P}(\xi+\eta<y)$ in view
of (\ref{eq:-157}), it follows from the above that 
\[
\mathbb{P}(\xi_{n}+\eta_{n}<u)-\mathbb{P}(\xi+\eta<u)\le c_{\xi,\eta}n^{-\min(\alpha,\beta)},
\]
or equivalently,
\[
\mathbb{P}(\xi+\eta\ge u)-\mathbb{P}(\xi_{n}+\eta_{n}\ge u)\le c_{\xi,\eta}n^{-\min(\alpha,\beta)},
\]
which, together with (\ref{eq:-159}), completes the proof.
\end{proof}


\begin{thebibliography}{10}
\bibitem{Aronson68}Aronson, D. G. (1968). Non-negative solutions
of linear parabolic equations. \emph{Annali della Scuola Normale Superiore
di Pisa-Scienze Fisiche e Matematiche}, 22(4), 607-694.

\bibitem{Feller51}Feller, W. (1951). Two singular diffusion problems.
\emph{Annals of Mathematics}, 54(1), 173-182.

\bibitem{GJ14}Gatheral, J., and Jacquier, A. (2014). Arbitrage-free
SVI volatility surfaces. \emph{Quantitative Finance}, 14(1), 59-71. 

\bibitem{Haeusler88}Haeusler, E. (1988). On the rate of convergence
in the central limit theorem for martingales with discrete and continuous
time. \emph{The Annals of Probability}, 275-299.

\bibitem{HHKRSV}Harvey, C. R., Hoyle, E., Korgaonkar, R., Rattray,
S., Sargaison, M., and Van Hemert, O. (2018). The impact of volatility
targeting. \emph{The Journal of Portfolio Management},\emph{ 45}(1),
14--33. 

\bibitem{Kakutani39}Kakutani, S. (1939). Weak convergence in uniformly
convex spaces. \emph{Tohoku Mathematical Journal, First Series}, 45,
188-193.

\bibitem{LG26}Liu, X., and Gauthier, M. (2026). On the exact limiting
distribution of a volatility target index. \emph{Mathematical Finance.}

\bibitem{Nualart06}Nualart, D. (2006). \emph{The Malliavin Calculus
and Related Topics}. Berlin, Heidelberg: Springer Berlin Heidelberg.

\bibitem{SSV12}Schilling, R. L., Song, R., and Vondra\v{c}ek, Z.
(2012). \emph{Bernstein Functions: Theory and Applications} (Vol.
37). Walter de Gruyter.

\bibitem{Stein70}Stein, E. M. (1970). \emph{Singular Integrals and
Differentiability Properties of Functions}. Princeton University Press.

\bibitem{SW71}Stein, E. M., and Weiss, G. L. (1971). \emph{Introduction
to Fourier Analysis on Euclidean Spaces}. Princeton University Press.

\end{thebibliography}
\end{document}